\documentclass{article}[12]

\usepackage{xy,amssymb,makeidx,cite}
\usepackage{amsmath,amscd,graphics,amsthm}
\usepackage{mathrsfs}  
\usepackage[hyperindex]{hyperref}
\xyoption{all}
\usepackage{tikz}
\usetikzlibrary{calc,3d,arrows,snakes,decorations.markings}
\tikzset{->-/.style={decoration={  markings,  mark=at position .75 with {\arrow{latex}}},postaction={decorate}}}
\tikzset{-<-/.style={decoration={  markings,  mark=at position .75 with {\arrow{latex reversed}}},postaction={decorate}}}
\usepackage{pgflibrarysnakes}

\newtheorem{prop}{Proposition}[subsection]

\newtheorem{conjecture}[prop]{Conjecture}
\newtheorem{theorem}[prop]{Theorem}
\newtheorem{definition}[prop]{Definition}
\newtheorem{lemma}[prop]{Lemma}
\newtheorem{remark}[prop]{Remark}
\newtheorem{example}[prop]{Example}
\newtheorem{intermezzo}[prop]{Intermezzo}

\newcommand{\weg}[1]{}
\newcommand{\iemph}[1]{\emph{#1}\index{#1}}
\newcommand{\rep}{\texttt{rep}}

\def\C{{\Bbb C}}
\def\Rl{{\Bbb R}}
\def\R{{\Bbb R}}

\def\Z{{\Bbb Z}}
\def\N{{\Bbb N}}

\def\KK{{\Bbb K}}

\def\implies{\Rightarrow}

\def\to{\rightarrow}
\def\onto{\twoheadrightarrow}

\def\<{\langle}
\def\>{\rangle}

\newsavebox{\mybox}

\font\ef = eufb10

\newcommand{\NC}[1]{\xymatrix@=.5cm{\vtx{1}\ar@2@/^/[r]&\vtx{1}\ar@/^/[l]}}

\newcommand{\ideal}[1]{\ensuremath{\text{\ef{#1}}}}

\renewcommand{\H}{\ensuremath{\mathsf{H}}}

\newcommand{\id}{\mathsf{1\!\!1}}

\newcommand{\rank}{\ensuremath{\mathsf{Rank\:}}}
\newcommand{\proj}{\ensuremath{\mathsf{Proj\:}}}
\newcommand{\Rep}{\ensuremath{\mathsf{Rep}}}
\newcommand{\CRep}{\ensuremath{\mathsf{Rep}}}

\newcommand{\trep}{\ensuremath{\mathsf{trep}}}
\newcommand{\CM}{{\ensuremath{\mathcal{M}}}}
\newcommand{\Db}{\ensuremath{\mathtt{D^b}}}

\newcommand{\GL}{\ensuremath{\mathsf{GL}}}

\newcommand{\Proj}{\ensuremath{\mathsf{Proj\:}}}
\newcommand{\Mat}{\ensuremath{\mathsf{Mat}}}
\newcommand{\RHom}{\ensuremath{\mathsf{Hom}}}
\newcommand{\Hom}{\ensuremath{\mathsf{Hom}}}
\newcommand{\End}{\ensuremath{\mathsf{End}}}
\newcommand{\Ext}{\ensuremath{\mathsf{Ext}}}

\newcommand{\Stab}{\ensuremath{\mathsf{Stab}}}

\newcommand{\Mod}{\ensuremath{\mathtt{Mod}}\,}
\newcommand{\Coh}{\ensuremath{\mathtt{Coh}}\,}
\newcommand{\Fuk}{\ensuremath{\mathtt{Fuk}}\,}
\newcommand{\Tilt}{\ensuremath{\mathtt{T}}}
\newcommand{\GHilb}{G\text{-}\mathsf{Hilb}}

\newcommand{\se}[1]{\begin{equation*}\begin{split}#1\end{split}\end{equation*}}

\newcommand {\vtx}[1]{*+<4pt>[o][F-]{\scriptstyle{#1}}}
\newcommand {\sqvtx}[1]{*+<4pt>[o][F-]{\scriptstyle{#1}}}

\newcommand {\Spec}{\mathsf {Spec}\,}

\newcommand {\PP}{\mathbb P}

\newcommand{\rib}{\mathrm \Gamma}
\newcommand{\qpol}{\mathrm Q}

\newcommand{\ot}{\leftarrow}
\newcommand{\tot}{\leftrightarrow}
\newcommand{\eps}{\epsilon}
\newcommand{\ful}[1]{|{#1}|}
\newcommand{\fat}[1]{|\mathring{#1}|}
\newcommand{\punct}[1]{|\dot{#1}|}
\newcommand{\mirror}[1]{\mathord{\reflectbox{$#1$}}}
\newcommand{\twist}[1]{\mathord{#1}^{\,\bowtie}}
\newcommand{\Maps}{\mathrm{Maps}}
\newcommand{\address}[1]{}

\newcommand{\PMs}{\mathsf{PM}}
\newcommand{\MP}{\mathsf{MP}}
\newcommand{\NP}{\mathsf{NP}}
\newcommand{\ZP}{\mathsf{ZP}}
\newcommand{\LP}{\mathsf{LP}}
\renewcommand{\SS}{\mathbb{S}}

\newcommand{\bir}{\twist{\rib}}
\newcommand{\polq}{\mirror{\qpol}}

\newcommand{\Surf}{\mathbb{S}}

\newcommand{\cL}{\mathcal{L}}
\newcommand{\cM}{\mathcal{M}}
\newcommand{\cG}{\mathcal{G}}
\newcommand{\cT}{\mathcal{T}}
\newcommand{\cX}{\mathcal{X}}
\newcommand{\cE}{\mathcal{E}}

\newcommand{\rL}{\mathcal{L}_\R}
\newcommand{\rM}{\mathcal{M}_\R}
\newcommand{\rG}{{\mathcal{G}_\R}}
\newcommand{\rT}{\mathcal{T}_\R}
\newcommand{\rX}{\mathcal{X}_\R}
\newcommand{\rC}{\mathcal{C}_\R}
\newcommand{\Log}{\mathrm{Log}}
\newcommand{\Am}{\mathcal{A}}

\newcommand{\tL}{\mathcal{L}_\trop}

\newcommand{\tX}{\mathcal{X}_\trop}
\newcommand{\tC}{\mathcal{C}_\trop}

\newcommand{\cC}{\mathcal{C}}
\newcommand{\pZ}{\mathtt{Z}}
\newcommand{\cP}{\mathtt{P}}
\newcommand{\PM}{\mathtt{P}}
\newcommand{\cR}{\mathtt{R}}

\newcommand{\Var}{\mathrm{Var}}
\newcommand{\floor}[1]{\lfloor #1 \rfloor}

\newcommand{\Jac}{\mathtt{J}}
\newcommand{\sG}{\mathscr{G}}
\newcommand{\sF}{\mathscr{F}}
\newcommand{\trop}{\mathtt{trop}}

\newcommand{\cF}{\mathcal{F}}
\newcommand{\FM}{\mathcal{F}}
\newcommand{\cO}{\mathscr{O}}
\newcommand{\cI}{\mathscr{I}}
\newcommand{\cTT}{\mathscr{T}}
\newcommand{\cPP}{\mathscr{P}}
\newcommand{\cU}{\mathscr{U}}
\newcommand{\Pic}{\mathrm{Pic}}

\title{A Dimer ABC}
\author{Raf Bocklandt}
\address{Raf Bocklandt\\
Korteweg de Vries institute\\
University of Amsterdam (UvA)\\
Science Park 904\\ 
1098 XH Amsterdam\\ 
The Netherlands
\email{raf.bocklandt@gmail.com}
}

\allowdisplaybreaks
\makeindex

\begin{document}
\bibliographystyle{plain}
\maketitle
\begin{abstract}
We give an overview of recent developments in the theory of dimer models. The viewpoint we take is inspired by mirror symmetry.
After an introduction to the combinatorics of dimer models,
we will first look at dimers in dynamical systems and statistical mechanics, which can be viewed as coming from the A-model in mirror symmetry. Then we will
discuss the role of dimers in the theory of resolutions of singularities, which is inspired by the B-model. The C stands for the connections that tie 
both subjects together: clusters, categories, and stability conditions.
In this final part we will give some ideas on how these two stories fit in a broader framework.
\end{abstract}

\tableofcontents

\section{Dimer combinatorics}

Dimer models originally come from statistical physics. They were introduced
as a model to study phase transitions in solid state physics \cite{kasteleyn,Fisher,temperley1961dimer}. Just like a polymer is something containing
many parts, the name dimer refers to something that is built up of two parts, black and white particles.
In the original dimer model these black and white particles were distributed on a square lattice such that
two neighboring particles have opposite colours. The states of the system are configurations where every particle is bound
to a neighboring particle. This is known as a perfect matching. Describing the physics of this system amounts to counting
the number of perfect matchings, subject to certain boundary conditions. 

If the boundary conditions are periodic, this is the same as working on a torus. In this view it is more natural to work with bipartite graphs embedded in a surface, where the nodes represent the particles and the edges the possible bonds that can be made.
Many of the physical properties can now be translated into graph theoretical concepts. In this first section we will introduce
dimers in this general setup and look at their combinatorics.

\subsection{Ribbon graphs}
If a graph is embedded in an oriented surface we can use the orientation of the surface to define a cyclic orientation on the edges in each node of the graph.
This can be captured in the notion of a ribbon graph \cite{igusa2002higher}.
\begin{definition}
A \iemph{ribbon graph} or \iemph{fat graph} $\rib=(H,\nu,\eps)$ consists of a finite set of half edges $H$ and two permutations
$\nu,\eps:H\to H$ such that $\eps$ is an involution without fixed points.
\begin{enumerate}
\item
We call the orbits of $\nu$ \iemph{nodes} and denote the set of 
nodes by $\rib_0=H/\<\nu\>$.
\item
We call the orbits of $\eps$ \iemph{edges} and denote the set of 
edges by $\rib_1=H/\<\eps\>$.
\item 
We call the orbits of $\varphi:=\nu\circ\eps$ \iemph{faces} and denote the set of 
faces by $\rib_2=H/\<\varphi\>$.
\end{enumerate}
We say $x \in \rib_i$ and $y \in \rib_j$ are incident if they intersect. 

The \iemph{dual ribbon graph} $\rib^\vee=(H,\nu\circ \eps,\eps)$ has the same half edges  but the roles of $\nu$ and $\varphi$ are reversed.
Because $\eps^2=1$ we have that $\rib^{\vee\vee} =\rib$.
\end{definition} 

Given a ribbon graph $\rib$ we can construct a closed surface $\ful{\rib}$ by representing each face of size $n$ by an $n$-gon.
Each edge of the $n$-gon represents an edge incident with the corresponding face and we glue these polygons together by joint edges. 
For every half edge there is a quadrangle on the surface whose four points are the two nodes and the centers of the 2 faces incident with the edge.
These quadrangles tile the surface and the edges are diagonals of these quadrangles. The dual ribbon graph can be drawn on this surface 
by connecting the centers of the faces using the other diagonals of these quadrangles.

The name ribbon graph comes from the fact that a closed tubular neighbourhood of the embedded graph $(\rib_0,\rib_1)$ in $\ful{\rib}$ looks like a graph made of ribbons. 
This tubular neighborhood is sometimes denoted by $\fat{\rib}$. It is an oriented surface with boundary and its boundary components correspond
to the faces. If we glue punctured discs to these boundaries, we get a surface with punctures $\punct{\rib}$. 

\begin{example}\label{ex1}
If we take $\eps=(15)(26)(37)(48)$ and
 $\nu=(1234)(8765)$ then $\varphi = (18)(25)(36)(47)$. There are $2$ nodes, $4$ edges and $4$ faces, which are $2$-gons. This gives an Euler characteristic of $2-4+4=2$,
so $\rib$ is embedded in a sphere. On the left we drew the graph on the sphere and in the middle we drew the tubular neigborhood $\fat{\rib}$. On the right we drew the dual graph, which tiles the sphere with $2$ squares that have a common boundary. The quadrangles are represented by dotted lines.
\begin{center}
\begin{tikzpicture}[scale=.75]
\begin{scope}
\begin{scope}[rotate=40]
\draw[thick] (-2,0) arc (180:360:2cm and 1cm);
    \draw[thick,dashed] (-2,0) arc (180:0:2cm and 1cm);
    \draw[thick] (0,2) arc (90:270:1cm and 2cm);
    \draw[thick,dashed] (0,2) arc (90:-90:1 cm and 2cm);
    \draw[very thin] (0,0) circle (2cm);
\draw (.9,.9) node {$\bullet$};    
\draw (-.9,-.9) node {$\bullet$};    
\end{scope}
    \draw[dotted] (2,0) arc (0:360:2 cm and 1.25cm);
    \draw[dotted] (.1,0) arc (0:360:.1 cm and 2cm);
\draw (-.5,-1.0) node[shape=circle,fill=gray!0,inner sep=0pt] {$2$};    
\draw (.5,-1.0) node[shape=circle,fill=gray!0,inner sep=0pt] {$1$};    
\draw (-.5,-1.4) node[shape=circle,fill=gray!0,inner sep=0pt] {$3$};    
\draw (.5,-1.4) node[shape=circle,fill=gray!0,inner sep=0pt] {$4$};    
\draw (.5,1.0) node[shape=circle,fill=gray!0,inner sep=0pt] {$8$};    
\draw (-.5,1.0) node[shape=circle,fill=gray!0,inner sep=0pt] {$7$};    
\draw (.5,1.4) node[shape=circle,fill=gray!0,inner sep=0pt] {$5$};    
\draw (-.5,1.4) node[shape=circle,fill=gray!0,inner sep=0pt] {$6$};    
\end{scope}
\begin{scope}[xshift=5cm]
\draw[dashed] (1.5,0) arc (0:360:1.5cm and 2cm);
\draw[thick] (0,1.75) arc (90:450:1.25cm and 1.75cm);
\draw[thick] (0,1.75) arc (90:450:.50cm and 1.75cm);
\draw[dashed] (0,1.5) arc (90:450:.25cm and 1.5cm);

\draw[dashed] (-.55,1.25) arc (105:255:.6cm and 1.25cm);
\draw[dashed] (-.55,1.25) arc (105:255:.25cm and 1.25cm);
\draw[dashed] (.55,-1.25) arc (285:435:.6cm and 1.25cm);
\draw[dashed] (.55,-1.25) arc (285:435:.25cm and 1.25cm);

\draw (0,1.75) node {$\bullet$};    
\draw (0,-1.75) node {$\bullet$};    
\end{scope}
\begin{scope}[xshift=10cm]
\draw[thick] (2,0) arc (0:360:2);
\draw (0,2) node {$\bullet$};    
\draw (-2,0) node {$\bullet$};    
\draw (2,0) node {$\bullet$};    
\draw (0,-2) node {$\bullet$};    
    \draw[dotted] (2,0) arc (0:360:2 cm and 1.25cm);
    \draw[dotted] (.1,0) arc (0:360:.1 cm and 2cm);
\draw (-10 :2) node[shape=circle,fill=gray!0,inner sep=0pt] {$8$};    
\draw (10 :2) node[shape=circle,fill=gray!0,inner sep=0pt] {$1$};
\draw (80 :2) node[shape=circle,fill=gray!0,inner sep=0pt] {$5$};    
\draw (100 :2) node[shape=circle,fill=gray!0,inner sep=0pt] {$2$};
\draw (170 :2) node[shape=circle,fill=gray!0,inner sep=0pt] {$6$};    
\draw (190 :2) node[shape=circle,fill=gray!0,inner sep=0pt] {$3$};
\draw (260 :2) node[shape=circle,fill=gray!0,inner sep=0pt] {$7$};    
\draw (280 :2) node[shape=circle,fill=gray!0,inner sep=0pt] {$4$};
\end{scope}
\end{tikzpicture}
\end{center}
\end{example}

In some situations it is convenient to consider infinite graphs. If this is the case we assume
that all orbits of $\nu,\eps$ and $\nu\circ\eps$ are finite. This ensures that the graph is locally finite. An example of this is the universal cover of a ribbon graph:
we can take the universal cover of the closed surface $\ful{\rib}$ and lift the embedded graph $\rib$ to this cover. The resulting graph is an infinite ribbon graph
and we will denote it by $\tilde \rib$.

\subsection{Dimer models}

Depending on one's point of view a dimer model can be defined in two ways that are dual to each other. 
\begin{definition}
A \iemph{dimer graph} $\rib$ is a bipartite ribbon graph: there is a partition of the nodes into
black and white nodes $\rib_0=\rib_0^\bullet\sqcup \rib_0^\circ$
such that every edge is incident with one black and one white node.
\end{definition}

For the dual definition we will make use of the quiver formalism. A \iemph{quiver} $Q$ is an oriented graph and we will use the maps $h,t:Q_1\to Q_0$ to denote the head and tail of an arrow. Paths will be written from the right to the left. 
\[
a_1 \dots a_k := \xymatrix{\vtx{}&\vtx{}\ar[l]_{a_1}&\vtx{}\ar@{.>}[l]&\vtx{}\ar[l]_{a_k} }
\]
We also use the head and tail notation for paths, e.g. $h(a_1 \dots a_k) =h(a_1)$, and
a path $p$ is called a cyclic if $h(p)=t(p)$. A cycle is a cyclic path path considered
up to cyclic permutations of the arrows.

\begin{definition}
A \iemph{dimer quiver} $\qpol$ is an oriented ribbon graph such that the faces are oriented cycles. 
We can split  $\qpol_2=\qpol_2^+ \cup \qpol_2^-$ in two types the anticlockwise and the clockwise cycles, depending on their orientation
on the surface.
\end{definition}

Given a dimer graph $\rib$, we can make a dimer quiver by taking the dual $\qpol = \rib^\vee$. The edges of the dual graph run perpendicular to the edges 
of $\rib$ and we orient them such that they keep the black node on the left. In this way the arrows cycle anticlockwise around the black vertices and 
clockwise around the white: $\qpol_2^+=(\rib_0^\bullet)^\vee$ and $\qpol_2^-=(\rib_0^\circ)^\vee$. 

\begin{example}\label{spp}
The suspended pinchpoint \cite{franco2006brane}[section 4.1] is an example of a dimer model on the torus. 
The quiver has 3 vertices corresponding to the three faces on the ribbon graph (one hexagon (1) and two quadrangles (2,3)).

\begin{center}
\begin{tikzpicture}
\begin{scope}[scale=1.3]
\draw [-latex,dotted,shorten >=5pt] (0,0) -- (3,0);
\draw [-latex,dotted,shorten >=5pt] (3,0) -- (3,1);
\draw [-latex,dotted,shorten >=5pt] (3,1) -- (0,0);
\draw [-latex,dotted,shorten >=5pt] (0,0) -- (0,1);
\draw [-latex,dotted,shorten >=5pt] (0,1) -- (3,2);
\draw [-latex,dotted,shorten >=5pt] (3,2) -- (3,3);
\draw [-latex,dotted,shorten >=5pt] (3,2) -- (3,1);
\draw [-latex,dotted,shorten >=5pt] (0,2) -- (0,1);
\draw [-latex,dotted,shorten >=5pt] (0,3) -- (3,3);
\draw [-latex,dotted,shorten >=5pt] (3,3) -- (0,2);
\draw [-latex,dotted,shorten >=5pt] (0,2) -- (0,3);
\draw (0,0) node[circle,draw,fill=white,minimum size=10pt,inner sep=1pt] {{\tiny1}};
\draw (0,1) node[circle,draw,fill=white,minimum size=10pt,inner sep=1pt] {{\tiny2}};
\draw (0,2) node[circle,draw,fill=white,minimum size=10pt,inner sep=1pt] {{\tiny3}};
\draw (0,3) node[circle,draw,fill=white,minimum size=10pt,inner sep=1pt] {{\tiny1}};
\draw (3,0) node[circle,draw,fill=white,minimum size=10pt,inner sep=1pt] {{\tiny1}};
\draw (3,1) node[circle,draw,fill=white,minimum size=10pt,inner sep=1pt] {{\tiny2}};
\draw (3,2) node[circle,draw,fill=white,minimum size=10pt,inner sep=1pt] {{\tiny3}};
\draw (3,3) node[circle,draw,fill=white,minimum size=10pt,inner sep=1pt] {{\tiny1}};
\draw[thick] (2,.5)--(1.5,0);
\draw[thick] (2,.5)--(1.5,1);
\draw[thick] (1.5,2)--(1.5,1);
\draw[thick] (1.5,2)--(1,2.5);
\draw[thick] (1,2.5)--(1.5,3);
\draw[thick] (1,2.5)--(0,2.3);
\draw[thick] (3,2.3)--(1.5,2);
\draw[thick] (0,1.5)--(1.5,2);
\draw[thick] (3,1.5)--(1.5,1);
\draw[thick] (2,.5)--(3,.7);
\draw[thick] (1.5,1)--(0,.7);
\draw (2,.5) node {$\bullet$};    
\draw (1.5,2) node {$\bullet$};    
\draw (1.5,1) node[shape=circle,fill=gray!0,inner sep=0pt] {$\circ$};    
\draw (1,2.5) node[shape=circle,fill=gray!0,inner sep=0pt] {$\circ$};    
\end{scope}
\draw (4,2.5) node[right] {$\vcenter{\xymatrix{&&&&\\&&\vtx{1}\ar@(lu,ru)\ar@/^/[dr]\ar@/^/[dl]&&\\&\vtx{2}\ar@/^/[ur]\ar@/^/[rr]&&\vtx{3}\ar@/^/[ul]\ar@/^/[ll]&}}$};
\end{tikzpicture}
\end{center}
\end{example}
\newcommand{\kk}{k}

We will use both formalisms, so a \iemph{dimer model} or \iemph{dimer} will be a pair $(\rib,\qpol)$ consisting of a dimer graph and a dimer quiver which are dual to each other.
To avoid confusion we will consistently use the terminology \iemph{nodes}, \iemph{edges} and \iemph{faces} for the dimer graph and vertices, arrows and cycles for the dimer quiver. 
Paths and cycles in the bipartite graph will be called (closed) \iemph{walks}, while
for the quiver we will keep the names paths and cycles.

\subsection{Homology and cohomology}

The surface in which a dimer model is embedded, will be denoted by $\ful{\rib}=\ful{\qpol}$ and its Euler characteristic by
\[
\chi(\rib)=\chi(\qpol) = \# \rib_0 -\# \rib_1 + \# \rib_2  = \# \qpol_2 - \# \qpol_1 + \# \qpol_0
\]
To calculate the homology and cohomology of the surface we can use the chain and cochain complexes of the surface.

For the dimer quiver we define the chain complex
\[
\Z \qpol_\bullet :=  \xymatrix{\Z\qpol_0&\Z\qpol_1\ar[l]_d& \Z\qpol_2\ar[l]_d}
\]
consisting of formal sums of vertices, arrows and cycles with coefficients in $\Z$. The differential is given by $dc=\sum_{a \in c} a$ and $da=h(a)-t(a)$ 
for $a \in \qpol_1$ and $c \in \qpol_2$.

For any abelian group $\kk$ we get a cochain complex
\[
\Maps^\bullet(\qpol,\kk) :=  \xymatrix{\Maps(\qpol_0,\kk)\ar[r]^{d}&\Maps(\qpol_1,\kk) \ar[r]^{d} &\Maps(\qpol_2,\kk)} 
\]
with $d\phi(a) = \phi(h(a))-\phi(t(a))$ and $d\phi(c) = \sum_{a \in c} \phi(a)$ for $a \in \qpol_1$ and $c \in \qpol_2$.
 
Similarly we can define a chain complex for the graph:
\[
\Z \rib_\bullet :=  \xymatrix{\Z\rib_0&\Z\rib_1\ar[l]_d& \Z\rib_2\ar[l]_d}
\]
with $de = b(e)-w(e)$ and $df=\sum_{e \in f} \pm e$, where we put 
a plus sign if the black node of the edge follows the white in the anticlockwise direction around $f$. In this convention a face is considered as an anticlockwise cycle and
edges are oriented as going from white to black.

The cochain complex looks like
\[
\Maps^\bullet(\rib,\kk) :=  \xymatrix{\Maps(\rib_0,\kk)\ar[r]^{d}&\Maps(\rib_1,\kk) \ar[r]^{d} &\Maps(\rib_2,\kk) }
\]
with $d\phi(e) = \phi(b(e))-\phi(w(e))$ and $d\phi(f) = \sum_{e \in f} \pm\phi(e)$ for $e \in \rib_1$ and $f \in \rib_2$.

\begin{definition}
If $(\qpol,\rib)$ is a dimer on a surface $\Surf=\ful{\qpol}$ and $\pi:\Surf'\to \Surf$ is a finite unbranched cover we can lift
$(\qpol,\rib)$ to a new dimer $(\qpol',\rib')$ on $\Surf'$. If $G$ is the group of cover automorphisms
\[
 G = \{ \phi:\Surf'\to \Surf'\,|\, \pi\circ\phi=\pi\}
\]
We have a free action of $G$ on $(\rib',\qpol')$ with $\rib=\rib'/G$ and $\qpol=\qpol'/G$. 
We call $(\rib',\qpol')$ a \iemph{Galois cover} of $\qpol$ with cover group $G$. 
\end{definition}
For Galois covers we have that the $G$-invariant part of $\Z\qpol'_\bullet$ is isomorphic to $\Z \qpol_{\bullet}$ and $\chi(\qpol') = \chi(\qpol)|G|$.

\subsection{The twisted dimer}
The ribbon graph picture of dimers allows us to define an involution on the space of dimers. The \iemph{twist of a bipartite ribbon graph} was introduced by Feng et al. in \cite{feng2008dimer} and it is the same graph but with the cyclic orders of
the white faces reversed. 
\[
\twist{\rib} = (H,\nu',\eps) \text{ with } \nu'(x) = \begin{cases}
\nu(x) &\text{the $\<\nu\>$-orbit of $x$ is black}\\
\nu^{-1}(x) &\text{the $\<\nu\>$-orbit of $x$ is white}
\end{cases} 
\]
It is called the twist because the surface with boundary of
this new ribbon graph can be obtained by cutting the ribbons in two, giving them a half twist and gluing them back together again.
The new dimer is also called the \iemph{untwisted dimer}\cite{feng2008dimer}, the \iemph{mirror dimer}\cite{bocklandt2015noncommutative} or the \iemph{specular dual}\cite{hanany2012brane}. The bipartite graph of this dimer is 
the same but it is embedded in a different surface.

\begin{example}
The specular dual of the dimer from example \ref{ex1} is the same graph but embedded in a torus.
\begin{center}
\begin{tikzpicture}
\begin{scope}
\clip(-2,.25) rectangle (2,2.25);
\draw[dashed] (1.5,0) arc (0:360:1.5cm and 2cm);
\draw[dashed] (0,1.5) arc (90:450:.25cm and 1.5cm);
\draw[dashed] (-.55,1.25) arc (105:255:.6cm and 1.25cm);
\draw[dashed] (-.55,1.25) arc (105:255:.25cm and 1.25cm);
\draw[dashed] (.55,-1.25) arc (285:435:.6cm and 1.25cm);
\draw[dashed] (.55,-1.25) arc (285:435:.25cm and 1.25cm);
\end{scope}
\begin{scope}
\clip(-2,-2.25) rectangle (2,-.25);
\draw[dashed] (1.5,0) arc (0:360:1.5cm and 2cm);
\draw[dashed] (0,1.5) arc (90:450:.25cm and 1.5cm);
\draw[dashed] (-.55,1.25) arc (105:255:.6cm and 1.25cm);
\draw[dashed] (-.55,1.25) arc (105:255:.25cm and 1.25cm);
\draw[dashed] (.55,-1.25) arc (285:435:.6cm and 1.25cm);
\draw[dashed] (.55,-1.25) arc (285:435:.25cm and 1.25cm);
\end{scope}

\draw[thick] (0,1.75) arc (90:450:1.25cm and 1.75cm);
\draw[thick] (0,1.75) arc (90:450:.50cm and 1.75cm);
\draw (0,1.75) node[shape=circle,fill=gray!0,inner sep=0pt] {$\circ$};    
\draw (0,-1.75) node {$\bullet$};    

\draw[dashed] (-1.5,.25) .. controls (-1.5,0) and (-1,0) .. (-1,-.25); 
\draw[dashed,xshift=.75cm] (-1.5,.25) .. controls (-1.5,0) and (-1,0) .. (-1,-.25); 
\draw[dashed,xshift=1.75cm] (-1.5,.25) .. controls (-1.5,0) and (-1,0) .. (-1,-.25); 
\draw[dashed,xshift=2.5cm] (-1.5,.25) .. controls (-1.5,0) and (-1,0) .. (-1,-.25); 

\draw[dashed] (-1.5,-.25) .. controls (-1.5,0) and (-1,0) .. (-1,.25); 
\draw[dashed,xshift=.75cm] (-1.5,-.25) .. controls (-1.5,0) and (-1,0) .. (-1,.25); 
\draw[dashed,xshift=1.75cm] (-1.5,-.25) .. controls (-1.5,0) and (-1,0) .. (-1,.25); 
\draw[dashed,xshift=2.5cm] (-1.5,-.25) .. controls (-1.5,0) and (-1,0) .. (-1,.25); 

\draw (2.25,0) node {$\sim$};    
\begin{scope}[xshift=5cm]
\draw[dotted] (-2,2)--(-2,-2)--(2,-2)--(2,2)--(-2,2);
\draw[thick] (-2,2)--(2,-2);
\draw[thick] (2,2)--(-2,-2);
\draw (2,2) node {$\bullet$};    
\draw (2,-2) node {$\bullet$};    
\draw (-2,2) node {$\bullet$};    
\draw (-2,-2) node {$\bullet$};    
\draw (0,0) node[shape=circle,fill=gray!0,inner sep=0pt] {$\circ$};    
\draw[dashed] (-2,1.75) .. controls (-.5,0) and (-.5,0) .. (-2,-1.75); 
\draw[dashed] (2,1.75) .. controls (.5,0) and (.5,0) .. (2,-1.75); 
\draw[dashed] (1.75,2) .. controls (0,.5) and (0,.5) .. (-1.75,2); 
\draw[dashed] (1.75,-2) .. controls (0,-.5) and (0,-.5) .. (-1.75,-2); 
\end{scope}
\end{tikzpicture}
\end{center}
\end{example}

If we look at what happens to the dimer quiver, we can understand this in the following way. First we cut open the surface along the arrows of the quiver, then
we flip over all the clockwise cycles to introduce the twist in the ribbons of the ribbon graph.
Finally we glue everything back together again by identifying the arrows of the cycles. The result is called the \iemph{mirror quiver} or \iemph{twisted quiver}
and we denote it by $\polq$.

\begin{example}\label{lotsofmirrors}
We illustrate this with some examples from \cite{bocklandt2015noncommutative}:
\begin{center}
\begin{tabular}{ccccc}
$\qpol$&$\vcenter{
\xymatrix@C=.75cm@R=.75cm{
\vtx{1}\ar[r]_{a}&\vtx{2}\ar[d]|z&\vtx{1}\ar[l]^b\\
\vtx{3}\ar[u]_c\ar[d]^d&\vtx{4}\ar[l]|w\ar[r]|y&\vtx{3}\ar[u]^c\ar[d]_d\\
\vtx{1}\ar[r]^{a}&\vtx{2}\ar[u]|x&\vtx{1}\ar[l]_b
}}$
&
$\vcenter{
\xymatrix@C=.4cm@R=.75cm{
\vtx{1}\ar[rrr]_a\ar[dr]|{u_1}&&&\vtx{1}\ar[ld]|z\\
&\vtx{3}\ar[r]|y\ar[ld]|x&\vtx{2}\ar[ull]|{u_2}\ar[dr]|{v_2}&\\
\vtx{1}\ar[rrr]^a\ar[uu]_b&&&\vtx{1}\ar[ull]|{v_1}\ar[uu]^b
}}$
&
$\vcenter{
\xymatrix@C=.75cm@R=.75cm{
\vtx{1}\ar[r]_{a}\ar[d]^{b}&\vtx{1}\ar[r]_b&\vtx{1}\ar[d]_c\\
\vtx{1}\ar[d]^a&&\vtx{1}\ar[d]_d\\	
\vtx{1}\ar[r]^{d}&\vtx{1}\ar[r]^c&\vtx{1}\ar[uull]|x
}}$
&
$\vcenter{
\xymatrix@C=.4cm@R=.75cm{
\vtx{1}\ar[dd]\ar[rrr]&&&\vtx{2}\ar[dll]\ar@{.>}[dl]\\
&\vtx{5}\ar[ul]\ar[drr]&\vtx{6}\ar@{.>}[ull]\ar@{.>}[dr]&\\
\vtx{4}\ar[ur]\ar@{.>}[urr]&&&\vtx{3}\ar[uu]\ar[lll]
}}$
\vspace{.5cm}
\\
$\polq$&
$\vcenter{
\xymatrix@C=.75cm@R=.75cm{
\sqvtx{1}\ar[r]_{a}&\sqvtx{2}\ar[d]|c&\sqvtx{1}\ar[l]^y\\
\sqvtx{3}\ar[u]_z\ar[d]^d&\sqvtx{4}\ar[l]|w\ar[r]|b&\sqvtx{3}\ar[u]^z\ar[d]_d\\
\sqvtx{1}\ar[r]^{a}&\sqvtx{2}\ar[u]|x&\sqvtx{1}\ar[l]_y
}}$
&
$\vcenter{
\xymatrix@C=.4cm@R=.75cm{
\sqvtx{1}\ar[rrr]_z\ar[dr]|{u_1}&&&\sqvtx{1}\ar[ld]|a\\
&\sqvtx{3}\ar[r]|y\ar[ld]|b&\sqvtx{2}\ar[ull]|{u_2}\ar[dr]|{v_1}&\\
\sqvtx{1}\ar[rrr]^z\ar[uu]_x&&&\sqvtx{1}\ar[ull]|{v_2}\ar[uu]^x
}}$
&
$\vcenter{
\xymatrix@C=.75cm@R=.75cm{
\sqvtx{1}\ar[r]_{a}&\sqvtx{2}\ar[r]_b&\sqvtx{1}\ar[ld]|c\\
&\sqvtx{3}\ar[ld]|d&\\	
\sqvtx{1}\ar[uu]_x\ar[r]^{a}&\sqvtx{2}\ar[r]^b&\sqvtx{1}\ar[uu]^x
}}$
&
$\vcenter{
\xymatrix@C=.75cm@R=.75cm{
\sqvtx{1}\ar[r]&\sqvtx{2}\ar[r]\ar[ld]&\sqvtx{1}\ar[ld]\\
\sqvtx{3}\ar[r]\ar[u]&\sqvtx{4}\ar[r]\ar[u]\ar[ld]&\sqvtx{3}\ar[u]\ar[ld]\\
\sqvtx{1}\ar[r]\ar[u]&\sqvtx{2}\ar[r]\ar[u]&\sqvtx{1}\ar[u]
}}$
\end{tabular}
\end{center}
On the top row the first two quivers are embedded in a torus, the third in a surface with genus $2$ and the fourth in a sphere.
The mirros on the bottom row are all embedded in a torus.
Note that the first $2$ are isomorphic to their mirrors, but in a nontrivial way.  
\end{example}

Although the graph of the twisted dimer is the same, the mirror quiver $\polq$ will be different. 
If one looks at the collection of edges corresponding to an orbit of $\nu'$ one can see that
in the original dimer they trace out a walk of edges that turns alternatingly to the left and to the right.
Such a walk is called a \iemph{zigzag walk}. In the quiver the arrows corresponding to these
edges also trace out a path that turns alternatingly left and right. We call this path a \iemph{zigzag path} or \iemph{zigzag cycle}. 
In the other direction every zigzag path comes from a cycle
in $\nu'$, so the zigzag paths map one to one to the vertices of the dual dimer.

For each zigzag path (or walk) on the surface $\ful{\qpol}$ we can draw a curve by connecting the midpoints of the edges or arrows. 
The curves obtained in this way are called the \iemph{strands} and we orient them the 
same way as the corresponding zigzag path in the quiver. If you follow a strand and look
at the intersection points with the other strands then the strand will be crossed alternatingly from
the left and from the right. We will also need the \iemph{left opposite path} of a zigzag path. This runs in the opposite direction along the same
positive cycles as the zigzag path. The \iemph{right opposite path} does the same along the negative cycles. 
\vspace{.3cm}
\begin{center}
\begin{tikzpicture}
\draw (0,1)--(2,0)--(4,1)--(6,0)--(8,1);
\draw[-latex,dotted,shorten >=5pt] (6,1) arc (60:120:4);
\draw[-latex,dotted,shorten >=5pt] (2,1) arc (60:90:4);
\draw[latex-,dotted,shorten >=5pt] (6,1) arc (120:90:4);
\draw [-latex,shorten >=5pt] (0,0) -- (2,1);
\draw [-latex,shorten >=5pt] (2,1) -- (4,0);
\draw [-latex,shorten >=5pt] (4,0) -- (6,1);
\draw [-latex,shorten >=5pt] (6,1) -- (8,0);
\draw [-latex,dashed] (0,.5) -- (8,0.5);
\draw [-latex,dashed] (1,0) -- (1,1);
\draw [-latex,dashed] (3,1) -- (3,0);
\draw [-latex,dashed] (5,0) -- (5,1);
\draw [-latex,dashed] (7,1) -- (7,0);
\draw (0,1) node {$\bullet$};    
\draw (4,1) node {$\bullet$};    
\draw (8,1) node {$\bullet$};    
\draw[-latex,dotted,shorten >=5pt] (8,0) arc (-60:-120:4);
\draw[-latex,dotted,shorten >=5pt] (4,0) arc (-60:-120:4);

\draw (2,0) node[shape=circle,fill=gray!0,inner sep=0pt] {$\circ$};    
\draw (6,0) node[shape=circle,fill=gray!0,inner sep=0pt] {$\circ$};    
\draw (0,0) node[circle,draw,fill=white,minimum size=10pt,inner sep=1pt] {{\tiny1}};
\draw (2,1) node[circle,draw,fill=white,minimum size=10pt,inner sep=1pt] {{\tiny2}};
\draw (4,0) node[circle,draw,fill=white,minimum size=10pt,inner sep=1pt] {{\tiny3}};
\draw (6,1) node[circle,draw,fill=white,minimum size=10pt,inner sep=1pt] {{\tiny4}};
\draw (8,0) node[circle,draw,fill=white,minimum size=10pt,inner sep=1pt] {{\tiny5}};
\draw (8.5,1.5) node[right] {left opposite path};    
\draw (8.5,0) node[right] {zigzag path};    
\draw (8.5,.5) node[right] {strand};    
\draw (8.5,1) node[right] {zigzag walk};    
\draw (8.5,-.5) node[right] {right opposite path};    
\end{tikzpicture}
\end{center}
\hspace{.3cm}

The correspondence between the vertices of $\qpol$ and the zigzag paths/strands of $\polq$ also tells us how the strands intersect.
\begin{lemma}\label{intersection}
Let $\qpol$ be a dimer quiver and $\polq$ be its twist. Consider two vertices $v,w \in \qpol_0$ and let $s_v,s_w$ be the corresponding
strands in $\ful{\polq}$. The number of arrows from vertex $v$ to $w$ minus the number of arrows from $w$ to $v$ is the intersection number
between the corresponding strands in the dual dimer.
\[
\eps(s_v,s_w) = \#\{v\ot w\} - \#\{w\to v\}  
\]
\end{lemma}
\begin{proof}
This follows from the fact that if two zigzag paths share an arrow the strands will cross.
\end{proof}

\subsection{Alternating strand diagrams}

The idea of strands gives us another way to define a dimer. 
\begin{definition}
A collection of immersed oriented closed curves on an oriented surface which all intersect normally is called an
\iemph{alternating strand diagram} if every curve is intersected alternatingly from the left and from the right. 
We also assume that the complement of the curves is a disjoint union of discs.
\end{definition}
Given an alternating strand diagram, we can construct the dimer graph as follows. 
The complement of the curves consists of three types of discs, the ones for which the strands
go only clockwise around the boundary, the ones for which they go only anticlockwise and the ones for which they go in both directions.
Put a black node in the middle of every anticlockwise disc and a white node in every clockwise disc. For every strand intersection point 
put an edge connecting the black and white node that sit in a disc having this point in their boundary.
The vertices of the dimer quiver are in the centres of the third type of discs.  

\begin{example}
A dimer on a torus and its alternating strand diagram in both its graph and quiver presentation.
\begin{center}
\begin{tikzpicture}[scale=.75]
\draw[dotted] (-2,2)--(-2,-2)--(2,-2)--(2,2)--(-2,2);
\draw[thick] (-2,2)--(2,-2);
\draw[thick] (2,2)--(-2,-2);
\draw[thick] (0,2)--(2,0);
\draw[thick] (0,-2)--(-2,0);
\draw[thick] (0,2)--(-2,0);
\draw[thick] (0,-2)--(2,0);
\draw (2,2) node[shape=circle,fill=gray!0,inner sep=0pt] {$\circ$};    
\draw (2,-2) node[shape=circle,fill=gray!0,inner sep=0pt] {$\circ$};    
\draw (-2,2) node[shape=circle,fill=gray!0,inner sep=0pt] {$\circ$};    
\draw (-2,-2) node[shape=circle,fill=gray!0,inner sep=0pt] {$\circ$};    
\draw (0,2) node[shape=circle,fill=gray!0,inner sep=0pt] {$\circ$};    
\draw (0,-2) node[shape=circle,fill=gray!0,inner sep=0pt] {$\circ$};    
\draw (-2,0) node[shape=circle,fill=gray!0,inner sep=0pt] {$\circ$};    
\draw (2,0) node[shape=circle,fill=gray!0,inner sep=0pt] {$\circ$};    
\draw (0,0) node[shape=circle,fill=gray!0,inner sep=0pt] {$\circ$};    
\draw (1,1) node[shape=circle,fill=gray!0,inner sep=0pt] {$\bullet$};    
\draw (-1,1) node[shape=circle,fill=gray!0,inner sep=0pt] {$\bullet$};    
\draw (1,-1) node[shape=circle,fill=gray!0,inner sep=0pt] {$\bullet$};    
\draw (-1,-1) node[shape=circle,fill=gray!0,inner sep=0pt] {$\bullet$};    
\draw[latex-,dashed] (.5,-2)--(.5,2);
\draw[-latex,dashed] (-.5,-2)--(-.5,2);
\draw[-latex,dashed] (1.5,-2)--(1.5,2);
\draw[latex-,dashed] (-1.5,-2)--(-1.5,2);
\draw[latex-,dashed] (2,.5)--(-2,.5);
\draw[-latex,dashed] (2,-.5)--(-2,-.5);
\draw[-latex,dashed] (2,1.5)--(-2,1.5);
\draw[latex-,dashed] (2,-1.5)--(-2,-1.5);

\begin{scope}[xshift=5cm]
\draw[dotted] (-2,2)--(-2,-2)--(2,-2)--(2,2)--(-2,2);
\draw [-latex,shorten >=5pt] (-2,-1) -- (-1,-2);
\draw [-latex,shorten >=5pt] (-1,2) -- (-2,1);
\draw [-latex,shorten >=5pt] (-2,1) -- (-1,0);
\draw [-latex,shorten >=5pt] (-1,0) -- (-2,-1);
\draw [-latex,shorten >=5pt] (-1,-2) -- (0,-1);
\draw [-latex,shorten >=5pt] (0,-1) -- (-1,0);
\draw [-latex,shorten >=5pt] (-1,0) -- (0,1);
\draw [-latex,shorten >=5pt] (0,1) -- (-1,2);
\begin{scope}[xshift=2cm]
\draw [-latex,shorten >=5pt] (-2,-1) -- (-1,-2);
\draw [-latex,shorten >=5pt] (-1,2) -- (-2,1);
\draw [-latex,shorten >=5pt] (-2,1) -- (-1,0);
\draw [-latex,shorten >=5pt] (-1,0) -- (-2,-1);
\draw [-latex,shorten >=5pt] (-1,-2) -- (0,-1);
\draw [-latex,shorten >=5pt] (0,-1) -- (-1,0);
\draw [-latex,shorten >=5pt] (-1,0) -- (0,1);
\draw [-latex,shorten >=5pt] (0,1) -- (-1,2);
\end{scope}
\draw (-1,2) node[circle,draw,fill=white,minimum size=10pt,inner sep=1pt] {{\tiny1}};
\draw (1,2) node[circle,draw,fill=white,minimum size=10pt,inner sep=1pt] {{\tiny 2}};
\draw (-2,1) node[circle,draw,fill=white,minimum size=10pt,inner sep=1pt] {{\tiny 3}};
\draw (0,1) node[circle,draw,fill=white,minimum size=10pt,inner sep=1pt] {{\tiny 4}};
\draw (2,1) node[circle,draw,fill=white,minimum size=10pt,inner sep=1pt] {{\tiny 3}};
\draw (-1,0) node[circle,draw,fill=white,minimum size=10pt,inner sep=1pt] {{\tiny 5}};
\draw (1,0) node[circle,draw,fill=white,minimum size=10pt,inner sep=1pt] {{\tiny 6}};
\draw (-2,-1) node[circle,draw,fill=white,minimum size=10pt,inner sep=1pt] {{\tiny 7}};
\draw (0,-1) node[circle,draw,fill=white,minimum size=10pt,inner sep=1pt] {{\tiny 8}};
\draw (2,-1) node[circle,draw,fill=white,minimum size=10pt,inner sep=1pt] {{\tiny 7}};
\draw (-1,-2) node[circle,draw,fill=white,minimum size=10pt,inner sep=1pt] {{\tiny1}};
\draw (1,-2) node[circle,draw,fill=white,minimum size=10pt,inner sep=1pt] {{\tiny 2}};
\draw[latex-,dashed] (.5,-2)--(.5,2);
\draw[-latex,dashed] (-.5,-2)--(-.5,2);
\draw[-latex,dashed] (1.5,-2)--(1.5,2);
\draw[latex-,dashed] (-1.5,-2)--(-1.5,2);
\draw[latex-,dashed] (2,.5)--(-2,.5);
\draw[-latex,dashed] (2,-.5)--(-2,-.5);
\draw[-latex,dashed] (2,1.5)--(-2,1.5);
\draw[latex-,dashed] (2,-1.5)--(-2,-1.5);
\end{scope}
\end{tikzpicture} 
\end{center}
\end{example}

\subsection{Triple crossing diagrams}
A final related concept is a triple crossing diagram.
\begin{definition}
A \iemph{triple crossing diagram} is a collection of oriented curves on an oriented surface 
such that every intersection point of the curves is a triple crossing with alternating orientations.
\begin{center}
\begin{tikzpicture}[scale=.5]
\draw[dotted] (1.2,0) arc (0:360:1.2);
\draw[-latex] (60:1.2) -- (240:1.2);
\draw[-latex] (180:1.2) -- (0:1.2);
\draw[-latex] (300:1.2) -- (120:1.2);
\end{tikzpicture}
\end{center}
Furthermore we ask that the curves cut the surface in simply connected pieces.
\end{definition}
A slight deformation of a triple crossing diagram results in an alternating crossing diagram, which gives us a dimer model. 
\begin{center}
\begin{tikzpicture}
\begin{scope}[scale=.5]
\draw[dotted] (1.2,0) arc (0:360:1.2);
\draw[-latex] (60:1.2) -- (240:1.2);
\draw[-latex] (180:1.2) -- (0:1.2);
\draw[-latex] (300:1.2) -- (120:1.2);
\end{scope}
\draw (1.5,0) node {$\to$};
\begin{scope}[scale=.5,xshift=6cm]
\draw[dotted] (1.2,0) arc (0:360:1.2);
\draw[-latex] (70:1.2) -- (230:1.2);
\draw[-latex] (190:1.2) -- (350:1.2);
\draw[-latex] (310:1.2) -- (110:1.2);
\end{scope}
\draw (4.5,0) node {$\to$};
\begin{scope}[scale=.5,xshift=12cm]
\draw[dotted] (1.2,0) arc (0:360:1.2);
\draw (0,0) -- (90:1.2);
\draw (0,0) -- (210:1.2);
\draw (0,0) -- (330:1.2);
\draw (90:1.2) node[shape=circle,fill=gray!0,inner sep=0pt] {$\circ$};    
\draw (210:1.2) node[shape=circle,fill=gray!0,inner sep=0pt] {$\circ$};    
\draw (330:1.2) node[shape=circle,fill=gray!0,inner sep=0pt] {$\circ$};    
\draw (0,0) node[shape=circle,fill=gray!0,inner sep=0pt] {$\bullet$};    
\end{scope}
\end{tikzpicture}
\end{center}
To obtain this dimer directly from the triple crossing diagram we put black nodes in the triple crossing points and white nodes in the centres of the discs that are 
bounded clockwise. We draw an edge between them if the triple crossing point is on the boundary of the 
disk. Note that in this way we only obtain dimers for which the black nodes are trivalent. 
\begin{example}
A triple crossing diagram on a torus and its dimer graph
\begin{center}
\begin{tikzpicture}[scale=.75]
\begin{scope}[xshift=5cm]
\draw[dotted] (-2,2)--(-2,-2)--(2,-2)--(2,2)--(-2,2);
\draw[thick] (-2,2)--(-1,0)--(-2,-2);
\draw[thick] (2,2)--(1,0)--(2,-2);
\draw[thick] (-1,0)--(1,0);
\draw[-latex,dashed] (-2,.5) .. controls (-1,.5) and (-1,-.5)  .. (0,-.5) .. controls (1,-.5) and (1,.5)  .. (2,.5); 
\draw[latex-,dashed] (-2,-.5) .. controls (-1,-.5) and (-1,.5)  .. (0,.5) .. controls (1,.5) and (1,-.5)  .. (2,-.5); 
\draw[latex-,dashed] (1,-2) -- (1,2);
\draw[-latex,dashed] (-1,-2) -- (-1,2);
\draw (2,2) node[shape=circle,fill=gray!0,inner sep=0pt] {$\circ$};    
\draw (2,-2) node[shape=circle,fill=gray!0,inner sep=0pt] {$\circ$};    
\draw (-2,2) node[shape=circle,fill=gray!0,inner sep=0pt] {$\circ$};    
\draw (-2,-2) node[shape=circle,fill=gray!0,inner sep=0pt] {$\circ$};    
\draw (0,0) node[shape=circle,fill=gray!0,inner sep=0pt] {$\circ$};    
\draw (-1,0) node[shape=circle,fill=gray!0,inner sep=0pt] {$\bullet$};    
\draw (1,0) node[shape=circle,fill=gray!0,inner sep=0pt] {$\bullet$};    
\end{scope}
\end{tikzpicture}
\end{center}
\end{example}

We can also define strand diagrams and triple crossing diagrams on surfaces with boundary. 
For this we assume that every boundary component has an even number of marked points that are alternatingly called entry and exit points.
There are two types of strands: curves connect entry points with exit points and closed curves on the surface. 

Similarly it is possible to consider dimer models on an orientable surface with boundary. 
In the quiver picture we assume that the boundary itself is made up of arrows from the quiver.
These arrows will be contained in just one cycle of $\qpol_2$. The boundary components of the surface
are not required to be cycles of the quiver.

The dual bipartite graph has nodes that lie on the boundary, corresponding to the cycles
that contain arrows on the boundary, but its edges do not lie on the boundary. 

Triple crossing diagrams on a disk were considered by Thurston in \cite{thurston2004dominoes}. Alternating strand diagrams on a disc (with entry and exit points identified) were considered by Postnikov in \cite{postnikov2006total} and the bipartite graph is referred to as the 
plabic graph. The corresponding quiver appears in work by Baur, King and Marsh \cite{baur2013dimer}. Dimers on arbitrary surfaces with boundary were studied by Franco in \cite{franco2012bipartite}.

\subsection{Homotopy for triple crossing diagrams}

\begin{definition}
Two triple crossing diagrams on the same surface (and with the same marked points) are called \iemph{homotopic}
if there is a bijection between the curves that maps every strand to a homotopic strand.
\end{definition}

An easy invariant of triple crossing diagrams under homotopy are the homology classes of the set of strands in $H_1(S,\partial S)$.
Not every set of homology classes can occur because of the alternating orientations of the curves. The restrictions on
the homology classes are not very strong.

\begin{theorem}[Existence of triple crossing diagrams]\label{construction}
Let $S$ be a surface (with boundary) and $I,U$ two collections of $n$-marked points on the boundary (the entry and exit points) such
that on every boundary component the entry and exit points alternate. 

Let $\{\gamma_i: [0,1] \to S| i \in I\}$ be a collection of $n$ curves that each connect one entry point to an exit point:
\[
\gamma_i(0)=i , \{\gamma_i(1)|i \in I\}=U.
\]
Furthermore let $\{\kappa_i: \SS_1 \to S\}$ be a collection of $k$ closed curves with nontrivial homology.

If the total relative homology $\sum \gamma_i +\sum \kappa_i \in \H_1(S,\partial S)$ is zero then there exists a triple crossing diagram
for which the strands are isotopic to the curves $\{\gamma_i,\kappa_i\}$.
\end{theorem}
\begin{proof}
We first prove this when $S$ is a disk. This was originally done by Thurston in \cite{thurston2004dominoes}[Theorem 1].
Note that in this case because $H_1(S,\partial S)=0$, the $\gamma_i$ are complete determined by their end points and there are no closed curves.

We proceed by induction on $n$. For $n=1,2$ the statement is trivial because these can be connected without crossings.
If $n>2$ look for a curve $\gamma_k$ for which there is no $\gamma_i$ with both end points between $\gamma_k(0)$ and $\gamma_k(1)$. 
Connect the strands as indicated below
\begin{center}
\begin{tikzpicture}
\draw (225:2cm) arc (225:315:2cm and 2cm);
\draw (245:1.65cm) arc (245:295: 1.65cm and 1.65cm);
\foreach \s in {1,...,8}
{
  \node at ({225+10*\s}:2cm) {$\bullet$};
} 
\node[left] at ({225+10}:2.2cm) {$\gamma_k(0)$};
\node[right] at ({225+80}:2.2cm) {$\gamma_k(1)$};
\draw (245:1.65cm) arc (65:145:.35cm);
\draw[-latex] (295:1.65cm) arc (125:45:.35cm);
\draw [latex-,shorten >=-13pt] (245:2cm) -- (250:1.65cm);
\draw [-latex,shorten >=-13pt] (255:2cm) -- (250:1.65cm);
\draw [latex-,shorten >=-13pt] (265:2cm) -- (270:1.65cm);
\draw [-latex,shorten >=-13pt] (275:2cm) -- (270:1.65cm);
\draw [latex-,shorten >=-13pt] (285:2cm) -- (290:1.65cm);
\draw [-latex,shorten >=-13pt] (295:2cm) -- (290:1.65cm);
\draw[dotted] (245:1.5cm) arc (245:295: 1.5cm and 1.5cm);
\draw[dotted] (245:1.5cm) arc (65:145:.5cm);
\draw[dotted] (295:1.5cm) arc (125:45:.5cm);
\end{tikzpicture} 
\end{center}
If we cut off the dotted part we get a new disk with $n-1$ entry and exit points and we can use induction
to produce a triple crossing diagram for the remainder and by gluing in the dotted piece for the original. 
A triple crossing diagram constructed in this way is called \iemph{standard}.

To extend this to the case of general surfaces, we cut the surface open to a polygon. 
Because the total relative homology of all curves is zero and each side of the polygon represents a relative homology class,
the total number of points where the curves enter a given side is equal to the number of points where it exits a side. Up to isotopy
we can permute these points on any given side to ensure that on the boundary of the polygon entry and exit points alternate.
Then we can apply the theorem for the disk and after gluing the polygon together again we obtain a triple crossing diagram for the surface.
\end{proof}

We are now interested in how homotopic triple crossing diagrams are related. 
\begin{definition}
We call a triple crossing diagram \iemph{minimal}
if it has the least number of triple crossing points in its homotopy equivalence class.
\begin{itemize}
 \item 
A \iemph{monogon} is a curve segment that forms a contractible loop.
 \item 
A \iemph{bad digon} is a pair of curve segments that are homotopic and have the same orientation.
 \item 
A \iemph{good digon} is a pair of curve segments that are homotopic and have opposite orientations.
 \item 
A \iemph{small digon} is a good digon such that no other strands intersect the homotopic curve segments.
\end{itemize}
\begin{center}
\begin{tikzpicture}
\begin{scope}[scale=.75]
\draw [-latex] (-1,.5) .. controls (0,.5) and (1,1.5) .. (0,1.5)  .. controls (-1,1.5) and (0,.5) .. (1,.5);
\begin{scope}[xshift=3.5cm]
\draw [-latex] (-1.5,.5) .. controls (-1,.5) and (-1,1.5) .. (0,1.5)  .. controls (1,1.5) and (1,.5) .. (1.5,.5);
\draw [-latex] (-1.5,1.5) .. controls (-1,1.5) and (-1,.5) .. (0,.5)  .. controls (1,.5) and (1,1.5) .. (1.5,1.5);
\end{scope}
\begin{scope}[xshift=7.5cm]
\draw [-latex] (-1.5,.5) .. controls (-1,.5) and (-1,1.5) .. (0,1.5)  .. controls (1,1.5) and (1,.5) .. (1.5,.5);
\draw [latex-] (-1.5,1.5) .. controls (-1,1.5) and (-1,.5) .. (0,.5)  .. controls (1,.5) and (1,1.5) .. (1.5,1.5);
\end{scope}
\begin{scope}[xshift=11cm]
\begin{scope}[yshift=1cm]
\begin{scope}[rotate=270]
\draw[-latex] (-.5,-1) -- (0,-.5) .. controls (.5,0) and (.5,0) .. (0,.5) -- (-.5,1);
\draw[-latex ,rotate=180] (-.5,-1) -- (0,-.5) .. controls (.5,0) and (.5,0) .. (0,.5) -- (-.5,1);
\draw (-.5,.5) -- (.5,.5);
\draw (-.5,-.5) -- (.5,-.5);
\end{scope}
\end{scope}
\end{scope}
\end{scope}
\end{tikzpicture}
\end{center}
A triple crossing diagram is called \iemph{consistent} if it has no monogons or bad digons.
\end{definition}

\begin{definition}
The \iemph{$2\tot 2$-move} and the \iemph{$1 \to 0$-move} are local operations on triple crossing diagrams as indicated below.
The $2\tot 2$-move flips a small digon and the $1 \to 0$-move removes a monogon.
\begin{center}
\begin{tikzpicture} 
\begin{scope}[scale=.5]
\draw[dotted] (1.118,0) arc (0:360:1.118);
\draw (-.5,-1) -- (0,-.5) .. controls (.5,0) and (.5,0) .. (0,.5) -- (-.5,1);
\draw [rotate=180] (-.5,-1) -- (0,-.5) .. controls (.5,0) and (.5,0) .. (0,.5) -- (-.5,1);
\draw (-1,.5) -- (1,.5);
\draw (-1,-.5) -- (1,-.5);
\draw (2.5,0) node {$2\tot2$};
\begin{scope}[xshift=5cm]
\begin{scope}[rotate=270]
\draw[dotted] (1.118,0) arc (0:360:1.118);
\draw (-.5,-1) -- (0,-.5) .. controls (.5,0) and (.5,0) .. (0,.5) -- (-.5,1);
\draw [rotate=180] (-.5,-1) -- (0,-.5) .. controls (.5,0) and (.5,0) .. (0,.5) -- (-.5,1);
\draw (-1,.5) -- (1,.5);
\draw (-1,-.5) -- (1,-.5);
\end{scope}
\end{scope}
\end{scope}
\begin{scope}[xshift=6cm]
\begin{scope}[scale=.5]
\draw[dotted] (1.118,0) arc (0:360:1.118);
\draw (-.5,-1) -- (0,-.5) .. controls (1.5,1) and (-1.5,1) .. (0,-.5) -- (.5,-1);
\draw (-1,-.5) -- (1,-.5);
\draw (2.5,0) node {$1\to 0$};
\begin{scope}[xshift=5cm]
\draw[dotted] (1.118,0) arc (0:360:1.118);
\draw (-.5,-1) .. controls (0,-.75)  .. (.5,-1);
\draw (-1,-.5) -- (1,-.5);
\end{scope}
\end{scope}
\end{scope}
\end{tikzpicture}
\end{center}
\end{definition}

\begin{theorem}[Moving between triple crossing diagrams](Thurston \cite{thurston2004dominoes}[Theorem 2-4])\label{movetcd}
Let $S$ be the disk.
\begin{enumerate}
 \item Any triple crossing diagram on $S$ can be turned into a minimal triple crossing diagram by $2\tot 2$- and $1\to 0$-moves.
 \item Two homotopic minimal triple crossing diagrams on $S$ are related by $2\tot 2$-moves.
 \item A triple crossing diagram on $S$ is minimal if and only if it has no monogons or bad digons. In other words minimal is the same as consistent.
\end{enumerate}
\end{theorem}
\begin{proof}
The details of this proof can be found in \cite{thurston2004dominoes}.
The main idea is to start from any triple crossing diagram and show that we can transform it into a given standard triple crossing diagram
using $2\tot 2$- and $1\to 0$-moves. Just like in the proof of theorem \ref{construction} proceed curve by curve. Let $\gamma_k$ be a curve
that does not encompass another strand then Thurston shows that there is an algorithm that uses $2\tot 2$- and $1\to 0$-moves to bring it parallel to the boundary 
like the picture in \ref{construction}. Cutting of the curve $\gamma_k$ and performing the algorithm for another curve in the remaining piece
of the disk, we end up with the standard triple crossing diagram.

Because these moves never increase the number of triple crossing points, the standard one must be minimal and
if the starting one is also minimal we can only use $2\tot 2$-moves.
Finally, if a triple crossing diagram has monogons or bad digons, it still has monogons or bad digons after applying a $2\tot 2$-move, while
a standard one has no monogons or bad digons. 
\end{proof}
The theorem is expected to hold for any oriented surface with boundary, but no proof is currently available.
In the case where $S$ is a torus without boundary, checks have been performed for low crossing numbers and a partial proof
is given by Goncharov and Kenyon in \cite{goncharov2013dimers}. In this paper they construct certain standard triple crossing diagrams and show that these are all related
by $2\tot2$ moves. However they do not show that every minimal triple crossing diagram is standard or can be transformed to a standard one. Therefore we have the following
conjecture.

\begin{conjecture}\label{conjtcd}
Theorem \ref{movetcd} holds for any oriented surface with boundary.
\end{conjecture}

\subsection{Homotopy for dimer models}

We will now translate the previous section to dimer models.

\begin{definition}
\begin{enumerate}
 \item 
Two dimers are called \iemph{homotopic} if there is a homeomorphism between their surfaces such
that the two strand diagrams coming from the zigzag paths are \iemph{homotopic}.
\item
A dimer is called \iemph{reduced} if all positive and negative cycles in the quiver have length at least $3$ (or all nodes in the bipartite graph have valency at least $3$).
\item
A dimer is called \iemph{zigzag consistent} if its alternating strand diagram has no monogons or bad digons.
\end{enumerate}
\end{definition}

\begin{example}
Two homotopic dimers:
\begin{center}
\begin{tikzpicture}[scale=.75]
\draw[dotted] (-2,2)--(-2,-2)--(2,-2)--(2,2)--(-2,2);
\draw [-latex,shorten >=5pt] (-2,-2) -- (0,-2);
\draw [-latex,shorten >=5pt] (2,-2) -- (0,-2);
\draw [-latex,shorten >=5pt] (-2,2) -- (0,2);
\draw [-latex,shorten >=5pt] (2,2) -- (0,2);
\draw [-latex,shorten >=5pt] (-2,0) -- (-2,-2);
\draw [-latex,shorten >=5pt] (-2,0) -- (-2,2);
\draw [-latex,shorten >=5pt] (2,0) -- (2,-2);
\draw [-latex,shorten >=5pt] (2,0) -- (2,2);
\draw [-latex,shorten >=5pt] (0,2) -- (0,0);
\draw [-latex,shorten >=5pt] (0,-2) -- (0,0);
\draw [-latex,shorten >=5pt] (0,0) -- (2,0);
\draw [-latex,shorten >=5pt] (0,0) -- (-2,0);
\draw [-latex, dashed] (-1,-2) -- (2,1);
\draw [latex-, dashed] (-1,-2) -- (-2,-1);
\draw [-latex, dashed] (1,2) -- (-2,-1);
\draw [latex-, dashed] (1,2) -- (2,1);
\draw [-latex, dashed] (-2,1) -- (-1,2);
\draw [latex-, dashed] (-2,1) -- (1,-2);
\draw [latex-, dashed] (2,-1) -- (-1,2);
\draw [-latex, dashed] (2,-1) -- (1,-2);
\draw (-2,-2) node[circle,draw,fill=white,minimum size=10pt,inner sep=1pt] {{\tiny1}};
\draw (2,-2) node[circle,draw,fill=white,minimum size=10pt,inner sep=1pt] {{\tiny1}};
\draw (-2,2) node[circle,draw,fill=white,minimum size=10pt,inner sep=1pt] {{\tiny1}};
\draw (2,2) node[circle,draw,fill=white,minimum size=10pt,inner sep=1pt] {{\tiny1}};
\draw (0,2) node[circle,draw,fill=white,minimum size=10pt,inner sep=1pt] {{\tiny 2}};
\draw (0,-2) node[circle,draw,fill=white,minimum size=10pt,inner sep=1pt] {{\tiny 2}};
\draw (0,0) node[circle,draw,fill=white,minimum size=10pt,inner sep=1pt] {{\tiny 3}};
\draw (2,0) node[circle,draw,fill=white,minimum size=10pt,inner sep=1pt] {{\tiny 4}};
\draw (-2,0) node[circle,draw,fill=white,minimum size=10pt,inner sep=1pt] {{\tiny 4}};

\begin{scope}[xshift=5cm]
\draw[dotted] (-2,2)--(-2,-2)--(2,-2)--(2,2)--(-2,2);
\draw [-latex,shorten >=5pt] (-2,-2) -- (0,-2);
\draw [-latex,shorten >=5pt] (2,-2) -- (0,-2);
\draw [-latex,shorten >=5pt] (-2,2) -- (0,2);
\draw [-latex,shorten >=5pt] (2,2) -- (0,2);
\draw [-latex,shorten >=5pt] (-2,0) -- (-2,-2);
\draw [-latex,shorten >=5pt] (-2,0) -- (-2,2);
\draw [-latex,shorten >=5pt] (2,0) -- (2,-2);
\draw [-latex,shorten >=5pt] (2,0) -- (2,2);
\draw [-latex,shorten >=5pt] (0,0) -- (0,2);
\draw [-latex,shorten >=5pt] (0,0) -- (0,-2);
\draw [-latex,shorten >=5pt] (2,0) -- (0,0);
\draw [-latex,shorten >=5pt] (-2,0) -- (0,0);

\draw [-latex,shorten >=5pt] (0,-2) -- (-2,0);
\draw [-latex,shorten >=5pt] (0,2) -- (2,0);
\draw [-latex,shorten >=5pt] (0,-2) -- (2,0);
\draw [-latex,shorten >=5pt] (0,2) -- (-2,0);

\draw [-latex, dashed] (-1,-2) ..controls (-.5,-1.5)  and (-1.5,-.5).. (-1,0) .. controls (-.5,.5) and (-.5,.5)..  (0,1) .. controls (.5,1.5) and (1.5,.5) .. (2,1);
\draw [latex-, dashed,rotate=90] (-1,-2) ..controls (-.5,-1.5)  and (-1.5,-.5).. (-1,0) .. controls (-.5,.5) and (-.5,.5)..  (0,1) .. controls (.5,1.5) and (1.5,.5) .. (2,1);
\draw [-latex, dashed,rotate=180] (-1,-2) ..controls (-.5,-1.5)  and (-1.5,-.5).. (-1,0) .. controls (-.5,.5) and (-.5,.5)..  (0,1) .. controls (.5,1.5) and (1.5,.5) .. (2,1);
\draw [latex-, dashed,rotate=270] (-1,-2) ..controls (-.5,-1.5)  and (-1.5,-.5).. (-1,0) .. controls (-.5,.5) and (-.5,.5)..  (0,1) .. controls (.5,1.5) and (1.5,.5) .. (2,1);
\draw [latex-, dashed] (-1,-2) -- (-2,-1);
\draw [latex-, dashed] (1,2) -- (2,1);
\draw [latex-, dashed] (-2,1) -- (-1,2);
\draw [-latex, dashed] (2,-1) -- (1,-2);
\draw (-2,-2) node[circle,draw,fill=white,minimum size=10pt,inner sep=1pt] {{\tiny1}};
\draw (2,-2) node[circle,draw,fill=white,minimum size=10pt,inner sep=1pt] {{\tiny1}};
\draw (-2,2) node[circle,draw,fill=white,minimum size=10pt,inner sep=1pt] {{\tiny1}};
\draw (2,2) node[circle,draw,fill=white,minimum size=10pt,inner sep=1pt] {{\tiny1}};
\draw (0,2) node[circle,draw,fill=white,minimum size=10pt,inner sep=1pt] {{\tiny 2}};
\draw (0,-2) node[circle,draw,fill=white,minimum size=10pt,inner sep=1pt] {{\tiny 2}};
\draw (0,0) node[circle,draw,fill=white,minimum size=10pt,inner sep=1pt] {{\tiny 3}};
\draw (2,0) node[circle,draw,fill=white,minimum size=10pt,inner sep=1pt] {{\tiny 4}};
\draw (-2,0) node[circle,draw,fill=white,minimum size=10pt,inner sep=1pt] {{\tiny 4}};
\end{scope}
\end{tikzpicture} 
\end{center}
\end{example}

We also define moves between dimers.  
\begin{definition}
\begin{enumerate}
 \item 
The \iemph{join-move} removes a bivalent node and joins the two nodes connected to it. The inverse of a \iemph{join-move} is called a \iemph{split move}.
In the quiver picture the split move inserts a bigon in a cycle.
\begin{center}
\begin{tikzpicture}
\begin{scope}[scale=.5]
\begin{scope}[xshift=5cm]
\draw[dotted] (1.5,0) arc (0:360:1.5);
\draw (-1,-1) node[shape=circle,fill=gray!0,inner sep=0pt] {$\circ$} -- (0,0);
\draw (-1.5,0) node[shape=circle,fill=gray!0,inner sep=0pt] {$\circ$} -- (0,0);
\draw (-1,1)node[shape=circle,fill=gray!0,inner sep=0pt] {$\circ$} -- (0,0);
\draw (1,-1) node[shape=circle,fill=gray!0,inner sep=0pt] {$\circ$} -- (0,0);
\draw (1,1) node[shape=circle,fill=gray!0,inner sep=0pt] {$\circ$}-- (0,0);
\draw (0,0) node[shape=circle,fill=gray!0,inner sep=0pt] {$\bullet$};    
\end{scope}
\draw[dotted] (1.5,0) arc (0:360:1.5);
\draw (-1,-1) node[shape=circle,fill=gray!0,inner sep=0pt] {$\circ$} -- (-.75,0);
\draw (-1.5,0) node[shape=circle,fill=gray!0,inner sep=0pt] {$\circ$} -- (-.75,0);
\draw (-1,1)node[shape=circle,fill=gray!0,inner sep=0pt] {$\circ$} -- (-.75,0);
\draw (1,-1) node[shape=circle,fill=gray!0,inner sep=0pt] {$\circ$} -- (.75,0);
\draw (1,1) node[shape=circle,fill=gray!0,inner sep=0pt] {$\circ$}-- (.75,0);
\draw (-.75,0) node[shape=circle,fill=gray!0,inner sep=0pt] {$\bullet$};    
\draw (.75,0) node[shape=circle,fill=gray!0,inner sep=0pt] {$\bullet$};    
\draw (-.75,0) -- (.75,0);
\draw (0,0) node[shape=circle,fill=gray!0,inner sep=0pt] {$\circ$};    
\draw (2.5,0) node {$\tot$};
\end{scope}
\begin{scope}[xshift=6cm]
\begin{scope}[scale=.5]
\begin{scope}[xshift=5cm]
\draw[dotted] (1.5,0) arc (0:360:1.5);
\draw [-latex,shorten >= 2pt] (216:1.5) node[shape=circle,fill=gray!0,inner sep=0pt] {$\circ$} -- (144:1.5);
\draw [-latex,shorten >= 2pt] (144:1.5) node[shape=circle,fill=gray!0,inner sep=0pt] {$\circ$} -- (72:1.5);
\draw [-latex,shorten >= 2pt] (72:1.5) node[shape=circle,fill=gray!0,inner sep=0pt] {$\circ$} -- (0:1.5);
\draw [-latex,shorten >= 2pt] (0:1.5) node[shape=circle,fill=gray!0,inner sep=0pt] {$\circ$} -- (288:1.5);
\draw [-latex,shorten >= 2pt] (288:1.5) node[shape=circle,fill=gray!0,inner sep=0pt] {$\circ$} -- (216:1.5);
\draw  (0:1.5) node[shape=circle,fill=gray!0,inner sep=0pt] {$\circ$};
\draw  (72:1.5) node[shape=circle,fill=gray!0,inner sep=0pt] {$\circ$};
\draw (144:1.5) node[shape=circle,fill=gray!0,inner sep=0pt] {$\circ$};
\draw  (216:1.5) node[shape=circle,fill=gray!0,inner sep=0pt] {$\circ$};
\draw (288:1.5) node[shape=circle,fill=gray!0,inner sep=0pt] {$\circ$};
\end{scope}
\draw[dotted] (1.5,0) arc (0:360:1.5);
\draw[dotted] (1.5,0) arc (0:360:1.5);
\draw [-latex,shorten >= 2pt] (216:1.5) node[shape=circle,fill=gray!0,inner sep=0pt] {$\circ$} -- (144:1.5);
\draw [-latex,shorten >= 2pt] (144:1.5) node[shape=circle,fill=gray!0,inner sep=0pt] {$\circ$} -- (72:1.5);
\draw [-latex,shorten >= 2pt] (72:1.5) node[shape=circle,fill=gray!0,inner sep=0pt] {$\circ$} -- (0:1.5);
\draw [-latex,shorten >= 2pt] (0:1.5) node[shape=circle,fill=gray!0,inner sep=0pt] {$\circ$} -- (288:1.5);
\draw [-latex,shorten >= 2pt] (288:1.5) node[shape=circle,fill=gray!0,inner sep=0pt] {$\circ$} -- (216:1.5);
\draw [-latex,shorten >= 2pt] (72:1.5) .. controls (.2,0) .. (288:1.5);
\draw [-latex,shorten >= 2pt] (288:1.5) .. controls (.62,0) .. (72:1.5);
\draw  (0:1.5) node[shape=circle,fill=gray!0,inner sep=0pt] {$\circ$};
\draw  (72:1.5) node[shape=circle,fill=gray!0,inner sep=0pt] {$\circ$};
\draw (144:1.5) node[shape=circle,fill=gray!0,inner sep=0pt] {$\circ$};
\draw  (216:1.5) node[shape=circle,fill=gray!0,inner sep=0pt] {$\circ$};
\draw (288:1.5) node[shape=circle,fill=gray!0,inner sep=0pt] {$\circ$};
\draw (2.5,0) node {$\tot$};
\end{scope}
\end{scope}
\end{tikzpicture}
\end{center}
The join move can be used to remove all $2$-valent nodes from the dimer graph to obtain a reduced dimer. 
It is easy to see that the end result does not depend on the order of the join moves, so 
this reduced dimer is unique. 

\item
The \iemph{spider move} does locally the following:
\begin{center}
\begin{tikzpicture}
\begin{scope}[scale=.5]
\begin{scope}
\draw[dotted] (1.5,0) arc (0:360:1.5);
\draw (-1.5,0) -- (-.5,0) -- (0,1.5) -- (.5,0) -- (0,1.5);
\draw[rotate=180] (-1.5,0) -- (-.5,0) -- (0,1.5) -- (.5,0) -- (0,1.5);
\draw (-1.5,0) node[shape=circle,fill=gray!0,inner sep=0pt] {$\circ$};
\draw (1.5,0) node[shape=circle,fill=gray!0,inner sep=0pt] {$\circ$};
\draw (-.5,0) node[shape=circle,fill=gray!0,inner sep=0pt] {$\bullet$};
\draw (.5,0) node[shape=circle,fill=gray!0,inner sep=0pt] {$\bullet$};
\draw (0,-1.5) node[shape=circle,fill=gray!0,inner sep=0pt] {$\circ$};
\draw (0,1.5) node[shape=circle,fill=gray!0,inner sep=0pt] {$\circ$};
\end{scope}
\draw (2.5,0) node {$\tot$};
\begin{scope}[xshift=5cm]
\draw[dotted] (1.5,0) arc (0:360:1.5);
\draw[rotate=90] (-1.5,0) -- (-.5,0) -- (0,1.5) -- (.5,0) -- (0,1.5);
\draw[rotate=270] (-1.5,0) -- (-.5,0) -- (0,1.5) -- (.5,0) -- (0,1.5);
\draw (-1.5,0) node[shape=circle,fill=gray!0,inner sep=0pt] {$\circ$};
\draw (1.5,0) node[shape=circle,fill=gray!0,inner sep=0pt] {$\circ$};
\draw (0,-.5) node[shape=circle,fill=gray!0,inner sep=0pt] {$\bullet$};
\draw (0,.5) node[shape=circle,fill=gray!0,inner sep=0pt] {$\bullet$};
\draw (0,-1.5) node[shape=circle,fill=gray!0,inner sep=0pt] {$\circ$};
\draw (0,1.5) node[shape=circle,fill=gray!0,inner sep=0pt] {$\circ$};
\end{scope}
\end{scope}

\begin{scope}[xshift=6cm]
\begin{scope}[scale=.5]
\begin{scope}
\draw[dotted] (1.5,0) arc (0:360:1.5);
\draw [-latex,shorten >= 2pt] (-1.06,-1.06) -- (0,0);
\draw [-latex,shorten >= 2pt] (1.06,1.06) -- (0,0);
\draw [-latex,shorten >= 2pt] (0,0)-- (-1.06,1.06) ;
\draw [-latex,shorten >= 2pt] (0,0) --(1.06,-1.06) ;
\draw [-latex,shorten >= 2pt] (1.06,-1.06) -- (1.06,1.06);
\draw [-latex,shorten >= 2pt] (-1.06,1.06) -- (-1.06,-1.06);
\draw (-1.06,-1.06) node[shape=circle,fill=gray!0,inner sep=0pt] {$\circ$};
\draw (-1.06,1.06) node[shape=circle,fill=gray!0,inner sep=0pt] {$\circ$};
\draw (1.06,-1.06) node[shape=circle,fill=gray!0,inner sep=0pt] {$\circ$};
\draw (1.06,1.06) node[shape=circle,fill=gray!0,inner sep=0pt] {$\circ$};
\draw (0,0) node[shape=circle,fill=gray!0,inner sep=0pt] {$\circ$};
\end{scope}
\draw (2.5,0) node {$\tot$};
\begin{scope}[xshift=5cm]
\draw[dotted] (1.5,0) arc (0:360:1.5);
\draw [-latex,shorten >= 2pt] (-1.06,1.06) -- (0,0);
\draw [-latex,shorten >= 2pt] (1.06,-1.06) -- (0,0);
\draw [-latex,shorten >= 2pt] (0,0)-- (-1.06,-1.06) ;
\draw [-latex,shorten >= 2pt] (0,0) --(1.06,1.06) ;
\draw [-latex,shorten >= 2pt] (1.06,1.06) -- (-1.06,1.06);
\draw [-latex,shorten >= 2pt] (-1.06,-1.06) -- (1.06,-1.06);
\draw (-1.06,-1.06) node[shape=circle,fill=gray!0,inner sep=0pt] {$\circ$};
\draw (-1.06,1.06) node[shape=circle,fill=gray!0,inner sep=0pt] {$\circ$};
\draw (1.06,-1.06) node[shape=circle,fill=gray!0,inner sep=0pt] {$\circ$};
\draw (1.06,1.06) node[shape=circle,fill=gray!0,inner sep=0pt] {$\circ$};
\draw (0,0) node[shape=circle,fill=gray!0,inner sep=0pt] {$\circ$};
\end{scope}
\end{scope}
\end{scope}
\end{tikzpicture}
\end{center}
In the quiver picture we take a $4$-valent vertex, reverse the arrows that meet it and add 4 extra arrows for the paths of length two that ran through this vertex.

The spider move was first introduced by Kuperberg and is also known as urban renewal
\cite{ciucu1998complementation,goncharov2013dimers,kenyon2001trees}
\end{enumerate}
\end{definition}
The spider move is closely related to \iemph{quiver mutations} that were introduced by Fomin and Zelevinski in \cite{fomin2002cluster,fomin2003cluster}. The mutation of a quiver $Q$ at a vertex $v$, reverses all arrows in $v$, adds extra arrows
for all paths of length $2$ trough $v$ and then deletes all $2$-cycles. We will come back to quiver mutations in \ref{qmutsec}.
It is easy to see that mutation of $4$-valent vertex in a dimer quiver, can always be seen as a combination of join/splits-moves and a spider move.  
If you mutate a dimer quiver at a vertex with higher valency, the result cannot be interpreted as a new dimer.

\begin{lemma}
A $2\tot 2$-move on a triple crossing diagram corresponds to a 
\begin{enumerate}
\item
A join move followed by a split move if the middle piece contains a white node 
\item
A spider move if the middle piece is empty.
\end{enumerate}
\end{lemma}
\begin{proof}~\\
\begin{center}
\begin{tikzpicture} 
\begin{scope}[scale=.75]
\draw[dotted] (1.118,0) arc (0:360:1.118);
\draw[dotted] (-.5,-1) -- (0,-.5) .. controls (.5,0) and (.5,0) .. (0,.5) -- (-.5,1);
\draw[dotted,rotate=180] (-.5,-1) -- (0,-.5) .. controls (.5,0) and (.5,0) .. (0,.5) -- (-.5,1);
\draw[dotted] (-1,.5) -- (1,.5);
\draw[dotted] (-1,-.5) -- (1,-.5);
\draw[thick] (0,-.5) -- (0,.5) -- (45:1.118);
\draw[thick] (0,.5) -- (135:1.118);
\draw[thick] (0,-.5) -- (225:1.118);
\draw[thick] (0,-.5) -- (315:1.118);
\draw (0,0) node[shape=circle,fill=gray!0,inner sep=0pt] {$\circ$};
\draw (0,.5) node[shape=circle,fill=gray!0,inner sep=0pt] {$\bullet$};
\draw (0,-.5) node[shape=circle,fill=gray!0,inner sep=0pt] {$\bullet$};
\draw (2,0) node {$\tot$};
\begin{scope}[xshift=4cm]
\begin{scope}[rotate=270]
\draw[dotted] (1.118,0) arc (0:360:1.118);
\draw[dotted] (-.5,-1) -- (0,-.5) .. controls (.5,0) and (.5,0) .. (0,.5) -- (-.5,1);
\draw[dotted,rotate=180] (-.5,-1) -- (0,-.5) .. controls (.5,0) and (.5,0) .. (0,.5) -- (-.5,1);
\draw[dotted] (-1,.5) -- (1,.5);
\draw[dotted] (-1,-.5) -- (1,-.5);
\draw[thick] (0,-.5) -- (0,.5) -- (45:1.118);
\draw[thick] (0,.5) -- (135:1.118);
\draw[thick] (0,-.5) -- (225:1.118);
\draw[thick] (0,-.5) -- (315:1.118);
\draw (0,0) node[shape=circle,fill=gray!0,inner sep=0pt] {$\circ$};
\draw (0,.5) node[shape=circle,fill=gray!0,inner sep=0pt] {$\bullet$};
\draw (0,-.5) node[shape=circle,fill=gray!0,inner sep=0pt] {$\bullet$};
\end{scope}
\end{scope}
\end{scope}

\begin{scope}[xshift=6cm]
\begin{scope}[scale=.75]
\draw[dotted] (1.118,0) arc (0:360:1.118);
\draw[dotted] (-.5,-1) -- (0,-.5) .. controls (.5,0) and (.5,0) .. (0,.5) -- (-.5,1);
\draw[dotted,rotate=180] (-.5,-1) -- (0,-.5) .. controls (.5,0) and (.5,0) .. (0,.5) -- (-.5,1);
\draw[dotted] (-1,.5) -- (1,.5);
\draw[dotted] (-1,-.5) -- (1,-.5);
\draw[thick] (0,.5) -- (0:1.118);
\draw[thick] (0,.5) -- (90:1.118);
\draw[thick] (0,.5) -- (180:1.118);
\draw[thick] (0,-.5) -- (0:1.118);
\draw[thick] (0,-.5) -- (270:1.118);
\draw[thick] (0,-.5) -- (180:1.118);
\draw (0:1.118) node[shape=circle,fill=gray!0,inner sep=0pt] {$\circ$};
\draw (90:1.118) node[shape=circle,fill=gray!0,inner sep=0pt] {$\circ$};
\draw (180:1.118) node[shape=circle,fill=gray!0,inner sep=0pt] {$\circ$};
\draw (270:1.118) node[shape=circle,fill=gray!0,inner sep=0pt] {$\circ$};
\draw (0,.5) node[shape=circle,fill=gray!0,inner sep=0pt] {$\bullet$};
\draw (0,-.5) node[shape=circle,fill=gray!0,inner sep=0pt] {$\bullet$};
\draw (2,0) node {$\tot$};
\begin{scope}[xshift=4cm]
\begin{scope}[rotate=270]
\draw[dotted] (1.118,0) arc (0:360:1.118);
\draw[dotted] (-.5,-1) -- (0,-.5) .. controls (.5,0) and (.5,0) .. (0,.5) -- (-.5,1);
\draw[dotted,rotate=180] (-.5,-1) -- (0,-.5) .. controls (.5,0) and (.5,0) .. (0,.5) -- (-.5,1);
\draw[dotted] (-1,.5) -- (1,.5);
\draw[dotted] (-1,-.5) -- (1,-.5);
\draw[thick] (0,.5) -- (0:1.118);
\draw[thick] (0,.5) -- (90:1.118);
\draw[thick] (0,.5) -- (180:1.118);
\draw[thick] (0,-.5) -- (0:1.118);
\draw[thick] (0,-.5) -- (270:1.118);
\draw[thick] (0,-.5) -- (180:1.118);
\draw (0:1.118) node[shape=circle,fill=gray!0,inner sep=0pt] {$\circ$};
\draw (90:1.118) node[shape=circle,fill=gray!0,inner sep=0pt] {$\circ$};
\draw (180:1.118) node[shape=circle,fill=gray!0,inner sep=0pt] {$\circ$};
\draw (270:1.118) node[shape=circle,fill=gray!0,inner sep=0pt] {$\circ$};
\draw (0,.5) node[shape=circle,fill=gray!0,inner sep=0pt] {$\bullet$};
\draw (0,-.5) node[shape=circle,fill=gray!0,inner sep=0pt] {$\bullet$};
\end{scope}
\end{scope}
\end{scope}
\end{scope}
\end{tikzpicture}
\end{center}
\end{proof}

Using this observation it clear that 
\begin{lemma}\label{zigzagbij}
A quiver mutation on a 4-valent vertex induces a bijection between the zigzag paths that maps each zigzag path of $\qpol$ to a zigzag path of $\mu\qpol$ with the same homotopy class.
\end{lemma}
\begin{proof}
This is because both the join/split moves and the spider move induces these bijections on the strands.
\end{proof}

We can also translate theorems \ref{construction} and \ref{movetcd} to the setting of reduced dimers. 
The first states that for every surface with boundary and every choice of homotopy classes with total relative homology $0$
we can find a zigzag consistent dimer. 
The second states that reduced zigzag consistent homotopic dimers on the disk are related by quiver mutation.
The statement that reduced zigzag consistent homotopic dimers on general surfaces are related by mutation is still open.

\newcommand{\TT}{\mathbb{T}}
\subsection{Dimer models on a torus}
In the case that the underlying surface of the dimer is a torus without boundary, zigzag consistency has been studied in detail.
In this section we will describe some equivalent notions of consistency for dimers on a torus.

\subsubsection{The zigzag polygon}

The homology group of the torus is $\Z^2$. This means that we can define a cyclic order on the directions of closed curves on the torus
by looking at the projection of the homology class on the unit circle $(a,b)\mapsto (a,b)/\sqrt{a^2+b^2}$. In \cite{gulotta2008properly} Gulotta used this cyclic order to introduce the notion
of a properly ordered dimer.

\begin{definition}
A dimer quiver $\qpol$ is called \iemph{properly ordered} if the cyclic order of the zigzag paths incident with a given cycle in $\qpol_2$ is the same as the cyclic order
of their homology classes. (This also implies zigzag paths that meet the same cycle cannot have the same direction.)
\end{definition}
\begin{theorem}[Ishii-Ueda \cite{ishii2010note} Proposition 4.4]
A dimer on a torus is zigzag consistent if and only if it is properly ordered. 
\end{theorem}
\begin{proof}
Two zigzag paths that are consecutive in a cycle share an arrow. If they swap order at infinity they must intersect a second time which results
in a bad digon. If they have the same direction at infinity their lifts in the universal cover must intersect an infinity number of times because of
periodicity. This also results in a bad digon.

If two zigzag paths are not consecutive then we can use induction to find a bad digon. Look at the zigzag paths in between. Because an intermediate zigzag path must 
leave the piece cut out by the two original zigzag paths, it must swap order with one of these two original zigzag paths. In this way we get a pair of zigzag paths that swaps with less zigzag paths in between.

In the other direction it is clear that a bad digon between two zigzag paths results in two common arrows. The order of the zigzag paths at the two anticlockwise cycles
containing these arrows is reversed.
\end{proof}

\begin{definition}
Let $\pZ_1,\dots,\pZ_k$ be the zigzag paths of a dimer quiver $\qpol$, ordered cyclicly by their homology classes $\vec v_1,\dots,\vec v_k \in \H_1(\TT)=\Z^2$ and
for each vector $\vec v_i=(a,b)$ we denote its normal by $\vec v_i^\perp=(-b,a)$.

The \iemph{zigzag polygon} $\ZP(\qpol)$ is the convex hull of the points $$\vec v_1^\perp, \vec v_1^\perp+\vec v_2^\perp, \dots, \vec v_1^\perp+\dots+\vec v_k^\perp.$$
In other words it is the convex lattice polygon whose outward pointing normals are the homology classes of the zigzag paths.
\end{definition}

\begin{example}
If we look at the suspended pinchpoint \ref{spp} then there are five zigzag paths with homology classes
$(1,0),(2,1),2\times (-1,0),(-1,-1)$. This results in the following zigzag polygon.
\begin{center}
\begin{tikzpicture}
\draw[thick](0,0) -- (0,2) -- (1,0) -- (1,-1) -- (0,0); 
\draw (0,0) node{$\bullet$};
\draw (0,2) node{$\bullet$};
\draw (0,1) node{$\bullet$};
\draw (1,0) node{$\bullet$};
\draw (1,-1) node{$\bullet$};
\draw[-latex] (0,.5)--(-.5,.5);
\draw[-latex] (0,1.5)--(-.5,1.5);
\draw[-latex] (.5,1)--(1.5,1.5);
\draw[-latex] (1,-.5)--(1.5,-.5);
\draw[-latex] (.5,-.5)--(0,-1);
\begin{scope}[xshift=-1.5cm]
\begin{scope}[scale=.25]
\draw [dotted] (0,0) -- (3,0) -- (3,3) -- (0,3);
\draw [dotted] (0,0) -- (3,1) -- (3,2) -- (0,1)--(0,2)--(3,3);
\draw [thick,-latex] (3,0) -- (3,1)--(0,0);
\end{scope}
\end{scope}
\begin{scope}[xshift=-1.5cm,yshift=1cm]
\begin{scope}[scale=.25]
\draw [dotted] (0,0) -- (3,0) -- (3,3) -- (0,3);
\draw [dotted] (0,0) -- (3,1) -- (3,2) -- (0,1)--(0,2)--(3,3);
\draw [thick,-latex] (3,2) -- (3,3)--(0,2);
\end{scope}
\end{scope}
\begin{scope}[xshift=1.75cm,yshift=-1cm]
\begin{scope}[scale=.25]
\draw [dotted] (0,0) -- (3,0) -- (3,3) -- (0,3);
\draw [dotted] (0,0) -- (3,1) -- (3,2) -- (0,1)--(0,2)--(3,3);
\draw [thick,-latex] (0,1) -- (3,2)--(3,1);
\end{scope}
\end{scope}
\begin{scope}[xshift=-1.75cm,yshift=-1.5cm]
\begin{scope}[scale=.25]
\draw [dotted] (0,0) -- (3,0) -- (3,3) -- (0,3);
\draw [dotted] (0,0) -- (3,1) -- (3,2) -- (0,1)--(0,2)--(3,3);
\draw [dotted] (3,0) -- (6,0) -- (6,3) -- (3,3);
\draw [dotted] (3,0) -- (6,1) -- (6,2) -- (3,1)--(3,2)--(6,3);
\draw [thick,-latex] (6,3) -- (3,2)--(3,1)--(0,0)--(3,0);
\end{scope}
\end{scope}
\begin{scope}[xshift=1.75cm,yshift=1cm]
\begin{scope}[scale=.25]
\draw [dotted] (0,0) -- (3,0) -- (3,3) -- (0,3);
\draw [dotted] (0,0) -- (3,1) -- (3,2) -- (0,1)--(0,2)--(3,3);
\draw [dotted] (3,0) -- (6,0) -- (6,3) -- (3,3);
\draw [dotted] (3,0) -- (6,1) -- (6,2) -- (3,1)--(3,2)--(6,3);
\draw [thick,-latex] (0,0) -- (3,0)--(3,1)--(6,2)--(6,3);
\end{scope}
\end{scope}
\end{tikzpicture}
\end{center}
\end{example}

\begin{lemma}
Two dimer models on a torus are homotopic if and only if their zigzag polygons are equivalent under an affine transformation. 
\end{lemma}

\begin{theorem}[Existence of zigzag consistent dimers on a torus]
Let $\ZP \subset \Z^2$ be any convex lattice polygon then there exists a zigzag consistent dimer on a torus, $\qpol$, with
$\ZP = \ZP(\qpol)$.
\end{theorem}
\begin{proof}
This follows immediately from theorem \ref{construction}. Just take for the $\kappa_i$ curves with homology classes
perpendicular to the boundary of the zigzag polygon.
\end{proof}
This statement was first proved by Gulotta in \cite{gulotta2008properly} where
he gave an explicit algorithm to produce such a dimer. A similar proof by Ishii and Ueda,
which uses the special McKay correspondence can be found in \cite{ishii2009dimer}.
An algorithm that produces all consistent dimers for a given zigzag polygon can be found in \cite{bocklandt2012generating}.

\subsubsection{Isoradial embeddings}

The universal cover of a torus is the Euclidean plane, so studying dimers on a torus is the same as studying periodic dimers embedded in the Euclidean plane. We start with a definition due to Duffin \cite{duffin1968potential} and Mercat \cite{mercat2001discrete}
\begin{definition}
A periodic dimer quiver $\tilde \qpol$ is \iemph{isoradially embedded} in the plane if all cycles are polygons inscribed in circles with radius one.
\end{definition}
In this case every arrow $a \in \qpol_1$ stands on an arc $\theta_a \in (0,2\pi)$ and we have for every cycle in $\qpol_2$ that $\sum_{a} \theta_a=2 \pi$.
Furthermore if $ab$ are two consecutive arrows in a cycle then the inscribed angle theorem tells us that the angle between the arrows is $\frac 12(2\pi - \theta_a -\theta_b)$.
Taking the sum of all these angles around a vertex $v \in \qpol_0$ we get that $\sum_{h(b)=v} (\pi - \theta_a)+\sum_{t(a)=v} (\pi - \theta_b)=2\pi$.

Specifying the arc lengths $\theta_a$ of the arrows of the quiver characterizes the embedding up to isometry. This gives rise to the notion of a consistent $R$-charge (see \cite{gulotta2008properly}).
\begin{definition}
A dimer $(\rib,\qpol)$ has a consistent $R$-charge if there is a map $\cR : \qpol_1 \to (0,2)$ such that
\begin{itemize}
\item[R1] $\forall a_1 \dots a_k \in\qpol_2: \sum_i \cR_{a_i} = 2$,
\item[R2] $\forall v \in \qpol_0: \sum_{h(a)=v} (1-\cR_{a}) + \sum_{t(a)=v} (1-\cR_{a}) = 2$
\end{itemize}
\end{definition}
If we multiply the R-charges by $\pi$ we get the angles of an isoradially embedded dimer quiver: $\theta_a := \pi \cR_a$. 
\begin{lemma}
A dimer quiver $\qpol$ can be isoradially embedded if and only if it admits a consistent $\cR$-charge. 
\end{lemma}

We can also talk about isoradially embedded periodic dimer graphs, if all faces are polygons inscribed in circles with radius one.
There is a close connection between these two notions. Given an isoradially embedded dimer graph, we can assign to each edge $e$ a number $\theta_e \in (0,2 \pi)$
that measures the angle of the arc on which the edge stands. If each individual angle is smaller than $\pi$
the circumcenters lie inside the faces. 
Such dimers are called \iemph{geometrically consistent} and have been studied by Broomhead in \cite{broomhead2012dimer}.
If we connect the nodes with the centers we get a tiling of the plane with rhombi with sides of length one.
This is called a \iemph{rhombus tiling} or \iemph{quad-tiling} (see \cite{kenyon2005rhombic}). The edges of the dimer graph are diagonals of the rhombus tiling. The other diagonals can be seen as the arrows of the dual dimer quiver.
This gives an isoradial embedding of the quiver with $\theta_a = \pi - \theta_e$ if $e=a^\vee$. The process also goes in reverse:
an isoradially embedded dimer quiver gives an isoradially embedded dimer graph with supplementary angles, provided all angles are not bigger than $\pi$.

If we look at zigzag paths of a dimer quiver from the rhombus tiling point of view, we see that the arrows of the zigzag paths all sit in rhombi for which 
one pair of sides are all parallel. It looks like a train track of which these parallel sides are the sleepers.
We define the sleeper direction $\theta_\pZ$ of a zigzag path $\pZ$ to be the direction of the sleepers pointing to the right of the direction of the zigzag path. 
\vspace{.3cm}
\begin{center}
\begin{tikzpicture}
\draw[dashed, latex-] (0,0-1)--(0,0); \draw[dotted] (0,0-1)--(0.996917333733,-1.07845909573); \draw[dashed, latex-] (0.996917333733,-1.07845909573)--(0.996917333733,-1.07845909573+1); \draw[dotted] (0,0)--(0.996917333733,-1.07845909573+1); \draw[thick,-latex] (0,0)--(0.996917333733,-1.07845909573); \draw[dashed, latex-] (0.996917333733,-1.07845909573)--(0.996917333733,-1.07845909573+1); \draw[dotted] (0.996917333733,-1.07845909573+1) circle (1cm);\draw[dotted] (0.996917333733,-1.07845909573+1)--(1.92079686624,0.304224336637); \draw[dashed, latex-] (1.92079686624,0.304224336637-1)--(1.92079686624,0.304224336637); \draw[dotted] (0.996917333733,-1.07845909573)--(1.92079686624,0.304224336637-1); \draw[thick,-latex] (0.996917333733,-1.07845909573)--(1.92079686624,0.304224336637); \draw[dashed, latex-] (1.92079686624,0.304224336637-1)--(1.92079686624,0.304224336637); \draw[dotted] (1.92079686624,0.304224336637-1)--(2.91771419998,-0.774234759091); \draw[dashed, latex-] (2.91771419998,-0.774234759091)--(2.91771419998,-0.774234759091+1); \draw[dotted] (1.92079686624,0.304224336637)--(2.91771419998,-0.774234759091+1); \draw[thick,-latex] (1.92079686624,0.304224336637)--(2.91771419998,-0.774234759091); \draw[dashed, latex-] (2.91771419998,-0.774234759091)--(2.91771419998,-0.774234759091+1); \draw[dotted] (2.91771419998,-0.774234759091+1) circle (1cm);\draw[dotted] (2.91771419998,-0.774234759091+1)--(3.89008412038,-0.00768012294651); \draw[dashed, latex-] (3.89008412038,-0.00768012294651-1)--(3.89008412038,-0.00768012294651); \draw[dotted] (2.91771419998,-0.774234759091)--(3.89008412038,-0.00768012294651-1); \draw[thick,-latex] (2.91771419998,-0.774234759091)--(3.89008412038,-0.00768012294651); \draw[dashed, latex-] (3.89008412038,-0.00768012294651-1)--(3.89008412038,-0.00768012294651); \draw[dotted] (3.89008412038,-0.00768012294651-1)--(4.86245404077,-0.774234759091); \draw[dashed, latex-] (4.86245404077,-0.774234759091)--(4.86245404077,-0.774234759091+1); \draw[dotted] (3.89008412038,-0.00768012294651)--(4.86245404077,-0.774234759091+1); \draw[thick,-latex] (3.89008412038,-0.00768012294651)--(4.86245404077,-0.774234759091); \draw[dashed, latex-] (4.86245404077,-0.774234759091)--(4.86245404077,-0.774234759091+1); \draw[dotted] (4.86245404077,-0.774234759091+1) circle (1cm);\draw[dotted] (4.86245404077,-0.774234759091+1)--(5.83482396117,0.459210604765); \draw[dashed, latex-] (5.83482396117,0.459210604765-1)--(5.83482396117,0.459210604765); \draw[dotted] (4.86245404077,-0.774234759091)--(5.83482396117,0.459210604765-1); \draw[thick,-latex] (4.86245404077,-0.774234759091)--(5.83482396117,0.459210604765); \draw[dashed, latex-] (5.83482396117,0.459210604765-1)--(5.83482396117,0.459210604765); \draw[dotted] (5.83482396117,0.459210604765-1)--(6.8317412949,-0.619248490963); \draw[dashed, latex-] (6.8317412949,-0.619248490963)--(6.8317412949,-0.619248490963+1); \draw[dotted] (5.83482396117,0.459210604765)--(6.8317412949,-0.619248490963+1); \draw[thick,-latex] (5.83482396117,0.459210604765)--(6.8317412949,-0.619248490963); \draw[dashed, latex-] (6.8317412949,-0.619248490963)--(6.8317412949,-0.619248490963+1); \draw[dotted] (6.8317412949,-0.619248490963+1) circle (1cm);\draw[dotted] (6.8317412949,-0.619248490963+1)--(7.75562082741,-0.00193192332763); \draw[dashed, latex-] (7.75562082741,-0.00193192332763-1)--(7.75562082741,-0.00193192332763); \draw[dotted] (6.8317412949,-0.619248490963)--(7.75562082741,-0.00193192332763-1); \draw[thick,-latex] (6.8317412949,-0.619248490963)--(7.75562082741,-0.00193192332763);
\end{tikzpicture}
\end{center}
\vspace{.3cm}
Note that the sleeper direction of a zigzag path is also the direction of the middle vertex of the intersection of the zigzag path with a positive cycle, viewed from
the circumcenter of that cycle.

If a dimer model is properly ordered we can construct an isoradial embedding of the periodic dimer quiver as follows.
Choose points $\theta_1,\dots \theta_k$ on the unit circle that have the same cyclic order as the directions of the zigzag paths $\vec v_1,\dots \vec v_k$.
make sure that if two zigzag paths have the same homology class, their points on the unit circle must also coincide. 

For each positive cycle we can connect the points on the unit circle that come from zigzag paths that meet this cycle. 
Because $\qpol$ is properly ordered this gives a convex polygon and each arrow $a$ in the cycle is identified with the side of the polygon
that connects the angles of the zigzag paths that contain $a$.
For each negative cycle we do the same but we rotate the polygon over $180^\circ$. 
Then we translate all these polygons such that sides representing the same arrow are identified.

Using the standard facts of inscribed angles one can show that all these polygons fit together to form a tiling of the plane.
This gives us an isoradial embedding of the periodic dimer quiver for which the sleeper directions are precisely the chosen directions $\theta_i$.

\begin{example}
If we return to the suspended pinchpoint, we already saw that the homology classes of the zigzag paths are $(1,0)$, $(2,1)$, $2\times (-1,0)$, $(-1,-1)$.
We can assign sleeper directions $\theta=-90^\circ,-45^\circ, 2\times 90^\circ, 180^\circ$.
\begin{center}
\begin{tikzpicture}
\begin{scope}
\draw [-latex,shorten >=5pt] (0,0) -- (3,0);
\draw [-latex,shorten >=5pt] (3,0) -- (3,1);
\draw [-latex,shorten >=5pt] (3,1) -- (0,0);
\draw [-latex,shorten >=5pt] (0,0) -- (0,1);
\draw [-latex,shorten >=5pt] (0,1) -- (3,2);
\draw [-latex,shorten >=5pt] (3,2) -- (3,3);
\draw [-latex,shorten >=5pt] (3,2) -- (3,1);
\draw [-latex,shorten >=5pt] (0,2) -- (0,1);
\draw [-latex,shorten >=5pt] (0,3) -- (3,3);
\draw [-latex,shorten >=5pt] (3,3) -- (0,2);
\draw [-latex,shorten >=5pt] (0,2) -- (0,3);
\draw (0,0) node[circle,draw,fill=white,minimum size=10pt,inner sep=1pt] {{\tiny1}};
\draw (0,1) node[circle,draw,fill=white,minimum size=10pt,inner sep=1pt] {{\tiny2}};
\draw (0,2) node[circle,draw,fill=white,minimum size=10pt,inner sep=1pt] {{\tiny3}};
\draw (0,3) node[circle,draw,fill=white,minimum size=10pt,inner sep=1pt] {{\tiny1}};
\draw (3,0) node[circle,draw,fill=white,minimum size=10pt,inner sep=1pt] {{\tiny1}};
\draw (3,1) node[circle,draw,fill=white,minimum size=10pt,inner sep=1pt] {{\tiny2}};
\draw (3,2) node[circle,draw,fill=white,minimum size=10pt,inner sep=1pt] {{\tiny3}};
\draw (3,3) node[circle,draw,fill=white,minimum size=10pt,inner sep=1pt] {{\tiny1}};
\end{scope}
\begin{scope}[xshift=6cm,yshift=1.5cm]
\begin{scope}[rotate=270]
\draw[dotted] (1,0) arc (0:360:1);
\draw (0:1)--(45:1)--(180:1)--(270:1)--(360:1);
\end{scope}
\draw (-90:1) node[below] {{\small{(1,0)}}};
\draw (-45:1) node[right] {{\small{(2,1)}}};
\draw (90:1) node[above] {{\small{(-1,0)}}};
\draw (180:1) node[left] {{\small{(-1,-1)}}};
\end{scope}
\begin{scope}[xshift=10.5cm,yshift=0cm]
\begin{scope}[scale=.75]
\begin{scope}[rotate=270]
\draw (45:1)--(180:1)--(270:1)--(45:1);
\draw ($(180:1)-(90:1)+(180:1)$)--($(180:1)-(90:1)+(225:1)$)--($(180:1)-(90:1)+(0:1)$)--($(180:1)-(90:1)+(90:1)$)--($(180:1)-(90:1)+(180:1)$);
\draw ($(180:1)-(90:1)+(180:1)-(45:1)+(0:1)$)--($(180:1)-(90:1)+(180:1)-(45:1)+(45:1)$)--($(180:1)-(90:1)+(180:1)-(45:1)+(180:1)$)--($(180:1)-(90:1)+(180:1)-(45:1)+(270:1)$)--($(180:1)-(90:1)+(180:1)-(45:1)+(360:1)$);
\draw ($(180:1)-(90:1)+(180:1)-(45:1)+(270:1)-(0:1)+(225:1)$)--($(180:1)-(90:1)+(180:1)-(45:1)+(270:1)-(0:1)+(0:1)$)--($(180:1)-(90:1)+(180:1)-(45:1)+(270:1)-(0:1)+(90:1)$)--($(180:1)-(90:1)+(180:1)-(45:1)+(270:1)-(0:1)+(225:1)$);
\draw (45:1) node[circle,draw,fill=white,minimum size=10pt,inner sep=1pt] {{\tiny1}};
\draw (270:1) node[circle,draw,fill=white,minimum size=10pt,inner sep=1pt] {{\tiny1}};
\draw ($(180:1)-(90:1)+(90:1)$) node[circle,draw,fill=white,minimum size=10pt,inner sep=1pt] {{\tiny2}};
\draw ($(180:1)-(90:1)+(180:1)-(45:1)+(45:1)$) node[circle,draw,fill=white,minimum size=10pt,inner sep=1pt] {{\tiny3}};
\draw ($(180:1)-(90:1)+(180:1)-(45:1)+(180:1)$) node[circle,draw,fill=white,minimum size=10pt,inner sep=1pt] {{\tiny1}};
\draw ($(180:1)-(90:1)+(180:1)-(45:1)+(270:1)$) node[circle,draw,fill=white,minimum size=10pt,inner sep=1pt] {{\tiny3}};
\draw ($(180:1)-(90:1)+(180:1)-(45:1)+(0:1)$) node[circle,draw,fill=white,minimum size=10pt,inner sep=1pt] {{\tiny2}};
\draw ($(180:1)-(90:1)+(180:1)-(45:1)+(270:1)-(0:1)+(225:1)$) node[circle,draw,fill=white,minimum size=10pt,inner sep=1pt] {{\tiny1}};
\end{scope}
\end{scope}
\end{scope}
\end{tikzpicture}
\end{center}
\end{example}

\begin{theorem}[\cite{bocklandt2012consistency}]
A dimer is zigzag consistent if and only if its \iemph{quiver} can be isoradially embedded in the plane.
\end{theorem}
\begin{proof}
The condition is neccesary because zigzag consistency implies well ordering, which we can use to construct an isoradial embedding as explained above. 

To show that an isoradial embedded quiver is zigzag consistent, note that two zigzag paths meeting in one arrow must come from a different corner on the zigzag polygon.
Again, standard Euclidean geometry for polygons inscribed in circles shows that there cannot be monogons or bad digons.
\end{proof}
If one restricts to geometrically consistent dimers one has a similar theorem. The result
states that a dimer on a torus is geometrically consistent if and only if the strands of the zigzag paths of the dimer have no monogons, bad digons or good digons.
A proof of this statement is in \cite{kenyon2005rhombic}.

\begin{theorem}[Summary of consistency]
If $(\rib,\qpol)$ is a dimer on a torus then the following properties are equivalent.
\begin{enumerate}
 \item minimal,
 \item zigzag consistent,
 \item properly ordered,
 \item isoradially embeddable,
 \item $R$-charge consistent.
\end{enumerate}
If these conditions hold we call $\qpol$ \iemph{consistent}.
\end{theorem}

\subsubsection{Perfect matchings}

It is natural to describe the homology of the torus using the dimer quiver because we can represent the homology classes by oriented cycles in the quiver.
Every path in the quiver $p$ corresponds to an element $\bar p \in \Z\qpol_1$, which is the sum of all arrows in the path. Cyclic paths in the quiver
give rise to elements in $\H_1(\TT)=\H_1(\Z\qpol_\bullet)$. Note that all cycles in $\qpol_2$ give zero elements in $\H_1(\TT)$.

The complex $\Z\qpol_\bullet$ has a natural pairing with $\Z\rib_{2-\bullet}$ realized by the natural identifications of the vertices and the faces, the arrows and the edges, 
and the cycles and the nodes. We will denote these pairings by $\<,\>$. Therefore it is natural to identify the homology of
$\Z\rib_{2-\bullet}$ with the cohomology of the torus.

\begin{definition}
A \iemph{unit flow} is a weight map $\phi:\rib_1 \to \R_{\ge 0}$ such that  
in every node the sum the weights of all edges is $1$.

A \iemph{perfect matching} is a subset $\cP \subset \rib_1$ such that every node is incident with exactly one edge in $\cP$. The set of all perfect matchings of
a dimer is denoted by $\PMs(\rib)$.

If we assign weight one to every edge in a given perfect matching $\cP$ and
weight zero to the others then we get a unit flow $\phi_\cP$.

In some cases it is convenient to consider $\cP$ as a subset of $\qpol_1$ using the identification of arrows and edges.
In this way $\cP$ is a subset of arrows that contains exactly one arrow in each cycle of $\qpol_2$.

Given a path $p$ in the quiver, the pairing $\<\bar p,\cP\>$ counts how many arrows of $p$ sit in $\cP$. We sometimes will write this pairing also as $\deg_{\cP} p$.
If $\<\bar p,\cP\>\ne 0$ we say that $p$ and $\cP$ meet.
\end{definition}

\begin{lemma}\label{flowpm}
Every unit flow is a positive linear combination of perfect matchings.
\end{lemma}
\begin{proof}
Let $\phi$ be a unit flow, we show that there is a perfect matching $\cP$
with $\phi(e)\ne 0$ for all $e \in \cP$. Take an edge $e$ with $\phi(e)<1$
then this edge is connected on both sides to another edge. Because the dimer
is finite we can find a cyclic walk $e_1e_2\dots e_{2i}$ of edges with weights all smaller than $1$. This walk has even length because $\rib$ is bipartite.
Assume that $e_1$ has the smallest weight. Now subtract $\phi(e_1)$ from
all odd edges and add it to the even edges. This gives a new unit flow $\phi'$
that is nonzero on fewer edges. If we continu like this we get
a $\phi_{\cP}$ that comes from a perfect matching.
\end{proof}

Not every dimer has a perfect matching but if the dimer quiver is isoradially embedded we can easily construct perfect matchings.
Fix a point $\theta \in \SS_1$. For each arrow $a$ we can look at the unit circle around its positive cycle and draw $\theta$ on it.
Let $\cP_\theta$ be the set of all edges $a^\vee$ corresponding to arrows $a$ for which the arc on which they stand contains $\theta$.
\begin{center}
\begin{tikzpicture} 
\draw[dotted] (1,0) arc (0:360:1);
\draw[-latex,shorten >=5pt] (0:1)--(45:1); 
\draw[-latex,shorten >=5pt] (45:1)--(180:1); 
\draw[-latex,shorten >=5pt] (180:1)--(270:1); 
\draw[-latex,shorten >=5pt] (270:1)--(0:1); 
\draw (0:1) node[circle,draw,fill=white,minimum size=10pt,inner sep=1pt] {};
\draw (45:1) node[circle,draw,fill=white,minimum size=10pt,inner sep=1pt] {};
\draw (180:1) node[circle,draw,fill=white,minimum size=10pt,inner sep=1pt] {};
\draw (270:1) node[circle,draw,fill=white,minimum size=10pt,inner sep=1pt] {};
\draw (120:1) node {$\bullet$};
\draw (120:1) node[above] {$\theta$};
\draw ($(120:1)-(.75,.25)$) node[left] {$a^\vee \in \cP_\theta$};
\draw[dashed,-latex,shorten >=5pt] ($(120:1)-(.75,.25)$) -- ($.5*(180:1)+.5*(45:1)$);
\end{tikzpicture}
\end{center}
As we have seen above, the unit circle is split into arcs by the sleeper directions of the zigzag paths. If $\theta$ sits in the 
interior of one of the arcs $\cP_\theta$ is a perfect matching because the vertices are all located at sleeper directions. 
Two $\theta$ in the same arc will give the same perfect matching.

These matchings play a very important role in the theory of dimers.
They have been introduced in various disguises by Stienstra \cite{stienstra2007hypergeometric}, Gulotta \cite{gulotta2008properly}, Broomhead \cite{broomhead2012dimer} and many others.

All matchings have the same boundary because every node is contained in precisely one edge: 
$$d\cP = \sum_{n \in \rib_0^\bullet} n - \sum_{n \in \rib_0^\circ} n \in \Z\rib_0.$$
Therefore the difference between $2$ perfect matchings is a cycle in $\Z\<\rib_1\>$.

If we start from a reference perfect matching $\cP_0$ we can look at all cohomology classes in $H^1(\TT,\Z)$ that come 
from differences $\cP-\cP_0$, where $\cP$ runs over all perfect matchings. Using the identification $H^1(\TT,\Z)=\Z^2$
these homology classes can be seen as a set of lattice points. 
\begin{definition}
The convex hull of the lattice points $\cP-\cP_0$ viewed in $\Z^2=H^1(\TT,\Z)$ is called the 
\iemph{matching polygon} $\MP(\qpol)$ of the dimer and it is a convex lattice polygon inside $H^1(\TT,\Z)\otimes \R=\R^2$.
Using a different reference perfect matching will translate the matching polygon, therefore we will consider $\MP(\qpol)$ up to affine equivalence.
\end{definition}

If we fix $2$ paths $X,Y$ in the quiver that generate the homology of the torus, we can assign to every perfect matching $\cP$ the lattice point $(\<\bar X,\cP\>,\<\bar Y,\cP\>) \in \Z^2$. The convex hull of these points is
the same as the lattice polygon up to affine equivalence.

\begin{definition}
Every perfect matching $\cP$ is assigned to a lattice point $\cP-\cP_0$ in $\MP(\qpol)$. If the perfect matching sits
on a corner of $\MP(\qpol)$, we call it a \iemph{corner matching}. If it sits on the boundary we call it a \iemph{boundary matching} and if it
sits in the interior we call it an \iemph{internal matching}. 
\end{definition}

\begin{example}
Consider the following dimer quiver and let $X,Y$ be the dashed paths with homology classes $(1,0)$ and $(0,1)$.
\begin{center}
\begin{tikzpicture}
\draw[dotted] (0,0)--(3,0) --(3,3)--(0,3)--(0,0);
\draw[-latex,shorten >=5pt] (0,0)--(1.5,0); 
\draw[-latex,shorten >=5pt] (1.5,0)--(3,0); 
\draw[-latex,shorten >=5pt] (0,3)--(1.5,3); 
\draw[-latex,shorten >=5pt] (1.5,3)--(3,3);
\draw[-latex,shorten >=5pt] (.75,1.5)--(2.25,1.5);
\draw[-latex,shorten >=5pt] (0,1.5)--(.75,1.5);
\draw[shorten >=5pt] (2.25,1.5)--(3,1.5);
\draw[-latex,shorten >=5pt] (1.5,0)--(.75,1.5); 
\draw[-latex,shorten >=5pt] (.75,1.5)--(0,3); 
\draw[-latex,shorten >=5pt] (3,0)--(2.25,1.5); 
\draw[-latex,shorten >=5pt] (2.25,1.5)--(1.5,3); 
\draw[-latex,shorten >=5pt] (1.5,3)--(.75,1.5); 
\draw[-latex,shorten >=5pt] (.75,1.5)--(0,0); 
\draw[-latex,shorten >=5pt] (3,3)--(2.25,1.5); 
\draw[-latex,shorten >=5pt] (2.25,1.5)--(1.5,0); 
\draw[dashed] (0,-.1)--(3,-.1); 
\draw[dashed] (-.1,.1)--(1.32,.1)--(-.18,3.1); 
\draw (0,0) node[circle,draw,fill=white,minimum size=10pt,inner sep=1pt] {{\tiny1}};
\draw (0,3) node[circle,draw,fill=white,minimum size=10pt,inner sep=1pt] {{\tiny1}};
\draw (3,0) node[circle,draw,fill=white,minimum size=10pt,inner sep=1pt] {{\tiny1}};
\draw (3,3) node[circle,draw,fill=white,minimum size=10pt,inner sep=1pt] {{\tiny1}};
\draw (1.5,0) node[circle,draw,fill=white,minimum size=10pt,inner sep=1pt] {{\tiny1}};
\draw (1.5,3) node[circle,draw,fill=white,minimum size=10pt,inner sep=1pt] {{\tiny1}};
\draw (.75,1.5) node[circle,draw,fill=white,minimum size=10pt,inner sep=1pt] {{\tiny1}};
\draw (2.25,1.5) node[circle,draw,fill=white,minimum size=10pt,inner sep=1pt] {{\tiny1}};

\draw[dotted] (9.9,0)--(10.7,0)--(10.7,3.2)--(9.9,3.2)--(9.9,0);
\draw[dotted] (9.9,0)--(7,1.5);
\draw[dotted] (9.9,3.2)--(7,1.5);

\begin{scope}[xshift=6cm,yshift=.5cm]
\draw (0,0) -- (2,1) -- (0,2) -- (0,0);
\draw (0,0) node[circle,draw,fill=white,minimum size=10pt,inner sep=1pt] {{\tiny1}};
\draw (1,1) node[circle,draw,fill=white,minimum size=10pt,inner sep=1pt] {{\tiny4}};
\draw (0,1) node[circle,draw,fill=white,minimum size=10pt,inner sep=1pt] {{\tiny2}};
\draw (0,2) node[circle,draw,fill=white,minimum size=10pt,inner sep=1pt] {{\tiny1}};
\draw (2,1) node[circle,draw,fill=white,minimum size=10pt,inner sep=1pt] {{\tiny1}};
\end{scope}

\begin{scope}[xshift=5cm,yshift=0cm]
\begin{scope}[scale=.2]
\draw[dotted] (0,0)--(3,0) --(3,3)--(0,3)--(0,0);\draw[thick] (1.5,3)--(.75,1.5);
\draw[thick] (3,3)--(2.25,1.5);
\draw[thick] (.75,1.5)--(0,0);
\draw[thick] (2.25,1.5)--(1.5,0);
\end{scope}
\end{scope}

\begin{scope}[xshift=5cm,yshift=2.5cm]
\begin{scope}[scale=.2]
\draw[dotted] (0,0)--(3,0) --(3,3)--(0,3)--(0,0);\draw[thick] (.75,1.5)--(0,3);
\draw[thick] (2.25,1.5)--(1.5,3);
\draw[thick] (3,0)--(2.25,1.5);
\draw[thick] (1.5,0)--(.75,1.5);
\end{scope}
\end{scope}

\begin{scope}[xshift=8.5cm,yshift=1.25cm]
\begin{scope}[scale=.2]
\draw[dotted] (0,0)--(3,0) --(3,3)--(0,3)--(0,0);\draw[thick] (0,0)--(1.5,0);\draw[thick] (0,3)--(1.5,3);
\draw[thick] (1.5,0)--(3,0);\draw[thick] (1.5,3)--(3,3);
\draw[thick] (0,1.5)--(.75,1.5);\draw[thick] (2.25,1.5)--(3,1.5);
\draw[thick] (.75,1.5)--(2.25,1.5);
\end{scope}
\end{scope}

\begin{scope}[xshift=5cm,yshift=1.25cm]
\begin{scope}[scale=.2]
\draw[dotted] (0,0)--(3,0) --(3,3)--(0,3)--(0,0);\draw[thick] (1.5,3)--(.75,1.5);
\draw[thick] (3,3)--(2.25,1.5);
\draw[thick] (3,0)--(2.25,1.5);
\draw[thick] (1.5,0)--(.75,1.5);
\end{scope}
\end{scope}

\begin{scope}[xshift=4.3cm,yshift=1.25cm]
\begin{scope}[scale=.2]
\draw[dotted] (0,0)--(3,0) --(3,3)--(0,3)--(0,0);\draw[thick] (.75,1.5)--(0,3);
\draw[thick] (2.25,1.5)--(1.5,3);
\draw[thick] (.75,1.5)--(0,0);
\draw[thick] (2.25,1.5)--(1.5,0);
\end{scope}
\end{scope}

\begin{scope}[xshift=10cm,yshift=2.5cm]
\begin{scope}[scale=.2]
\draw[dotted] (0,0)--(3,0) --(3,3)--(0,3)--(0,0);\draw[thick] (0,0)--(1.5,0);\draw[thick] (0,3)--(1.5,3);
\draw[thick] (3,3)--(2.25,1.5);
\draw[thick] (3,0)--(2.25,1.5);
\draw[thick] (.75,1.5)--(2.25,1.5);
\end{scope}
\end{scope}

\begin{scope}[xshift=10cm,yshift=1.7cm]
\begin{scope}[scale=.2]
\draw[dotted] (0,0)--(3,0) --(3,3)--(0,3)--(0,0);\draw[thick] (0,0)--(1.5,0);\draw[thick] (0,3)--(1.5,3);
\draw[thick] (2.25,1.5)--(1.5,3);
\draw[thick] (0,1.5)--(.75,1.5);\draw[thick] (2.25,1.5)--(3,1.5);
\draw[thick] (2.25,1.5)--(1.5,0);
\end{scope}
\end{scope}

\begin{scope}[xshift=10cm,yshift=.9cm]
\begin{scope}[scale=.2]
\draw[dotted] (0,0)--(3,0) --(3,3)--(0,3)--(0,0);\draw[thick] (1.5,3)--(.75,1.5);
\draw[thick] (1.5,0)--(3,0);\draw[thick] (1.5,3)--(3,3);
\draw[thick] (0,1.5)--(.75,1.5);\draw[thick] (2.25,1.5)--(3,1.5);
\draw[thick] (1.5,0)--(.75,1.5);
\end{scope}
\end{scope}

\begin{scope}[xshift=10cm,yshift=0.1cm]
\begin{scope}[scale=.2]
\draw[dotted] (0,0)--(3,0) --(3,3)--(0,3)--(0,0);\draw[thick] (.75,1.5)--(0,3);
\draw[thick] (1.5,0)--(3,0);\draw[thick] (1.5,3)--(3,3);
\draw[thick] (.75,1.5)--(0,0);
\draw[thick] (.75,1.5)--(2.25,1.5);
\end{scope}
\end{scope}
\end{tikzpicture}
\end{center}
The dimer quiver is isoradially embedded, there are three corner matchings: 
\begin{itemize}
 \item $\cP_{270^\circ}$ contains all horizontal arrows and is located on $(2,1) \in H^1(\TT)$.
 \item $\cP_{0^\circ}$ contains all upward arrows and is located on $(0,2) \in H^1(\TT)$.
 \item $\cP_{180^\circ}$ contains all downward arrows and is located on $(0,0) \in H^1(\TT)$.
\end{itemize}
In total there are $9$ matchings, $4$ internal matchings, $5$ boundary matchings of which $3$ are corner matchings.
\end{example}

\begin{lemma}[Existence of minimal paths]
Let $\qpol$ be a zigzag consistent dimer quiver. For each homotopy class on the torus there is a path $p$ in $\qpol$ representing it
and a perfect matching $\cP_\theta$ such that $\deg_{\cP_{\theta}}p=0$.
\end{lemma}
\begin{proof}
The proof of this lemma uses ideas from section \ref{Bside} and can be found there as lemma \ref{minimalpaths}. The proof for geometrically consistent dimers was given by Broomhead in \cite{broomhead2012dimer} and adapted to the general setting in \cite{bocklandt2012consistency}.
\end{proof}

\begin{lemma}[Types of matchings]
Let $\qpol$ be a zigzag consistent dimer quiver.
A perfect matching $\cP$ is 
\begin{enumerate}
 \item an internal matching if every nontrivial cyclic path of the quiver meets the matching. 
 \item a boundary matching if there is at least one cyclic path on the quiver that does not meet the matching,
 \item a corner matching if there are two homologically independent cyclic paths on the quiver that do not meet the matching.
\end{enumerate}
\end{lemma}
\begin{proof}
A perfect matching $\cP$ is on the boundary of the matching polygon if there is a homology class
such that $\cP'-\cP$ is positive for all perfect matchings. By the previous lemma we can represent this
homology class by a path $p$, such that there is a $\cP_\theta$ with $\<p,\cP_\theta\>=0$ and therefore also $\<p,\cP\>=0$.
$\cP$ is on a corner if there are two independent homology classes with the same property. 
\end{proof}

\begin{intermezzo}[Lattice polygons and Pick's theorem]
The matching polygon is a lattice polygon and the geometry of lattice polygons has been studied for a long time. In lattice geometry it is more convenient to express
the length and area of objects in terms of \iemph{elementary length} and \iemph{elementary area}. It is the number of elementary building blocks (line segments or triangles with only lattice points on the corners)
in which you can decompose the object. For a line segment $(a,b)-(c,d)$ in $\Z^2$ the elementary length is $\gcd(a-c,b-d)$ and for a polygon the elementary area is twice 
the ordinary area. Both the elementary length and area are invariant under integral affine transformations and they are always integers.

An important theorem in lattice geometry is \iemph{Pick's theorem}. It states that
the elementary area of a lattice polygon is equal to 
\[
A = 2I + B - 2
\]
where $I$ denotes the number of lattice points in the interior and $B$ the number of lattice points on the boundary.
\end{intermezzo}

\begin{theorem}[zigzag and matching polygon]\label{ZPMP}
If $\qpol$ is a zigzag consistent dimer quiver then the zigzag polygon and the matching polygon coincide.
\end{theorem}
This theorem follows from a more precise theorem that describes 
how perfect matchings are distributed on the matching polygon.

\begin{theorem}[Gulotta \cite{gulotta2008properly} Theorems 3.1-3.8]
Let $\qpol$ be isoradially embedded and let $\theta_1,\dots ,\theta_k$ be the sleeper directions of the zigzag paths $\pZ_1,\dots,\pZ_k$.
\begin{enumerate}
 \item The corners of the matching polygon correspond to the matchings $\cP_\theta$.
 \item 
 For every $\theta_i$, the matchings $\cP_{\theta_i-\eps}$ and $\cP_{\theta_i+\eps}$ (with $\eps$ small) are
 the corner matchings of a side in the matching polygon. 
 $(\cP_{\theta_i+\eps} \cup \cP_{\theta_i-\eps}) \setminus (\cP_{\theta_i+\eps} \cap \cP_{\theta_i-\eps})$ is a union of all zigzag paths
 with direction $\theta$. If there are $k$ such zigzag paths the side of the matching polygon has elementary length $k$.
 \item If $n$ is a point on the side of the matching polygon between $\cP_{\theta_i-\eps}$ and $\cP_{\theta_i+\eps}$ and at distance $d$ from 
 $\cP_{\theta_i-\eps}$ then there are precisely $\binom{k}{d}$ perfect matchings on $n$.
 These perfect matchings contain the common edges of $\cP_{\theta_i-\eps}$ and $\cP_{\theta_i+\eps}$, all arrows from $\cP_{\theta_i-\eps}$ that are in $d$ chosen zigzag paths
with direction $\theta_i$ and all arrows from $\cP_{\theta_i+\eps}$ that are in the $k-d$ other zigzag paths with direction $\theta_i$.
\end{enumerate}
\end{theorem}
\begin{proof}
If $\theta \in (\theta_i,\theta_{i+1})$ then $\cP_\theta$ evaluates zero on the two paths running opposite to $\cP_{\theta_i}$ and $\cP_{\theta_{i+1}}$,
so every $\cP_\theta$ is a corner matching.

If $\pZ_i$ is a zigzag path with sleeper direction $\theta = \theta_i$ then $\cP_{\theta+\eps}$ will contain all even arrows of the zigzag path and 
$\cP_{\theta-\eps}$ all odd arrows. Vice versa if $a \in \cP_{\theta+\eps}\setminus \cP_{\theta-\eps}$ then the arrow $b$ preceding $a$ in its
positive cycle will satisfy $b \in \cP_{\theta-\eps}\setminus \cP_{\theta+\eps}$ and $ab$ will be part of a zigzag path with sleeper direction $\theta$.

If $(\cP_{\theta_i+\eps} \cup \cP_{\theta_i-\eps}) \setminus (\cP_{\theta_i+\eps} \cap \cP_{\theta_i-\eps})$ is a union of $k$ zigzag paths.
Then we can make new perfect matchings in the following way: choose a set of $d$ zigzag paths out of these $k$ and let $S$ be the set of arrows in these
zigzag paths.
\[
 \cP_S = (\cP_{\theta_i+\eps} \cap S) \cup (\cP_{\theta_i-\eps} \cap (\qpol_1\setminus  S)) 
\]
If $i=k$ then $\cP_S=\cP_{\theta_i+\eps}$ and if $i=0$ then $\cP_S = \cP_{\theta_i-\eps}$. For an intermediate $i$ this perfect matching
will lie on a lattice point at location $i/k$ between $\cP_{\theta_i+\eps}$ and $\cP_{\theta_i-\eps}$. To see this choose a basis
of the homology that contains a path $p$ that runs in the opposite direction of a zigzag path and one $q$ that is transverse to it using arrows of $\cP_{\theta_i+\eps}$.
The first does not contain any arrows of the zigzag paths and hence 
\[
 \<\bar p, \cP_S\>=\<\bar p, \cP_{\theta_i \pm \eps}\>=0.
\]
The second contains only one of each of these zigzag paths and that arrow is in $\cP_{\theta_i+ \eps}$.
Therefore 
\[
 \<\bar q, \cP_{\theta_i+ \eps}-\cP_{\theta_i- \eps}\>=k \text{ and } \<\bar q, \cP_S-\cP_{\theta_i- \eps}\>=i.
\]
All these perfect matchings evaluate zero on a cycle that runs opposite to these zigzag paths, so they are boundary matchings. 

So we have seen that the boundary of the zigzag polygon coincides with the boundary of the matching polygon. Now if $\cP$ is a perfect matching
on the boundary of the matching polygon then it must evaluate zero on all opposite path to zigzag paths in a certain direction. 
This means that it must contain either all even or all odd arrows from each of these zigzag paths. $\cP$ must hence be of the 
form we constructed above.
\end{proof}

\begin{example}
We illustrate this once more with the suspended pinchpoint \ref{spp}.
\begin{center}
\begin{tikzpicture}
\draw[thick](0,0) -- (0,2) -- (1,0) -- (1,-1) -- (0,0); 
\draw (0,0) node{$\bullet$};
\draw (0,2) node{$\bullet$};
\draw (0,1) node{$\bullet$};
\draw (1,0) node{$\bullet$};
\draw (1,-1) node{$\bullet$};
\draw[-latex] (0,.5)--(-.5,.5);
\draw[-latex] (0,1.5)--(-.5,1.5);
\draw[-latex] (.5,1)--(1.5,1.5);
\draw[-latex] (1,-.5)--(1.5,-.5);
\draw[-latex] (.5,-.5)--(0,-1);
\begin{scope}[xshift=-1.5cm]
\begin{scope}[scale=.25]
\draw [dotted] (0,0) -- (3,0) -- (3,3) -- (0,3);
\draw [dotted] (0,0) -- (3,1) -- (3,2) -- (0,1)--(0,2)--(3,3);
\draw [thick,-latex] (3,0) -- (3,1)--(0,0);
\end{scope}
\end{scope}
\begin{scope}[xshift=-1.5cm,yshift=1cm]
\begin{scope}[scale=.25]
\draw [dotted] (0,0) -- (3,0) -- (3,3) -- (0,3);
\draw [dotted] (0,0) -- (3,1) -- (3,2) -- (0,1)--(0,2)--(3,3);
\draw [thick,-latex] (3,2) -- (3,3)--(0,2);
\end{scope}
\end{scope}
\begin{scope}[xshift=1.75cm,yshift=-1cm]
\begin{scope}[scale=.25]
\draw [dotted] (0,0) -- (3,0) -- (3,3) -- (0,3);
\draw [dotted] (0,0) -- (3,1) -- (3,2) -- (0,1)--(0,2)--(3,3);
\draw [thick,-latex] (0,1) -- (3,2)--(3,1);
\end{scope}
\end{scope}
\begin{scope}[xshift=-1.75cm,yshift=-1.5cm]
\begin{scope}[scale=.25]
\draw [dotted] (0,0) -- (3,0) -- (3,3) -- (0,3);
\draw [dotted] (0,0) -- (3,1) -- (3,2) -- (0,1)--(0,2)--(3,3);
\draw [dotted] (3,0) -- (6,0) -- (6,3) -- (3,3);
\draw [dotted] (3,0) -- (6,1) -- (6,2) -- (3,1)--(3,2)--(6,3);
\draw [thick,-latex] (6,3) -- (3,2)--(3,1)--(0,0)--(3,0);
\end{scope}
\end{scope}
\begin{scope}[xshift=1.75cm,yshift=1cm]
\begin{scope}[scale=.25]
\draw [dotted] (0,0) -- (3,0) -- (3,3) -- (0,3);
\draw [dotted] (0,0) -- (3,1) -- (3,2) -- (0,1)--(0,2)--(3,3);
\draw [dotted] (3,0) -- (6,0) -- (6,3) -- (3,3);
\draw [dotted] (3,0) -- (6,1) -- (6,2) -- (3,1)--(3,2)--(6,3);
\draw [thick,-latex] (0,0) -- (3,0)--(3,1)--(6,2)--(6,3);
\end{scope}
\end{scope}

\begin{scope}[xshift=6cm]
\draw[thick](0,0) -- (0,2) -- (1,0) -- (1,-1) -- (0,0); 
\draw (0,0) node{$\bullet$};
\draw (0,2) node{$\bullet$};
\draw (0,1) node{$\bullet$};
\draw (1,0) node{$\bullet$};
\draw (1,-1) node{$\bullet$};

\begin{scope}[xshift=-1cm,yshift=1.5 cm]
\begin{scope}[scale=.25]
\draw [dotted] (0,0) -- (3,0) -- (3,3) -- (0,3);
\draw [dotted] (0,0) -- (3,1) -- (3,2) -- (0,1)--(0,2)--(3,3);
\draw [thick] (0,0) -- (0,1);
\draw [thick] (3,0) -- (3,1);
\draw [thick] (0,2) -- (0,3);
\draw [thick] (3,2) -- (3,3);
\end{scope}
\end{scope}

\begin{scope}[xshift=-1cm,yshift=-.5cm]
\begin{scope}[scale=.25]
\draw [dotted] (0,0) -- (3,0) -- (3,3) -- (0,3);
\draw [dotted] (0,0) -- (3,1) -- (3,2) -- (0,1)--(0,2)--(3,3);
\draw [thick] (0,0) -- (3,1);
\draw [thick] (0,2) -- (3,3);
\end{scope}
\end{scope}

\begin{scope}[xshift=1.5cm,yshift=-1.5cm]
\begin{scope}[scale=.25]
\draw [dotted] (0,0) -- (3,0) -- (3,3) -- (0,3);
\draw [dotted] (0,0) -- (3,1) -- (3,2) -- (0,1)--(0,2)--(3,3);
\draw [thick] (0,1) -- (0,2);
\draw [thick] (3,1) -- (3,2);
\draw [thick] (0,0) -- (3,0);
\draw [thick] (0,3) -- (3,3);
\end{scope}
\end{scope}

\begin{scope}[xshift=1.5cm,yshift=-.5cm]
\begin{scope}[scale=.25]
\draw [dotted] (0,0) -- (3,0) -- (3,3) -- (0,3);
\draw [dotted] (0,0) -- (3,1) -- (3,2) -- (0,1)--(0,2)--(3,3);
\draw [thick] (0,1) -- (3,2);
\draw [thick] (0,0) -- (3,0);
\draw [thick] (0,3) -- (3,3);
\end{scope}
\end{scope}

\end{scope}

\end{tikzpicture}
\end{center}
\end{example}

\begin{example}
If a dimer on a torus is not consistent, the matching polygon and the zigzag polygon can be different. Here we give an example where
the matching polygon is the unit square and the zigzag polygon is an elementary triangle.
\begin{center}
\begin{tikzpicture}
\draw[dotted] (0,0)--(3,0) --(3,3)--(0,3)--(0,0);
\draw[-latex,shorten >=5pt] (0,0)--(3,0);
\draw[-latex,shorten >=5pt] (0,0)--(0,3);
\draw[-latex,shorten >=5pt] (3,0)--(3,3);
\draw[-latex,shorten >=5pt] (0,3)--(3,3);
\draw[-latex,shorten >=5pt] (0,3)--(1,1);
\draw[-latex,shorten >=5pt] (3,0)--(1,1);
\draw[-latex,shorten >=5pt] (2,2)--(3,0);
\draw[-latex,shorten >=5pt] (2,2)--(0,3);
\draw[-latex,shorten >=5pt] (1,1)--(2,2);
\draw[-latex,shorten >=5pt] (3,3)--(2,2);
\draw[-latex,shorten >=5pt] (1,1)--(0,0);
\draw (0,0) node[circle,draw,fill=white,minimum size=10pt,inner sep=1pt] {{\tiny1}};
\draw (1,1) node[circle,draw,fill=white,minimum size=10pt,inner sep=1pt] {{\tiny2}};
\draw (2,2) node[circle,draw,fill=white,minimum size=10pt,inner sep=1pt] {{\tiny3}};
\draw (0,3) node[circle,draw,fill=white,minimum size=10pt,inner sep=1pt] {{\tiny1}};
\draw (3,0) node[circle,draw,fill=white,minimum size=10pt,inner sep=1pt] {{\tiny1}};
\draw (3,3) node[circle,draw,fill=white,minimum size=10pt,inner sep=1pt] {{\tiny1}};
\draw (-1.5,3) node{$\MP(\qpol)$};
\draw (4.5,3) node{$\ZP(\qpol)$};

\begin{scope}[xshift=-1.8cm,yshift=-.9cm]
\begin{scope}[xshift=6cm,yshift=2.9cm]
\begin{scope}[scale=.2]
\draw[dotted] (0,0)--(3,0) --(3,3)--(0,3)--(0,0);
\draw[thick,-latex] (0,0)--(3,0)--(1,1)--(2,2)--(0,3);
\end{scope}
\end{scope}
\begin{scope}[xshift=6.9cm,yshift=2.1cm]
\begin{scope}[scale=.2]
\draw[dotted] (0,0)--(3,0) --(3,3)--(0,3)--(0,0);
\draw[thick,-latex] (0,0)--(0,3)--(1,1)--(2,2)--(3,0);
\end{scope}
\end{scope}
\begin{scope}[xshift=6cm,yshift=1.6cm]
\begin{scope}[scale=.2]
\draw[dotted] (0,-3)--(-3,-3) --(-3,3)--(0,3)--(0,-3);
\draw[thick,-latex] (0,0)--(0,3)--(-1,2)--(-3,3)--(-2,1)--(-3,0)--(0,0)--(-1,-1)--(0,-3)--(-2,-2)--(-3,-3);
\end{scope}
\end{scope}
\draw[xshift=-2cm,yshift=-.1cm] (8.6,2.2) node {$\bullet$} -- (8,2.8) node {$\bullet$} -- (8.6,2.8) node {$\bullet$} -- (8.6,2.2);
\end{scope}

\begin{scope}[xshift=-8cm,yshift=1cm]
\begin{scope}[xshift=5cm,yshift=1cm]
\begin{scope}[scale=.2]
\draw[dotted] (0,0)--(3,0) --(3,3)--(0,3)--(0,0);
\draw[thick] (0,0)--(0,3) --(2,2);
\draw[thick] (3,3)--(3,0)--(1,1);
\end{scope}
\end{scope}

\begin{scope}[xshift=7cm,yshift=0cm]
\begin{scope}[scale=.2]
\draw[dotted] (0,0)--(3,0) --(3,3)--(0,3)--(0,0);
\draw[thick] (0,0)--(3,0) --(2,2);
\draw[thick] (3,3)--(0,3)--(1,1);
\end{scope}
\end{scope}

\begin{scope}[xshift=7cm,yshift=1cm]
\begin{scope}[scale=.2]
\draw[dotted] (0,0)--(3,0) --(3,3)--(0,3)--(0,0);
\draw[thick] (0,0)--(3,0) --(3,3)--(0,3)--(0,0);
\draw[thick] (1,1)--(2,2);
\end{scope}
\end{scope}

\begin{scope}[xshift=5cm,yshift=0cm]
\begin{scope}[scale=.2]
\draw[dotted] (0,0)--(3,0) --(3,3)--(0,3)--(0,0);
\draw[thick] (0,0)--(3,3);
\end{scope}
\end{scope}

\begin{scope}[xshift=4.6cm,yshift=-.8cm]
\begin{scope}[scale=.2]
\draw[dotted] (0,0)--(3,0) --(3,3)--(0,3)--(0,0);
\draw[thick] (0,0)--(1,1);
\draw[thick] (3,0)--(2,2)--(0,3);
\end{scope}
\end{scope}

\begin{scope}[xshift=5.4cm,yshift=-.8cm]
\begin{scope}[scale=.2]
\draw[dotted] (0,0)--(3,0) --(3,3)--(0,3)--(0,0);
\draw[thick] (3,3)--(2,2);
\draw[thick] (3,0)--(1,1)--(0,3);
\end{scope}
\end{scope}
\draw[xshift=-2cm,yshift=-.1cm] (8.6,.6) node {$\bullet$} -- (8,.6) node {$\bullet$} -- (8,1.2) node {$\bullet$} -- (8.6,1.2) node {$\bullet$} -- (8.6,.6);
\end{scope}
\end{tikzpicture}
\end{center}
\end{example}

We end this section with an interpretation of the geometry of the matching polygon.
\begin{theorem}[matching polygon versus dimer]\label{mirdim}
Let $\qpol$ be a zigzag consistent dimer quiver on a torus.
\begin{enumerate}
 \item The number of vertices of $\qpol$ equals the elementary area of the polygon $\MP(\qpol)$.
 \item The number of zigzag paths of $\qpol$ equals the elementary perimeter of the polygon $\MP(\qpol)$.
 \item The number of vertices of $\polq$ equals the number of boundary lattice points of the polygon $\MP(\qpol)$.
 \item The number of genus of $\ful{\polq}$ equals the number of internal lattice points of the polygon $\MP(\qpol)$.
\end{enumerate}
\end{theorem}
\begin{proof}
The second follows immediately from theorem \ref{ZPMP} and the third because the zigzag paths of $\qpol$ correspond to the vertices of $\polq$.
For the first one we calculate the area of the zigzag polygon. The corners of the polygon are $c_1,\dots,c_k$ with $c_i = \vec v_1^\perp+\dots+\vec v_i^\perp$ and the elementary
area is
\[
A=  |c_1\times c_2 + c_2\times c_3 +\dots + c_k \times c_1| = \left|\sum_{1\le i>j\le k} \vec v_i\times \vec v_j\right|. 
\]
where we considered the vectors inside $\R^2\subset \R^3$ to make sense of the cross product.
Now $\vec v_i\times \vec v_j$ equal the intersection number between $\pZ_i$ and $\pZ_j$, which by lemma \ref{intersection} 
is the number of arrows from vertex $i$ to vertex $j$ minus the number of arrows from vertex $j$ to vertex $i$ the twisted dimer quiver.
So the area is a signed count of the arrows. An arrow $i \to j$ contributes $+1$ if  $i>j$ and $-1$ if $i<j$.

Choose an isoradial embedding and let $\cP_\theta$ be the perfect matching with $\theta \in (\theta_{\pZ_k},\theta_{\pZ_1})$, then
the arrows that contribute negatively are precisely those in $\cP_\theta$. So 
\[
 A = \# {a \not\in \cP_\theta} - \#\{a \in \cP_\theta\}= \# \qpol_1 - 2 \#\{a \in \cP_\theta\} = \# \qpol_1 - \# \qpol_2 = \# \qpol_0.
\]

Finally using Pick's theorem we get
\se{
 g &= 1- \frac{1}2\chi(\polq) = 1 -\frac 12(\#\polq_0 -\#\polq_1 +\#\polq_2)\\&= 1 -\frac 12(\#\polq_0 -\#\qpol_0)= 1 - \frac 12 (B - A) = I,
}
which proves the fourth statement.
\end{proof}
\begin{example}\label{sppmir}
For the suspended pinchpoint we see that the area of the matching polygon is $3$, its perimeter is $5$ and it has no internal lattice points.
The dimer $\qpol$ sits in a torus and has $3$ vertices, while $\polq$ sits in a sphere and has $5$ vertices.
\begin{center}
\begin{tikzpicture}
\begin{scope}[scale=1]
\draw (1.5,3.5) node{$\qpol$};
\draw [-latex,shorten >=5pt] (0,0) -- (3,0);
\draw [-latex,shorten >=5pt] (3,0) -- (3,1);
\draw [-latex,shorten >=5pt] (3,1) -- (0,0);
\draw [-latex,shorten >=5pt] (0,0) -- (0,1);
\draw [-latex,shorten >=5pt] (0,1) -- (3,2);
\draw [-latex,shorten >=5pt] (3,2) -- (3,3);
\draw [-latex,shorten >=5pt] (3,2) -- (3,1);
\draw [-latex,shorten >=5pt] (0,2) -- (0,1);
\draw [-latex,shorten >=5pt] (0,3) -- (3,3);
\draw [-latex,shorten >=5pt] (3,3) -- (0,2);
\draw [-latex,shorten >=5pt] (0,2) -- (0,3);
\draw (0,0) node[circle,draw,fill=white,minimum size=10pt,inner sep=1pt] {{\tiny1}};
\draw (0,1) node[circle,draw,fill=white,minimum size=10pt,inner sep=1pt] {{\tiny2}};
\draw (0,2) node[circle,draw,fill=white,minimum size=10pt,inner sep=1pt] {{\tiny3}};
\draw (0,3) node[circle,draw,fill=white,minimum size=10pt,inner sep=1pt] {{\tiny1}};
\draw (3,0) node[circle,draw,fill=white,minimum size=10pt,inner sep=1pt] {{\tiny1}};
\draw (3,1) node[circle,draw,fill=white,minimum size=10pt,inner sep=1pt] {{\tiny2}};
\draw (3,2) node[circle,draw,fill=white,minimum size=10pt,inner sep=1pt] {{\tiny3}};
\draw (3,3) node[circle,draw,fill=white,minimum size=10pt,inner sep=1pt] {{\tiny1}};
\end{scope}

\begin{scope}[xshift=5.5cm,yshift=1.5cm]
\begin{scope}[scale=.66]
\draw[thick](0,0) -- (0,2) -- (1,0) -- (1,-1) -- (0,0); 
\draw (0,0) node{$\bullet$};
\draw (0,2) node{$\bullet$};
\draw (0,1) node{$\bullet$};
\draw (1,0) node{$\bullet$};
\draw (1,-1) node{$\bullet$};
\draw[-latex] (0,.5)--(-.5,.5);
\draw[-latex] (0,1.5)--(-.5,1.5);
\draw[-latex] (.5,1)--(1.5,1.5);
\draw[-latex] (1,-.5)--(1.5,-.5);
\draw[-latex] (.5,-.5)--(0,-1);
\begin{scope}[xshift=-1.5cm]
\begin{scope}[scale=.25]
\draw (-1.5,1.5) node {{\small c}};
\draw [dotted] (0,0) -- (3,0) -- (3,3) -- (0,3);
\draw [dotted] (0,0) -- (3,1) -- (3,2) -- (0,1)--(0,2)--(3,3);
\draw [thick,-latex] (3,0) -- (3,1)--(0,0);
\end{scope}
\end{scope}
\begin{scope}[xshift=-1.5cm,yshift=1cm]
\begin{scope}[scale=.25]
\draw (-1.5,1.5) node {{\small b}};
\draw [dotted] (0,0) -- (3,0) -- (3,3) -- (0,3);
\draw [dotted] (0,0) -- (3,1) -- (3,2) -- (0,1)--(0,2)--(3,3);
\draw [thick,-latex] (3,2) -- (3,3)--(0,2);
\end{scope}
\end{scope}
\begin{scope}[xshift=1.75cm,yshift=-1cm]
\begin{scope}[scale=.25]
\draw (4.5,1.5) node {{\small e}};
\draw [dotted] (0,0) -- (3,0) -- (3,3) -- (0,3);
\draw [dotted] (0,0) -- (3,1) -- (3,2) -- (0,1)--(0,2)--(3,3);
\draw [thick,-latex] (0,1) -- (3,2)--(3,1);
\end{scope}
\end{scope}
\begin{scope}[xshift=-1.75cm,yshift=-1.5cm]
\begin{scope}[scale=.25]
\draw (-1.5,1.5) node {{\small d}};
\draw [dotted] (0,0) -- (3,0) -- (3,3) -- (0,3);
\draw [dotted] (0,0) -- (3,1) -- (3,2) -- (0,1)--(0,2)--(3,3);
\draw [dotted] (3,0) -- (6,0) -- (6,3) -- (3,3);
\draw [dotted] (3,0) -- (6,1) -- (6,2) -- (3,1)--(3,2)--(6,3);
\draw [thick,-latex] (6,3) -- (3,2)--(3,1)--(0,0)--(3,0);
\end{scope}
\end{scope}
\begin{scope}[xshift=1.75cm,yshift=1cm]
\begin{scope}[scale=.25]
\draw (7.5,1.5) node {{\small a}};
\draw [dotted] (0,0) -- (3,0) -- (3,3) -- (0,3);
\draw [dotted] (0,0) -- (3,1) -- (3,2) -- (0,1)--(0,2)--(3,3);
\draw [dotted] (3,0) -- (6,0) -- (6,3) -- (3,3);
\draw [dotted] (3,0) -- (6,1) -- (6,2) -- (3,1)--(3,2)--(6,3);
\draw [thick,-latex] (0,0) -- (3,0)--(3,1)--(6,2)--(6,3);
\end{scope}
\end{scope}
\end{scope}
\end{scope}

\begin{scope}[xshift=9cm]
\begin{scope}[scale=1]
\draw (1.5,3.5) node{$\polq$};
\draw [-latex,shorten >=5pt] (0,.5) -- (3,-.5);
\draw [-latex,shorten >=5pt] (3,-.5) -- (3,.5);
\draw [-latex,shorten >=5pt] (3,.5) -- (0,.5);
\draw [-latex,shorten >=5pt] (0,.5) -- (0,1.5);
\draw [-latex,shorten >=5pt] (0,2.5) -- (0,1.5);
\draw [-latex,shorten >=5pt] (3,1.5) -- (3,.5);
\draw [-latex,shorten >=5pt] (3,1.5) -- (3,2.5);
\draw [-latex,shorten >=5pt] (3,1.5) -- (3,.5);
\draw [-latex,shorten >=5pt] (3,1.5) -- (3,2.5);
\draw [-latex,shorten >=5pt] (0,1.5) -- (3,1.5);
\draw [-latex,shorten >=5pt] (3,2.5) -- (0,2.5);
\draw [-latex,shorten >=5pt] (0,2.5) -- (3,3.5);
\draw [-latex,shorten >=5pt] (3,3.5) -- (3,2.5);
\draw (0,.5) node[circle,draw,fill=white,minimum size=10pt,inner sep=1pt] {{\tiny d}};
\draw (0,2.5) node[circle,draw,fill=white,minimum size=10pt,inner sep=1pt] {{\tiny d}};
\draw (0,1.5) node[circle,draw,fill=white,minimum size=10pt,inner sep=1pt] {{\tiny e}};
\draw (3,.5) node[circle,draw,fill=white,minimum size=10pt,inner sep=1pt] {{\tiny c}};
\draw (3,2.5) node[circle,draw,fill=white,minimum size=10pt,inner sep=1pt] {{\tiny b}};
\draw (3,1.5) node[circle,draw,fill=white,minimum size=10pt,inner sep=1pt] {{\tiny a}};
\draw (3,-.5) node[circle,draw,fill=white,minimum size=10pt,inner sep=1pt] {{\tiny a}};
\draw (3,3.5) node[circle,draw,fill=white,minimum size=10pt,inner sep=1pt] {{\tiny a}};
\end{scope}
\end{scope}
\end{tikzpicture}
\end{center}
We indicated the vertices in $\polq$ by the same letter as their corresponding zigzag paths. 
\end{example}

\subsubsection{Dimer moves}

To end this section we will look a bit closer at the dimer moves. We already know that the join move and the spider move both preserve zigzag consistency
because they come from $2\tot 2$-moves in the triple crossing diagrams. But we can also use these moves to relate the different notions we introduced
for dimers on a torus.

\begin{lemma}
If $\qpol$ and $\qpol'$ are dimer quivers related by a split move then
\begin{enumerate}
 \item There is a natural bijection between the vertices of $\qpol$ and $\qpol'$.
 \item There is a natural bijection between the zigzag paths that preserves the zigzag polygon.
 \item
 There is a natural bijection between the isoradial embeddings.
 \item
 There is a natural bijection between all {\bf perfect matchings} that preserves the matching polygon.
 \end{enumerate}
\end{lemma}
\begin{proof}
The first statement follows from the fact that the split adds just a two-cycle between already existing vertices.

Let $a_1\dots a_k$ be the cycle that is split in $a_1\dots a_ib_1$ and $b_2a_{i+1}\dots a_k$.
The bijection between the zigzag paths keeps all zigzag paths the same except those containing $a_{i}a_{i+1}$ and $a_ka_1$.
In these zigzag paths we insert an extra path of length two: $a_ia_{i+1} \to a_{i}b_1b_2a_{i+1}$ and $a_ka_1 \to a_kb_2b_1a_1$. 

Given an isoradial embedding of $\qpol$ we get one of $\qpol'$ by drawing $b_1$ and $b_2$ on the corresponding diagonal in $a_1\dots a_k$.
It is clear that $a_1\dots a_k$ is inscribed in a unit circle if and only if $a_1\dots a_ib_1$ and $b_2a_{i+1}\dots a_k$ are inscribed in
a unit circle.

Given a perfect matching $\cP$ for $\qpol$ we can obtain one for $\qpol'$ by adding $b_1$ if $\{a_{1},\dots, a_i\}\cap \cP=\emptyset$ and
adding $b_2$ if $\{a_{i+1},\dots, a_k\}\cap \cP=\emptyset$. The location of the perfect matching does not change because
we can use the same paths $X,Y$ (which do not contain $b_1,b_2$) to get the coordinates.
\end{proof}

\begin{lemma}\label{mutpres}
If $\qpol$ and $\qpol'$ are consistent dimer quivers related by a spider move then
\begin{enumerate}
 \item There is a natural bijection between the vertices of $\qpol$ and $\qpol'$.
 \item There is a natural bijection between the zigzag paths that preserves the zigzag polygon.
 \item
 There is a natural bijection between the isoradial embeddings.
 \item
 There is a natural bijection between the {\bf corner matchings} that preserves the matching polygon.
 \end{enumerate}
Remark: there is NO natural bijection between all perfect matchings.
\end{lemma}
\begin{proof}
The first and the second can be seen pictorially
\begin{center}
\begin{tikzpicture}
\begin{scope}[scale=.5]
\begin{scope}
\draw[dotted] (1.5,0) arc (0:360:1.5);
\draw [-latex,shorten >= 2pt] (-1.06,-1.06) -- (0,0);
\draw [-latex,shorten >= 2pt] (1.06,1.06) -- (0,0);
\draw [-latex,shorten >= 2pt] (0,0)-- (-1.06,1.06) ;
\draw [-latex,shorten >= 2pt] (0,0) --(1.06,-1.06) ;
\draw [-latex,shorten >= 2pt] (1.06,-1.06) -- (1.06,1.06);
\draw [-latex,shorten >= 2pt] (-1.06,1.06) -- (-1.06,-1.06);
\draw [-latex,dashed] (1.5,-.5) arc (20:160:1.6);
\draw [-latex,dashed] (-1.5,.5) arc (200:340:1.6);
\draw [-latex,dashed] (.5,1.5)--(.5,-1.5);
\draw [-latex,dashed] (-.5,-1.5)--(-.5,1.5);
\draw (-1.06,-1.06) node[shape=circle,fill=gray!0,inner sep=0pt] {$\circ$};
\draw (-1.06,1.06) node[shape=circle,fill=gray!0,inner sep=0pt] {$\circ$};
\draw (1.06,-1.06) node[shape=circle,fill=gray!0,inner sep=0pt] {$\circ$};
\draw (1.06,1.06) node[shape=circle,fill=gray!0,inner sep=0pt] {$\circ$};
\draw (0,0) node[shape=circle,fill=gray!0,inner sep=0pt] {$\circ$};
\end{scope}
\draw (2.5,0) node {$\tot$};
\begin{scope}[xshift=5cm]
\draw[dotted] (1.5,0) arc (0:360:1.5);
\draw [-latex,shorten >= 2pt] (-1.06,1.06) -- (0,0);
\draw [-latex,shorten >= 2pt] (1.06,-1.06) -- (0,0);
\draw [-latex,shorten >= 2pt] (0,0)-- (-1.06,-1.06) ;
\draw [-latex,shorten >= 2pt] (0,0) --(1.06,1.06) ;
\draw [-latex,shorten >= 2pt] (1.06,1.06) -- (-1.06,1.06);
\draw [-latex,shorten >= 2pt] (-1.06,-1.06) -- (1.06,-1.06);
\draw (-1.06,-1.06) node[shape=circle,fill=gray!0,inner sep=0pt] {$\circ$};
\draw (-1.06,1.06) node[shape=circle,fill=gray!0,inner sep=0pt] {$\circ$};
\draw (1.06,-1.06) node[shape=circle,fill=gray!0,inner sep=0pt] {$\circ$};
\draw (1.06,1.06) node[shape=circle,fill=gray!0,inner sep=0pt] {$\circ$};
\draw (0,0) node[shape=circle,fill=gray!0,inner sep=0pt] {$\circ$};
\begin{scope}[rotate=90]
\draw [-latex,dashed] (1.5,-.5) arc (20:160:1.6);
\draw [-latex,dashed] (-1.5,.5) arc (200:340:1.6);
\draw [-latex,dashed] (.5,1.5)--(.5,-1.5);
\draw [-latex,dashed] (-.5,-1.5)--(-.5,1.5);
\end{scope}
\end{scope}
\end{scope}
\end{tikzpicture}
\end{center}
Up to translation an isoradial embedding is given by assigning sleeper directions to the zigzag paths.
As there is a bijection between the zigzag paths we can use this bijection to get sleeper directions for the new dimer.

Using an isoradial embedding we can identify a corner matching $\cP_\theta$ with its corresponding corner matching $\cP'_\theta$.
This is a bijection because if $\qpol$ and $\qpol'$ are consistent all corner matchings are of this form. This bijection 
preserves the matching polygon because in the consistent case the matching polygon equals the zigzag polygon.
\end{proof}

\renewcommand{\thesection}{\Alph{section}} 
\section{The A-side: integrable systems and statistical physics}\label{Aside}

In this section we will review work of Goncharov and Kenyon \cite{goncharov2013dimers}, and Kenyon, Okounkov and Sheffield \cite{kenyon2006dimers,kenyon2006planar,sheffield2006random} that relates dimers with dynamical systems and statistical
physics. 

The main idea is that we will associate to each consistent dimer on a torus a pair of Poisson varieties $\cL(\rib)\to \cX(\rib)$, 
which we can glue together using quiver mutation. Thus we obtain a Poisson bundle $\cL\to\cX$ associated to the mutation class of the dimer. 
The Poisson variety $\cL$ is an integrable system and we can find an explicit maximal set
of commuting Hamiltonians. The zero locus of the sum of these Hamiltonians defines a hypersurface in $\cL$, which is a bundle of curves over $\cX$.
These curves have an interpretation in terms of the phase diagrams of a two-dimensional model in statistical physics. 

\subsection{A Poisson structure on a dimer}
We start with a consistent dimer on a torus $(\rib,\qpol)$ and we will denote the torus by $\TT=\ful{\rib}$. The mirror dimer $(\bir,\polq)$ will be in general not be a dimer
on a torus but on a surface with a different genus, which we denote by $\Surf=\ful{\bir}$.

In \cite{goncharov2013dimers} Goncharov and Kenyon study line bundles with a connection on the graph $\rib$. They put a one dimensional complex vector space on each node and to
each edge they assign an invertible map between the vector spaces of the nodes incident to it. Base change in the nodes will
transform a line bundle into an isomorphic one, so the space of line bundles $\cL(\rib)$ can be described as a quotient.
On the space $\cM(\rib):=\Maps(\rib_1,\C^*)$ we have an action of $\cG(\rib):=\Maps(\rib_0,\C^*)$ by conjugation:
$$g \cdot f: e \mapsto g_{b(e)}f(e)g_{w(e)}^{-1}.$$ 
The quotient 
\[
\cL(\rib) := \frac{\cM(\rib)}{\cG(\rib)} \cong \C^{*\# \rib_1- \#\rib_0 +1}
\]
is a $\# \rib_1- \#\rib_0 +1=\#\rib_2 + 1$-dimensional complex torus because the action has a one-dimensional kernel.

The ring of coordinates $\C[\cL(\rib)]$  is the subring of
\[
\C[\cM(\rib)]:=\C[X_e, X_e^{-1}| e \in \rib_0]
\]
generated by all cyclic walks in the graph. In this notation we assume $X_e$ represents the map from the white to the black vertex and $X_e^{-1}$ represents
the inverse map. Given an oriented walk $p=e_1\dots e_k$  in $\rib$ the product $W_p =\prod_i X_{e_i}^{\pm 1}$ (with the correct exponent for the orientation)
can be interpreted as the holonomy of the connection along $p=e_1\dots e_k$. This gives
\[
\C[\cL(\rib)] = \C[W_p|p \text{ is a cyclic walk in $\rib$}].
\]

The coordinate ring $\C[\cL(\rib)]$ can be interpreted as the group algebra of the first homology of the graph $\C[\Lambda]$ with $\Lambda=\H_1(\rib,\Z)=\H_1(\bir,\Z)$.
There is a natural map from $\Lambda$ to $\H_1(\Surf,\Z)$ coming from the embedding of the ribbon graph $\bir$ in the closed surface $\Surf$. 
The space $\H_1(\Surf,\Z)$ has an antisymmetric pairing coming from the intersection theory of curves on a surface and we can pull this back
to a pairing $\eps(-,-)$ on $\Lambda$. Note that this pairing is degenerate because $\Lambda$ has elements that are contractible curves in $\Surf$ (the faces of $\bir$, or equivalently
the zigzag walks in $\rib$). We can promote this pairing to a Poisson algebra structure.

\begin{definition}
The Poisson structure on 
$\C[\cL(\rib)]=\C[\Lambda]$ is defined by
\[
 \{W_p,W_q\} = \eps(p,q) W_pW_q,
\]
where $\eps(p,q)$ is the intersection number between $p$ and $q$ in $\Surf$. 
\end{definition}

The Poisson algebra $\C[\cL(\rib)]$ has two natural sub-Poisson algebras, the one generated by the face walks, 
$\C[W_F^{\pm 1}|F \in \rib_2]$, and the one generated by the zigzag walks $\C[W_Z^{\pm 1}|Z \in \twist{\rib}_2]$.
The former can be seen as the coordinate ring of a $\#\rib_2-1$-dimensional torus, which we will denote by $\cX(\rib)$:
\[
\C[\cX(\rib)] = \C[W_F^{\pm 1}|F \in \rib_2].
\]
The embedding $\C[\cX(\rib)] \subset \C[\cL(\rib)]$ gives a projection $\pi: \cL(\rib) \mapsto \cX(\rib)$.
The fibers of the map $\pi: \cX(\rib) \mapsto \cL(\rib)$ are $2$-dimensional 
complex tori because
\[
 \dim \cL(\rib)-\dim \cX(\rib)= \#\rib_2 + 1-(\#\rib_2 - 1)=2.
\]

The latter has dimension $\#\bir_2-1$ and it can be seen as the Poisson center of $\C[\cL(\rib)]$. Recall that the \iemph{Poisson center} consists of all
elements $z$ for which the bracket $\{z,-\}$ is zero. These elements are also called \iemph{Casimirs}.

\begin{lemma}
The Poisson center of $\C[\cL(\rib)]$ is generated by the holonomies of the zigzag walks in $\rib$:
\[
Z_P(\C[\cL(\rib)]) = \C[W_Z^{\pm 1}|Z \in \bir_2].
\]
\end{lemma}
\begin{proof}
If $W_p$ sits in the Poisson center then $p$ must have trivial intersection number with all curves in $\Surf$. This means that 
$p$, viewed as a curve in $\Surf$ has trivial homology. So $\bar p$ is a combination of faces in $\bir$, which are precisely the zigzag paths in $\rib$.
\end{proof}

Because the generators of $\C[\cL(\rib)]$ are of the form $W_p$ where $p$ is
a walk in $\rib$ or $\bir$, we can look at their homology degrees in $\ful{\rib}=\TT$ and $\ful{\bir}=\Surf$. This gives us two gradings on $\C[\cL(\rib)]$:
\[
 \deg_\TT W_p = \bar p \in \H_1(\TT)=\Z^2 \text{ and } \deg_\Surf W_p = \bar p \in \H_1(\Surf)=\Z^{2g}.
\]
Because the face walks are contractible in $\TT$ and the zigzag walk are contractible in $\Surf$ we have that
\[
 \underset{\deg_\TT=0}{\C[\cL(\rib)]}=\C[W_F^{\pm 1}|F \in \rib_2]=\C[\cX(\rib)] \text{ and } \underset{\deg_\Surf=0}{\C[\cL(\rib)]}=\C[W_Z^{\pm 1}|Z \in \bir_2]=Z_P(\C[\cL(\rib)]).
\]
The $\deg_\TT$-grading gives us an action of the $2$-dimensional torus $\cT=\Hom(\H_1(\TT),\C^*)={\C^*}^2$ on $\cL(\rib)$ such that
for all homogeneous $f \in \C[\cL(\rib)]$ and $t \in \cT$ we have $f(t\cdot x) := t^{\deg_\TT f} f(x)$. If we quotient out the $\cT$-action
we get $\cX(\rib)=\cL(\rib)/\cT$ and the map $\pi: \cL(\rib) \to \cX(\rib)$ is the quotient map.

\begin{lemma}
If $\rib$ and $\rib'$ are related by a split move, there is a natural map
$\C[\rib] \to \C[\rib']$ that induces a Poisson isomorphism
\[
 \C[\cL(\rib)] \to \C[\cL(\rib')].
\]
This map is both $\deg_\TT$- and $\deg_\Surf$-graded.
\end{lemma}
\begin{proof}
Let $b_1,b_2$ be the two new edges in $\rib'$. The map
$$
\psi:\C[\rib] \to \C[\rib']: X_e \mapsto \begin{cases}
                                          X_{b_1}^{-1}X_e &\text{if $b_1$ and $e$ share a node in $\rib'$,}\\
                                          X_{b_2}^{-1}X_e &\text{if $b_2$ and $e$ share a node in $\rib'$,}\\
                                          X_e&\text{otherwise,}
                                         \end{cases}
$$
does the trick because it maps face walks to face walks and zigzag walks to zigzag walks. 
\end{proof}
This lemma enables us to restrict our attention to reduced dimers.

\subsection{Cluster algebras and cluster Poisson varieties}

In this section we will quickly introduce cluster algebras, which were first described by Fomin and Zelevinski in \cite{fomin2002cluster,fomin2003cluster}, and the dual version  
developed by Fock and Goncharov in \cite{fock2003cluster,fock2009cluster}, cluster Poisson varieties.

We start with an abelian group $\Lambda$ with an antisymmetric linear form $\<,\>: \Lambda \times \Lambda \to \Z$
and a basis $\{v_1,\dots, v_n\}$. To this we can associate a quiver $Q$ with vertices $v_1,\dots,v_n$ and
if $\eps_{ij}=\<v_i,v_j\>>0$ we draw $\eps_{ij}$ arrows from $j$ to $i$. 

The pair $(\Lambda, \{v_1,\dots, v_n\})$ is called a \iemph{seed} and up to isomorphism a seed is described by its quiver.

\begin{definition}\label{qmut}\label{qmutsec}
Given a quiver $Q$ without loops or $2$-cycles, we define the mutation of the quiver at vertex $v \in Q_0$ as the new quiver obtained by the following 
steps.
\begin{enumerate}
 \item Reverse all arrows in $v$.
 \item Add new arrows $m_{ab}$ for every path $ab$ of length $2$ in $Q$ that runs through $v$: $t(a)=h(b)=v$.
 \item Delete all $2$-cycles.
\end{enumerate}
\[
Q = \vcenter{\xymatrix@=.5cm{
 \vtx{}\ar@2[dr]&&\vtx{}\ar[dd]\\
&\vtx{v}\ar[ur]\ar@2[dl]&\\
 \vtx{}\ar[uu]&&\vtx{}\ar[ul]
 }}
\to
\mu_vQ = \vcenter{\xymatrix@=.5cm{
 \vtx{}\ar@3[dd]\ar@2[rr]&&\vtx{}\ar[dl]\\
&\vtx{v}\ar@2[ul]\ar[dr]&\\
 \vtx{}\ar@2[ur]&&\vtx{}\ar@2[ll]
 }}
\]
\end{definition}

We can also represent mutation on the level of the seed. 
To mutate at $v_i$ we substitute $v_i$ by the new basis element $v_i' = -v_i +\sum_{j}[\eps_{ij}]_+ v_j$ where $[a]_+$ stands for $\max(a,0)$.

Starting from a seed with quiver $Q$, we consider the function field $\C(x_1,\dots,x_n)$  and
look at the set of variables $\{x_1,\dots,x_n\}$. This set is called a cluster and if we mutate $Q$ 
at $v=v_i$ we can associate a new cluster to $\mu_v Q$ by sustituting $x_i$ for 
\[
 x_i' = \left(\prod_{j| \eps_{ij}>0}x_j^{\eps_{ij}} + \prod_{j|\eps_{ij}<0}x_j^{-\eps_{ij}}\right){x_i^{-1}}
\]
while keeping all other variables the same.
The algebra generated by all clusters that can be obtained by all possible sequences of mutation starting from $Q$
is called the cluster algebra and denoted by $A_Q$.

To this construction there exists a dual version, which constructs a Poisson variety $\mathbb{X}_Q$.
The starting point is the dual abelian group $\Lambda^\vee= \Hom(\Lambda,\Z)$ with dual basis
$\{v_1^*,\dots, v_n^*\}$. The mutation will result in a new dual basis, 
$\{{v_1'}^*,\dots, {v_n'}^*\}$ with
\[
{w'}^* \mapsto \begin{cases}
                                                         -v^* &v=w\\
                                                         w^* + [\eps_{vw}]_+ v^*&v \ne w 
\end{cases}
                                                         \]
Note that now the new basis differs in more than one place from the old.
                                                         
The algebra $\C[\Lambda^\vee_Q]\cong \C[X_v|v \in Q_0]$ comes with a Poisson bracket 
\[
 \{X_v,X_w\}= \eps_{vw}X_vX_w
\]
and this gives us a Poisson variety $\mathbb X_Q\cong \C^{*n}$.

We can upgrade the isomorphism between $\Lambda^\vee_Q$ and $\Lambda^\vee_{\mu_v Q}$ to a
birational map between the Poisson varieties:
\[
\phi: \C(\mathbb{X}_{\mu_v Q}) \to \C(\mathbb{X}_{Q}): X_{w'} \mapsto
\begin{cases}
X_{v}^{-1} &w=v\\
X_{w}(1+ X_{v}^{\mathrm{sgn}(\eps_{vw})})^{\eps_{vw}} &w\ne v.
                                                        \end{cases}\]
We can use this birational map to glue $\mathbb{X}_Q$ and $\mathbb{X}_{\mu_vQ}$ together.
If we do this for all possible mutations we get a Poisson variety $\mathbb X$.

\subsection{Gluing dimer Poisson varieties}

If we go back to the dimer case we see that the Poisson variety $\cX(\rib)$ that we defined looks very much like $\mathbb{X}_Q$.
The generators $W_F$ correspond to the faces of $\rib$ and hence to the vertices of $\qpol$. The intersection form between
the face cycles is the same as the bilinear form coming from the quiver in the cluster algebra setting.
The only difference is that unlike in the cluster setting the $W_F$ satisfy an extra relation: $\prod_F W_F=1$. This means
that $\cX(\rib)$ is a hypersurface in $\mathbb{X}_{\qpol}$.
In the previous section we saw that we could glue the varieties $\mathbb{X}_{\qpol}$ together to a cluster variety. In this section we will do a similar thing.
We will glue the projections $\cL(\rib) \to \cX(\rib)$ together to obtain a bundle $\cL \to \cX$.

In the dimer graph picture we mutate a face $F$ by means of a spider move. This spider move
only affects $F$ and its four adjacent faces  $F_1,\dots, F_4$. 
\begin{center}
\begin{tikzpicture}
\begin{scope}[scale=.66]
\begin{scope}
\draw (0,2.3) node {$\rib$};
\draw[dotted] (1.5,0) arc (0:360:1.5);
\draw (-1.5,0) to node [rectangle,fill=white,inner sep=1pt] {{\tiny f}} (-.5,0) to node [rectangle,fill=white,inner sep=1pt] {{\tiny b}} (0,1.5) to node [rectangle,fill=white,inner sep=1pt] {{\tiny a}} (.5,0);
\draw[rotate=180] (-1.5,0) to node [rectangle,fill=white,inner sep=1pt] {{\tiny e}} (-.5,0) to node [rectangle,fill=white,inner sep=1pt] {{\tiny d}} (0,1.5) to node [rectangle,fill=white,inner sep=1pt] {{\tiny c}} (.5,0);
\draw (-1.5,0) node[circle,draw,fill=white,inner sep=1pt] {{\tiny 3}};
\draw (1.5,0) node[circle,draw,fill=white,inner sep=1pt] {{\tiny 1}};
\draw (0,-1.5) node[circle,draw,fill=white,inner sep=1pt] {{\tiny 4}};
\draw (0,1.5) node[circle,draw,fill=white,inner sep=1pt] {{\tiny 2}};
\draw (-.5,0) node[shape=circle,fill=gray!0,inner sep=0pt] {$\bullet$};
\draw (.5,0) node[shape=circle,fill=gray!0,inner sep=0pt] {$\bullet$};
\draw (0,0) node[shape=circle,fill=gray!0,inner sep=0pt] {$\scriptstyle{F}$};
\draw (45:1) node[shape=circle,fill=gray!0,inner sep=0pt] {$\scriptstyle{F_1}$};
\draw (135:1) node[shape=circle,fill=gray!0,inner sep=0pt] {$\scriptstyle{F_2}$};
\draw (225:1) node[shape=circle,fill=gray!0,inner sep=0pt] {$\scriptstyle{F_3}$};
\draw (315:1) node[shape=circle,fill=gray!0,inner sep=0pt] {$\scriptstyle{F_4}$};
\end{scope}
\draw (2.5,0) node {$\tot$};
\begin{scope}[xshift=5cm]
\draw (0,2.3) node {$\mu_F\rib$};
\draw[dotted] (1.5,0) arc (0:360:1.5);
\draw[rotate=90] (-1.5,0) to node [rectangle,fill=white,inner sep=1pt] {{\tiny F}} (-.5,0) to node [rectangle,fill=white,inner sep=1pt] {{\tiny C}} (0,1.5) to node [rectangle,fill=white,inner sep=1pt] {{\tiny B}} (.5,0);
\draw[rotate=270] (-1.5,0) to node [rectangle,fill=white,inner sep=1pt] {{\tiny E}} (-.5,0) to node [rectangle,fill=white,inner sep=1pt] {{\tiny A}} (0,1.5) to node [rectangle,fill=white,inner sep=1pt] {{\tiny D}} (.5,0);
\draw (0,-.5) node[shape=circle,fill=gray!0,inner sep=0pt] {$\bullet$};
\draw (0,.5) node[shape=circle,fill=gray!0,inner sep=0pt] {$\bullet$};
\draw (-1.5,0) node[circle,draw,fill=white,inner sep=1pt] {{\tiny 3}};
\draw (1.5,0) node[circle,draw,fill=white,inner sep=1pt] {{\tiny 1}};
\draw (0,-1.5) node[circle,draw,fill=white,inner sep=1pt] {{\tiny 4}};
\draw (0,1.5) node[circle,draw,fill=white,inner sep=1pt] {{\tiny 2}};
\draw (0,0) node[shape=circle,fill=gray!0,inner sep=0pt] {$\scriptstyle{F'}$};
\draw (45:1) node[shape=circle,fill=gray!0,inner sep=0pt] {$\scriptstyle{F'_1}$};
\draw (135:1) node[shape=circle,fill=gray!0,inner sep=0pt] {$\scriptstyle{F'_2}$};
\draw (225:1) node[shape=circle,fill=gray!0,inner sep=0pt] {$\scriptstyle{F'_3}$};
\draw (315:1) node[shape=circle,fill=gray!0,inner sep=0pt] {$\scriptstyle{F'_4}$};
\end{scope}
\begin{scope}[xshift=8.5cm,yshift=-1.5cm]
\draw (1.5,3.8) node {$\bir$};
\begin{scope}[scale=.75]
\draw[dotted] (0,0)--(4,0)--(4,4)--(0,4)--(0,0);
\draw[dashed,-latex] (1,-.5) node[below] {{$\scriptstyle{F}$}} --(1,4.5);
\draw[dashed,-latex] (-.2,.3) node[left] {{$\scriptstyle{F_1}$}} arc (130:50:3.5);
\draw[dashed,-latex] (-.2,2.3) node[left] {{$\scriptstyle{F_3}$}} arc (130:50:3.5);
\draw[dashed,-latex] (4.2,3.6) node[right] {{$\scriptstyle{F_2}$}} --(-.2,1.4);
\draw[dashed,-latex] (4.2,1.6) node[right] {{$\scriptstyle{F_4}$}} --(-.2,-.6);
\draw (2,1) node[shape=circle,fill=gray!0,inner sep=0pt] {$\bullet$};
\draw (2,3) node[shape=circle,fill=gray!0,inner sep=0pt] {$\bullet$};
\draw (2,1) to node [rectangle,fill=white,inner sep=1pt] {{\tiny a}} (0,1);
\draw (0,1) to node [rectangle,fill=white,inner sep=1pt] {{\tiny b}} (2,3);
\draw (2,3) to node [rectangle,fill=white,inner sep=1pt] {{\tiny c}} (0,3);
\draw (0,3) -- (1,4) node [rectangle,fill=white,inner sep=1pt] {{\tiny d}};
\draw (2,1) -- (1,0) node [rectangle,fill=white,inner sep=1pt] {{\tiny d}};
\draw (2,3) to node [rectangle,fill=white,inner sep=1pt] {{\tiny f}} (4,3);
\draw (2,1) to node [rectangle,fill=white,inner sep=1pt] {{\tiny e}} (4,1);
\draw (4,1) node[circle,draw,fill=white,inner sep=1pt] {{\tiny 1}};
\draw (4,3) node[circle,draw,fill=white,inner sep=1pt] {{\tiny 3}};
\draw (0,1) node[circle,draw,fill=white,inner sep=1pt] {{\tiny 2}};
\draw (0,3) node[circle,draw,fill=white,inner sep=1pt] {{\tiny 4}};
\end{scope}
\end{scope}
\draw (12.5,0) node {$\tot$};
\begin{scope}[xshift=13.5cm,yshift=-1.5cm]
\draw (1.5,3.8) node {$\twist{\mu_F\rib}$};
\begin{scope}[scale=.75]
\draw[dotted] (0,0)--(4,0)--(4,4)--(0,4)--(0,0);
\draw[dashed,latex-] (3,-.5) node[below] {{$\scriptstyle{F'}$}} --(3,4.5);
\draw[dashed,-latex] (-.2,.3) node[left] {{$\scriptstyle{F'_1}$}} arc (130:50:3.5);
\draw[dashed,-latex] (-.2,2.3) node[left] {{$\scriptstyle{F'_3}$}} arc (130:50:3.5);
\draw[dashed,-latex] (4.2,3.4) node[right] {{$\scriptstyle{F'_2}$}} --(3,4);
\draw[dashed,-latex] (3,0) --(-.2,1.6);
\draw[dashed,-latex] (4.2,1.4) node[right] {{$\scriptstyle{F'_4}$}} --(-.2,3.6);
\draw (2,1) node[shape=circle,fill=gray!0,inner sep=0pt] {$\bullet$};
\draw (2,3) node[shape=circle,fill=gray!0,inner sep=0pt] {$\bullet$};
\draw (2,1) to node [rectangle,fill=white,inner sep=1pt] {{\tiny E}} (0,1);
\draw (4,1) to node [rectangle,fill=white,inner sep=1pt] {{\tiny D}} (2,3);
\draw (2,3) to node [rectangle,fill=white,inner sep=1pt] {{\tiny F}} (0,3);
\draw (4,3) -- (3,4) node [rectangle,fill=white,inner sep=1pt] {{\tiny B}};
\draw (2,1) -- (3,0) node [rectangle,fill=white,inner sep=1pt] {{\tiny B}};
\draw (2,3) to node [rectangle,fill=white,inner sep=1pt] {{\tiny C}} (4,3);
\draw (2,1) to node [rectangle,fill=white,inner sep=1pt] {{\tiny A}} (4,1);
\draw (4,1) node[circle,draw,fill=white,inner sep=1pt] {{\tiny 1}};
\draw (4,3) node[circle,draw,fill=white,inner sep=1pt] {{\tiny 3}};
\draw (0,1) node[circle,draw,fill=white,inner sep=1pt] {{\tiny 2}};
\draw (0,3) node[circle,draw,fill=white,inner sep=1pt] {{\tiny 4}};
\end{scope}
\end{scope}
\end{scope}
\end{tikzpicture}
\end{center}
If we look at the mirror version, the face walk $F$ becomes a zigzag walk and the mutation only changes a tubular neighborhood of this 
zigzag walk. Above we represented this neighborhood by a dotted square of which the top and bottom sides are identified.
The new mirror dimer is embedded in the same punctured surface as the original mirror dimer. 
This gives us an isomorphism 
\[
 \phi: \Lambda_{\mu_F\rib}  \to \Lambda_{\rib} 
\]
that is compatible with the intersection pairing.

\begin{lemma}[\cite{goncharov2013dimers} lemma 4.1]
If $\rib$ and $\mu_F\rib$ are related by mutation at the face $F$ and
all faces in the mutation picture are different, then 
the isomorphism $\phi : \Lambda_{\rib}\to \Lambda_{\mu_F\rib}$ 
maps every face walk to its corresponding face walk in the mutated dimer 
except 
\[
F' \mapsto -F,~F'_i \mapsto \begin{cases}
F_i &i=1,3\\
F_i + F&i=2,4
                           \end{cases}
\]
This isomorphism preserves the intersection pairing.
\end{lemma}
\begin{proof}
This follows immediately from the picture above.
\end{proof}

Because $\Lambda_{\rib}$ and $\Lambda_{\mu_F\rib}$ are isomorphic, there is a corresponding isomorphism between the Poisson algebras $\phi:\C(\cL(\mu_F\rib)) \to \C(\cL(\rib))$.

Given a vector $v \in \Lambda_{\rib}$ we can also define a Poisson automorphism of $\C(\cL(\rib))$
\[
 \mu_v^* : \C(\cL(\rib)) \to \C(\cL(\rib)): X_u \mapsto X_u(1+X_v)^{\eps(u,v)}
\]
If we compose these two morphisms we get a morphism that looks like the mutation operation for cluster Poisson varieties.

\begin{definition}
Given a mutation between $\rib$ and $\mu\rib$ we define the \iemph{cluster transformation} as
\[
 \phi_\mu : \C(\mu\cL(\rib)) \to \C(\cL(\rib)): X_{u'} \mapsto  X_{\phi(u')}(1+X_v)^{\eps(\phi(u'),v)}
\]
where $u' \in \Lambda_{\mu\rib}$.
\end{definition}

\begin{lemma}[\cite{goncharov2013dimers} Lemma 4.8]
If $\rib$ and $\mu_F\rib$ are related by mutation at the face $F$ and all faces in the spider move are different,
then $\phi_\mu$ maps the holonomy of every face walk to the holonomy of
its corresponding face walk in the mutated dimer
except
\[
W_F' \mapsto W_{F}^{-1},~W_{F_i'}\mapsto 
\begin{cases}
W_{F_i}(1+W_{F})&i=1,3\\
W_{F_i}(1+W_{F}^{-1})^{-1}&i=2,4
\end{cases}
\]
and $\phi$ is compatible with the Poisson structure.
\end{lemma}
\begin{proof}
This is an easy computation.
\end{proof}
Note that we can restrict $\phi_\mu$ to $\C(W_F|F\in \rib_F)=\C(\cX(\rib))$ and the $\deg_\TT$-graded ring $\C(\cX(\rib)[W_Z^{\pm 1}|Z \in \bir_2]$.
The last restriction turns $\phi_\mu$ into a $\deg_\TT$-graded homomorphism because $\phi$ identifies identical curves on $\ful{\rib}=\ful{\mu\rib}$
and $\mu_v$ multiplies each generator with a degree zero element. This implies that $\phi_\mu$ defines a $\cT$-equivariant birational morphism
between $\cL(\rib)$ and $\cL(\mu\rib)$.

Using the maps $\phi_\mu$ we can construct the major object of study of the $A$-model.
\begin{definition}
The Poisson bundle $\pi:\cL\to \cX$ associated to a mutation class of reduced dimers is the colimit of all mutation maps
\[
 \xymatrix{
 \cL(\rib)\ar@{->>}[d]\ar[r]^{\phi_\mu}&\cL(\mu\rib) \ar@{->>}[d]\\
 \cX(\rib)\ar[r]&\cL(\mu\rib)
 }
\]
The space $\cL$ has an action of the $2$-dimensional torus $\cT$ and $\pi:\cL\to \cX$ is the quotient map for this action.
\end{definition}

\subsection{The partition function}

\newcommand{\Part}{\mathcal P}
\begin{definition}
Let $\cP$ be a perfect matching. The \iemph{weight function} of the perfect matching is the map
\[
 W_\cP : \cM \to \C: (x_e)_{e \in \rib_1} \mapsto \prod_{e \in \cP} x_e.
\]
\end{definition}
\newcommand{\LL}{l}
\newcommand{\Kast}{\mathrm{Kast}}

We can package the weights of all perfect matchings into one nice function, the partition function.
To do this we start with a \iemph{Kasteleyn weight system}. 
This is a map $\kappa: \rib_1 \to \C$ such that the holonomy of a face $F$ in $\rib_2$ is $(-1)^{\ell(F)/2+1}$ if
the face walk has length $\ell(F)$ and $\pm 1$ around noncontractible loops.
Note that $\kappa$ can also be considered as a line bundle in $\cL(\rib)$ and up to isomorphism of line bundles there are $2g$ choices for $\kappa$ 
because $\dim \cL(\rib)= \# \rib_2 -1 + 2g$.

Given a map $\LL:\rib_1 \to \C$, the \iemph{Kasteleyn operator}
associated to $\LL$ is the linear map
\[
 \Kast(\LL) : \C^{\# \rib_0^\bullet} \to \C^{\# \rib_0^\circ} : b \mapsto \sum_{e:b(e)=b} \kappa_e\LL_e w.
\]
\begin{theorem}[Kasteleyn \cite{kasteleyn,tesler2000matchings}]
\[
 \det \Kast(\LL) = \sum_{\PM \in \PMs}\pm W_\PM(\LL)
\]
where the sign of $W_\PM$ only depends on the location of $\PM$ in the matching polygon. 
\end{theorem}

\begin{definition}
The \iemph{partition function} $\Part: \cM \to \C$ is the determinant of the Kasteleyn operator:
\[
 \Part(\LL) = \det \Kast(\LL). 
\]
If $\qpol$ is isoradially embedded and $\cP_\theta$ is a corner matching then we define the deformed partition function
\[
 \Part_\theta = \frac{\Part}{W_{\cP_\theta}}.
\]
(In \cite{goncharov2013dimers} the deformed partition function is defined via a function $\Phi_\alpha$, which corresponds in our terminology to a corner matching.)
\end{definition}

Each weight function transforms under the action of the group $\cG$ with the same character
\[
 g \cdot W_{\cP} = \prod_{b \in \rib_0^\bullet} g_{b}\prod_{w \in \rib_0^\bullet} g_{w} W_{\cP},
\]
so the partition function is a semi-invariant. The deformed partition function is a quotient of two 
semi-invariants with the same character, so it is an invariant: $\Part_\theta \in \C[\cL(\rib)]$.

Using the $\deg_\TT$-grading we can split $\Part_\theta$ into homogeneous components. Each homogeneous component
will contain the contributions of all perfect matchings on a fixed lattice point in the matching polygon.
By Kasteleyn's theorem they all have the same sign so we can write
\[
\Part_\theta = \sum_{a \in \Z^2} \pm H_{a,\theta}
\]
where $H_{a,\theta}= \sum_{\PM \in \PMs_a} \frac {W_{\PM}}{W_{\cP_\theta}}$ and $\PMs_a$ stands for the set of all matchings on lattice point $a$.

Now let $(\rib,\qpol)$ be a consistent dimer and $(\mu\rib,\mu\qpol)$ its mutation at vertex $v$/face $F$. 
Because of lemma \ref{mutpres}, we know that $\mu$ preserves isoradial embeddings we can assume that both $\qpol$ and $\mu\qpol$ are isoradially embedded.
This gives us a way to identify the corner matchings $\cP_\theta \tot \cP_\theta'$ and hence also the deformed partition functions
$\Part_\theta \tot \Part_\theta'$. Finally because mutation preserves the matching polygon, it is possible to identify the 
graded parts $H_{a,\theta} \tot H_{a,\theta}'$.

\begin{theorem}[Goncharov-Kenyon \cite{goncharov2013dimers} theorem 4.7]
The map $\phi_\mu: \C(\cL(\rib))\to \C(\cL(\mu\rib))$ identifies graded components of the deformed partition function:
\[
 \phi_\mu(H_{a,\theta}) = H_{a,\theta}'.
\]
\end{theorem}
\begin{proof}
The proof consists of a subtle bookkeeping argument that splits $H_{a,\theta}$ into parts $p_{ij}$ coming from matchings that connect
the nodes $i,j$ in the mutation picture to nodes outside the picture. For the details of this proof we refer to \cite{goncharov2013dimers}. 
\end{proof}

\begin{example}\label{sppmut}
Start with the suspended pinchpoint and mutate at vertex $2$.
\begin{center}
\begin{tikzpicture}
\begin{scope}[scale=1]
\draw[thick] (2,.5)--(1.5,0);
\draw[thick] (2,.5)--(1.5,1);
\draw[thick] (1.5,2)--(1.5,1);
\draw[thick] (1.5,2)--(1,2.5);
\draw[thick] (1,2.5)--(1.5,3);
\draw[thick] (1,2.5)--(0,2.3);
\draw[thick] (3,2.3)--(1.5,2);
\draw[thick] (0,1.5)--(1.5,2);
\draw[thick] (3,1.5)--(1.5,1);
\draw[thick] (2,.5)--(3,.7);
\draw[thick] (1.5,1)--(0,.7);
\draw (2,.5) node {$\bullet$};    
\draw (1.5,2) node {$\bullet$};    
\draw (1.5,1) node[shape=circle,fill=gray!0,inner sep=0pt] {$\circ$};    
\draw (1,2.5) node[shape=circle,fill=gray!0,inner sep=0pt] {$\circ$};    
\draw [-latex,dotted,shorten >=5pt] (0,0) -- (3,0);
\draw [-latex,dotted,shorten >=5pt] (3,0) -- (3,1);
\draw [-latex,dotted,shorten >=5pt] (3,1) -- (0,0);
\draw [-latex,dotted,shorten >=5pt] (0,0) -- (0,1);
\draw [-latex,dotted,shorten >=5pt] (0,1) -- (3,2);
\draw [-latex,dotted,shorten >=5pt] (3,2) -- (3,3);
\draw [-latex,dotted,shorten >=5pt] (3,2) -- (3,1);
\draw [-latex,dotted,shorten >=5pt] (0,2) -- (0,1);
\draw [-latex,dotted,shorten >=5pt] (0,3) -- (3,3);
\draw [-latex,dotted,shorten >=5pt] (3,3) -- (0,2);
\draw [-latex,dotted,shorten >=5pt] (0,2) -- (0,3);
\draw (1.5,0) node[circle,fill=white,inner sep=1pt] {{\tiny a}};
\draw (1.5,3) node[circle,fill=white,inner sep=1pt] {{\tiny a}};
\draw (1.8,0.65) node[circle,fill=white,inner sep=1pt] {{\tiny c}};
\draw (1.5,1.5) node[circle,fill=white,inner sep=1pt] {{\tiny d}};
\draw (1.2,2.35) node[circle,fill=white,inner sep=1pt] {{\tiny f}};
\draw (0,.7) node[circle,fill=white,inner sep=1pt] {{\tiny b}};
\draw (0,1.5) node[circle,fill=white,inner sep=1pt] {{\tiny e}};
\draw (0,2.3) node[circle,fill=white,inner sep=1pt] {{\tiny g}};
\draw (3,.7) node[circle,fill=white,inner sep=1pt] {{\tiny b}};
\draw (3,1.5) node[circle,fill=white,inner sep=1pt] {{\tiny e}};
\draw (3,2.3) node[circle,fill=white,inner sep=1pt] {{\tiny g}};
\draw (0,0) node[circle,draw,fill=white,minimum size=10pt,inner sep=1pt] {{\tiny1}};
\draw (0,1) node[circle,draw,fill=white,minimum size=10pt,inner sep=1pt] {{\tiny2}};
\draw (0,2) node[circle,draw,fill=white,minimum size=10pt,inner sep=1pt] {{\tiny3}};
\draw (0,3) node[circle,draw,fill=white,minimum size=10pt,inner sep=1pt] {{\tiny1}};
\draw (3,0) node[circle,draw,fill=white,minimum size=10pt,inner sep=1pt] {{\tiny1}};
\draw (3,1) node[circle,draw,fill=white,minimum size=10pt,inner sep=1pt] {{\tiny2}};
\draw (3,2) node[circle,draw,fill=white,minimum size=10pt,inner sep=1pt] {{\tiny3}};
\draw (3,3) node[circle,draw,fill=white,minimum size=10pt,inner sep=1pt] {{\tiny1}};
\end{scope}

\begin{scope}[xshift=5cm]
\draw[thick] (2,2.5)--(1.5,2);
\draw[thick] (2,2.5)--(1.5,3);
\draw[thick] (1.5,1)--(1.5,0);
\draw[thick] (1.5,1)--(1,1.5);
\draw[thick] (1,1.5)--(1.5,2);
\draw[thick] (1,1.5)--(0,1.3);
\draw[thick] (3,1.3)--(1.5,1);
\draw[thick] (0,.5)--(1.5,1);
\draw[thick] (3,0.5)--(1.5,0);
\draw[thick] (2,2.5)--(3,2.7);
\draw[thick] (1.5,3)--(0,2.7);
\draw [-latex,dotted,shorten >=5pt] (0,2) -- (3,2);
\draw [-latex,dotted,shorten >=5pt] (3,2) -- (3,3);
\draw [-latex,dotted,shorten >=5pt] (3,3) -- (0,2);
\draw [-latex,dotted,shorten >=5pt] (0,2) -- (0,3);
\draw [-latex,dotted,shorten >=5pt] (0,0) -- (3,1);
\draw [-latex,dotted,shorten >=5pt] (3,1) -- (3,2);
\draw [-latex,dotted,shorten >=5pt] (3,1) -- (3,3);
\draw [-latex,dotted,shorten >=5pt] (0,1) -- (0,3);
\draw [-latex,dotted,shorten >=5pt] (0,2) -- (3,2);
\draw [-latex,dotted,shorten >=5pt] (3,2) -- (0,1);
\draw [-latex,dotted,shorten >=5pt] (0,1) -- (0,2);
\draw [-latex,dotted,shorten >=5pt] (0,1) -- (0,0);
\draw [-latex,dotted,shorten >=5pt] (3,1) -- (3,0);
\draw[loosely dotted] (0,0)--(3,0);
\draw[loosely dotted] (0,3)--(3,3);
\draw (1.5,2) node[circle,fill=white,inner sep=1pt] {{\tiny DE}};
\draw (1.8,2.65) node[circle,fill=white,inner sep=1pt] {{\tiny F}};
\draw (1.5,0.5) node[circle,fill=white,inner sep=1pt] {{\tiny C*}};
\draw (1.2,1.35) node[circle,fill=white,inner sep=1pt] {{\tiny D*}};
\draw (0,2.7) node[circle,fill=white,inner sep=1pt] {{\tiny G}};
\draw (0,0.5) node[circle,fill=white,inner sep=1pt] {{\tiny B*}};
\draw (0,1.3) node[circle,fill=white,inner sep=1pt] {{\tiny E*}};
\draw (3,2.7) node[circle,fill=white,inner sep=1pt] {{\tiny G}};
\draw (3,.5) node[circle,fill=white,inner sep=1pt] {{\tiny B*}};
\draw (3,1.3) node[circle,fill=white,inner sep=1pt] {{\tiny E*}};
\draw (0,0) node[circle,draw,fill=white,minimum size=10pt,inner sep=1pt] {{\tiny1}};
\draw (0,1) node[circle,draw,fill=white,minimum size=10pt,inner sep=1pt] {{\tiny2}};
\draw (0,2) node[circle,draw,fill=white,minimum size=10pt,inner sep=1pt] {{\tiny3}};
\draw (0,3) node[circle,draw,fill=white,minimum size=10pt,inner sep=1pt] {{\tiny1}};
\draw (3,0) node[circle,draw,fill=white,minimum size=10pt,inner sep=1pt] {{\tiny1}};
\draw (3,1) node[circle,draw,fill=white,minimum size=10pt,inner sep=1pt] {{\tiny2}};
\draw (3,2) node[circle,draw,fill=white,minimum size=10pt,inner sep=1pt] {{\tiny3}};
\draw (3,3) node[circle,draw,fill=white,minimum size=10pt,inner sep=1pt] {{\tiny1}};
\draw (2,2.5) node {$\bullet$};    
\draw (1.5,1) node {$\bullet$};    
\draw (1.5,0) node[shape=circle,fill=gray!0,inner sep=0pt] {$\circ$};    
\draw (1.5,3) node[shape=circle,fill=gray!0,inner sep=0pt] {$\circ$};    
\draw (1,1.5) node[shape=circle,fill=gray!0,inner sep=0pt] {$\circ$};    
\end{scope}
\end{tikzpicture}
\end{center}
The mutated dimer is isomorphic to the original translated by $(0,\frac{2}3)$.
We take $\cP_\theta=\{c,f\} \sim \{F,D^*\}$ for our reference corner matching. 
The Hamiltonian on the point $(1,0)$ is
\se{
H_{a,\theta} &= \frac{cg}{cf}+\frac{bf}{cf} = (1+W_1)\frac gf\\
H'_{a,\theta} &= \frac{FE^*}{FD^*}+\frac{GD^*}{FD^*} = ({W_3'}^{-1}+1)\frac G{F}
}
For the mutation transformation we have to be a bit careful. Note that both
$W_1$ and $W_3$ are adjacent to $W_2$ twice and therefore they get $2$ factors
\se{
 W_1' &= W_1(1+W_2)(1+W_2^{-1})^{-1}=W_1W_2 = W_3^{-1} \\
 W_2' &= W_2^{-1}\\
 W_3' &= W_3(1+W_2)(1+W_2^{-1})^{-1}= W_3W_2=W_1^{-1}
}
The functions $\frac{g}{f}$ and $\frac GF$ represent the same homology class
and that does not intersect $W_2$  in $\twist{\rib}$ so these functions
are identified under the cluster transformation.
From this it is easy to see that $\phi_\mu(H'_{a,\theta})= H_{a,\theta}$.
Note that because $\bir$ is a sphere, the Poisson structure in this example is trivial.
\end{example}

\begin{example}\label{P1P1}
Let $\rib$ and $\mu\rib$ be the following dimer graphs.
They are related by a mutation at vertex $4$ which consists of two split moves at the white nodes, a spider move at face $4$, and two join moves.
\begin{center}
\begin{tikzpicture}
\begin{scope}[scale=.66]
\begin{scope}
\draw[dotted,-latex,shorten >= 5pt] (0,0) -- (2,0);
\draw[dotted,-latex,shorten >= 5pt] (2,0) -- (2,2);
\draw[dotted,-latex,shorten >= 5pt] (2,2) -- (0,2);
\draw[dotted,-latex,shorten >= 5pt] (0,2) -- (0,0);
\draw[dotted,-latex,shorten >= 5pt] (2,2) -- (4,2);
\draw[dotted,-latex,shorten >= 5pt] (4,2) -- (4,4);
\draw[dotted,-latex,shorten >= 5pt] (4,4) -- (2,4);
\draw[dotted,-latex,shorten >= 5pt] (2,4) -- (2,2);
\draw[dotted,-latex,shorten >= 5pt] (0,2) -- (0,4);
\draw[dotted,-latex,shorten >= 5pt] (0,4) -- (2,4);
\draw[dotted,-latex,shorten >= 5pt] (4,2) -- (4,0);
\draw[dotted,-latex,shorten >= 5pt] (4,0) -- (2,0);

\draw[thick] (1,0)--(1,1)--(0,1);
\draw[thick] (3,4)--(3,3)--(4,3);
\draw[thick] (3,0)--(3,1)--(4,1);
\draw[thick] (0,3)--(1,3)--(1,4);
\draw[thick] (3,0)--(3,1)--(4,1);
\draw[thick] (1,1)--(3,1)--(3,3)--(1,3)--(1,1);
\draw (3,1) node[shape=circle,fill=gray!0,inner sep=0pt] {$\circ$};
\draw (1,3) node[shape=circle,fill=gray!0,inner sep=0pt] {$\circ$};
\draw (1,1) node[shape=circle,fill=gray!0,inner sep=0pt] {$\bullet$};
\draw (3,3) node[shape=circle,fill=gray!0,inner sep=0pt] {$\bullet$};
\draw (1,0) node[circle,fill=white,inner sep=1pt] {{\tiny a}};
\draw (3,0) node[circle,fill=white,inner sep=1pt] {{\tiny b}};
\draw (1,4) node[circle,fill=white,inner sep=1pt] {{\tiny a}};
\draw (3,4) node[circle,fill=white,inner sep=1pt] {{\tiny b}};
\draw (0,1) node[circle,fill=white,inner sep=1pt] {{\tiny c}};
\draw (0,3) node[circle,fill=white,inner sep=1pt] {{\tiny d}};
\draw (4,1) node[circle,fill=white,inner sep=1pt] {{\tiny c}};
\draw (4,3) node[circle,fill=white,inner sep=1pt] {{\tiny d}};
\draw (2,1) node[circle,fill=white,inner sep=1pt] {{\tiny e}};
\draw (3,2) node[circle,fill=white,inner sep=1pt] {{\tiny f}};
\draw (2,3) node[circle,fill=white,inner sep=1pt] {{\tiny g}};
\draw (1,2) node[circle,fill=white,inner sep=1pt] {{\tiny h}};

\draw (0,0) node[circle,draw,fill=white,inner sep=1pt] {{\tiny 1}};
\draw (4,0) node[circle,draw,fill=white,inner sep=1pt] {{\tiny 1}};
\draw (4,4) node[circle,draw,fill=white,inner sep=1pt] {{\tiny 1}};
\draw (0,4) node[circle,draw,fill=white,inner sep=1pt] {{\tiny 1}};
\draw (2,2) node[circle,draw,fill=white,inner sep=1pt] {{\tiny 4}};
\draw (2,0) node[circle,draw,fill=white,inner sep=1pt] {{\tiny 2}};
\draw (4,2) node[circle,draw,fill=white,inner sep=1pt] {{\tiny 3}};
\draw (2,4) node[circle,draw,fill=white,inner sep=1pt] {{\tiny 2}};
\draw (0,2) node[circle,draw,fill=white,inner sep=1pt] {{\tiny 3}};
\end{scope}
\begin{scope}[xshift=8cm, yshift=2cm]
\begin{scope}[xshift=-2cm,yshift=-2cm]
\draw[dotted,-latex,shorten >= 5pt] (0,0) -- (2,0);
\draw[dotted,-latex,shorten >= 5pt] (2,0) -- (0,2);
\draw[dotted,-latex,shorten >= 5pt] (0,2) -- (0,0);
\draw[dotted,-latex,shorten >= 5pt] (4,2) -- (4,4);
\draw[dotted,-latex,shorten >= 5pt] (4,4) -- (2,4);
\draw[dotted,-latex,shorten >= 5pt] (2,4) -- (4,2);
\draw[dotted,-latex,shorten >= 5pt] (2,2) -- (2,4);
\draw[dotted,-latex,shorten >= 5pt] (2,2) -- (2,0);
\draw[dotted,-latex,shorten >= 5pt] (2,0) -- (0,2);
\draw[dotted,-latex,shorten >= 5pt] (4,2) -- (2,2);
\draw[dotted,-latex,shorten >= 5pt] (0,2) -- (2,2);
\draw[dotted,-latex,shorten >= 5pt] (0,2) -- (2,2);
\draw[dotted,-latex,shorten >= 5pt] (2,0) -- (4,2);
\draw[dotted,-latex,shorten >= 5pt] (2,4) -- (0,2);

\draw[dotted,-latex,shorten >= 5pt] (0,2) -- (0,4);
\draw[dotted,-latex,shorten >= 5pt] (0,4) -- (2,4);
\draw[dotted,-latex,shorten >= 5pt] (4,2) -- (4,0);
\draw[dotted,-latex,shorten >= 5pt] (4,0) -- (2,0);
\draw (0,0) node[circle,draw,fill=white,inner sep=1pt] {{\tiny 1}};
\draw (4,0) node[circle,draw,fill=white,inner sep=1pt] {{\tiny 1}};
\draw (4,4) node[circle,draw,fill=white,inner sep=1pt] {{\tiny 1}};
\draw (0,4) node[circle,draw,fill=white,inner sep=1pt] {{\tiny 1}};
\draw (2,2) node[circle,draw,fill=white,inner sep=1pt] {{\tiny 4}};
\draw (2,0) node[circle,draw,fill=white,inner sep=1pt] {{\tiny 2}};
\draw (4,2) node[circle,draw,fill=white,inner sep=1pt] {{\tiny 3}};
\draw (2,4) node[circle,draw,fill=white,inner sep=1pt] {{\tiny 2}};
\draw (0,2) node[circle,draw,fill=white,inner sep=1pt] {{\tiny 3}};

\end{scope}
\draw[thick] (-2,-4/3)--(-4/3,-4/3)--(-4/3,-2)--(-4/3,-4/3)--(-2/3,-2/3) -- (-2/3,0) -- (-2/3,-2/3) -- (0,-2/3);
\draw[rotate=90,thick] (-2,-4/3)--(-4/3,-4/3)--(-4/3,-2)--(-4/3,-4/3)--(-2/3,-2/3) -- (-2/3,0) -- (-2/3,-2/3) --(0,-2/3);
\draw[rotate=180,thick] (-2,-4/3)--(-4/3,-4/3)--(-4/3,-2)--(-4/3,-4/3)--(-2/3,-2/3) -- (-2/3,0) -- (-2/3,-2/3) --(0,-2/3);
\draw[rotate=270,thick] (-2,-4/3)--(-4/3,-4/3)--(-4/3,-2)--(-4/3,-4/3)--(-2/3,-2/3) -- (-2/3,0) -- (-2/3,-2/3) --(0,-2/3);

\draw (-4/3,-2) node[circle,fill=white,inner sep=1pt] {{\tiny A}};
\draw (4/3,-2) node[circle,fill=white,inner sep=1pt] {{\tiny B}};
\draw (-4/3,2) node[circle,fill=white,inner sep=1pt] {{\tiny A}};
\draw (4/3,2) node[circle,fill=white,inner sep=1pt] {{\tiny B}};
\draw (-2,-4/3) node[circle,fill=white,inner sep=1pt] {{\tiny C}};
\draw (-2,4/3) node[circle,fill=white,inner sep=1pt] {{\tiny D}};
\draw (2,-4/3) node[circle,fill=white,inner sep=1pt] {{\tiny C}};
\draw (2,4/3) node[circle,fill=white,inner sep=1pt] {{\tiny D}};
\draw (0,-.66) node[circle,fill=white,inner sep=1pt] {{\tiny E}};
\draw (.66,0) node[circle,fill=white,inner sep=1pt] {{\tiny F}};
\draw (0,.66) node[circle,fill=white,inner sep=1pt] {{\tiny G}};
\draw (-.66,0) node[circle,fill=white,inner sep=1pt] {{\tiny H}};

\draw (-1,-1) node[circle,fill=white,inner sep=1pt] {{\tiny I}};
\draw (1,-1) node[circle,fill=white,inner sep=1pt] {{\tiny J}};
\draw (1,1) node[circle,fill=white,inner sep=1pt] {{\tiny K}};
\draw (-1,1) node[circle,fill=white,inner sep=1pt] {{\tiny L}};
\draw (-2/3,-2/3) node[shape=circle,fill=gray!0,inner sep=0pt] {$\circ$};
\draw (-4/3,-4/3) node[shape=circle,fill=gray!0,inner sep=0pt] {$\bullet$};
\draw (2/3,-2/3) node[shape=circle,fill=gray!0,inner sep=0pt] {$\bullet$};
\draw (4/3,-4/3) node[shape=circle,fill=gray!0,inner sep=0pt] {$\circ$};
\draw (2/3,2/3) node[shape=circle,fill=gray!0,inner sep=0pt] {$\circ$};
\draw (4/3,4/3) node[shape=circle,fill=gray!0,inner sep=0pt] {$\bullet$};
\draw (-2/3,2/3) node[shape=circle,fill=gray!0,inner sep=0pt] {$\bullet$};
\draw (-4/3,4/3) node[shape=circle,fill=gray!0,inner sep=0pt] {$\circ$};
\end{scope}
\end{scope}
\end{tikzpicture}
\end{center}
The dimer quivers are isoradially embedded and we consider
\[
\cP_0=\{d,e\} \text{ and } \cP'_0=\{D,G,I,J\}.
\]
If we put $\cP_0$ on the origin the matching polygon becomes the square 
\begin{center}
\begin{tikzpicture}[scale=.25]
\draw (1,0) node{$\scriptstyle{\bullet}$} -- (0,1) node{$\scriptstyle{\bullet}$} --(-1,0) node{$\scriptstyle{\bullet}$} -- (0,-1) node{$\scriptstyle{\bullet}$} --(1,0);
\draw (0,0) node{$\scriptstyle{\bullet}$}; 
\end{tikzpicture}
\end{center}
which has one internal lattice point $u$. For this lattice point we can calculate
\se{
 H_{u,0} &= \frac{ab+cd+eg+hf}{ed}\\
 &=(W_1+1+\frac1{W_1W_2}+W_1W_2W_4)\frac ce\\
  H'_{u,0} &= \frac{CDFH+CDEG+ABFH+ABEG+IJKL}{DGIJ}\\
  &= (\frac 1{W_4'}+1 +\frac{W_1'}{W_4'}+W_1'+W_1'W_2')\frac {CE}{IJ}
}
To find the transformation, we need to look at the mirror dimers.
\begin{center}
\begin{tikzpicture}
\begin{scope}[scale=.66]
\begin{scope}
\draw[dotted,-latex,shorten >= 5pt] (0,0) -- (2,0);
\draw[dotted,-latex,shorten >= 5pt] (2,0) -- (2,2);
\draw[dotted,-latex,shorten >= 5pt] (2,2) -- (0,2);
\draw[dotted,-latex,shorten >= 5pt] (0,2) -- (0,0);
\draw[dotted,-latex,shorten >= 5pt] (2,2) -- (4,2);
\draw[dotted,-latex,shorten >= 5pt] (4,2) -- (4,4);
\draw[dotted,-latex,shorten >= 5pt] (4,4) -- (2,4);
\draw[dotted,-latex,shorten >= 5pt] (2,4) -- (2,2);
\draw[dotted,-latex,shorten >= 5pt] (0,2) -- (0,4);
\draw[dotted,-latex,shorten >= 5pt] (0,4) -- (2,4);
\draw[dotted,-latex,shorten >= 5pt] (4,2) -- (4,0);
\draw[dotted,-latex,shorten >= 5pt] (4,0) -- (2,0);

\draw[thick] (1,0)--(1,1)--(0,1);
\draw[thick] (3,4)--(3,3)--(4,3);
\draw[thick] (3,0)--(3,1)--(4,1);
\draw[thick] (0,3)--(1,3)--(1,4);
\draw[thick] (3,0)--(3,1)--(4,1);
\draw[thick] (1,1)--(3,1)--(3,3)--(1,3)--(1,1);
\draw (3,1) node[shape=circle,fill=gray!0,inner sep=0pt] {$\circ$};
\draw (1,3) node[shape=circle,fill=gray!0,inner sep=0pt] {$\circ$};
\draw (1,1) node[shape=circle,fill=gray!0,inner sep=0pt] {$\bullet$};
\draw (3,3) node[shape=circle,fill=gray!0,inner sep=0pt] {$\bullet$};
\draw (1,0) node[circle,fill=white,inner sep=1pt] {{\tiny a}};
\draw (3,0) node[circle,fill=white,inner sep=1pt] {{\tiny f}};
\draw (1,4) node[circle,fill=white,inner sep=1pt] {{\tiny a}};
\draw (3,4) node[circle,fill=white,inner sep=1pt] {{\tiny f}};
\draw (0,1) node[circle,fill=white,inner sep=1pt] {{\tiny c}};
\draw (0,3) node[circle,fill=white,inner sep=1pt] {{\tiny g}};
\draw (4,1) node[circle,fill=white,inner sep=1pt] {{\tiny c}};
\draw (4,3) node[circle,fill=white,inner sep=1pt] {{\tiny g}};
\draw (2,1) node[circle,fill=white,inner sep=1pt] {{\tiny e}};
\draw (3,2) node[circle,fill=white,inner sep=1pt] {{\tiny b}};
\draw (2,3) node[circle,fill=white,inner sep=1pt] {{\tiny d}};
\draw (1,2) node[circle,fill=white,inner sep=1pt] {{\tiny h}};

\draw (0,0) node[circle,draw,fill=white,inner sep=1pt] {{\tiny 1}};
\draw (4,0) node[circle,draw,fill=white,inner sep=1pt] {{\tiny 1}};
\draw (4,4) node[circle,draw,fill=white,inner sep=1pt] {{\tiny 1}};
\draw (0,4) node[circle,draw,fill=white,inner sep=1pt] {{\tiny 1}};
\draw (2,2) node[circle,draw,fill=white,inner sep=1pt] {{\tiny 4}};
\draw (2,0) node[circle,draw,fill=white,inner sep=1pt] {{\tiny 2}};
\draw (4,2) node[circle,draw,fill=white,inner sep=1pt] {{\tiny 3}};
\draw (2,4) node[circle,draw,fill=white,inner sep=1pt] {{\tiny 2}};
\draw (0,2) node[circle,draw,fill=white,inner sep=1pt] {{\tiny 3}};
\end{scope}
\begin{scope}[xshift=8cm, yshift=2cm]
\begin{scope}[xshift=-2cm,yshift=-2cm]
\draw[dotted,-latex,shorten >= 5pt] (0,4) -- (2,4);
\draw[dotted,-latex,shorten >= 5pt] (2,4) -- (2,2);
\draw[dotted,-latex,shorten >= 5pt] (2,2) -- (4,2);
\draw[dotted,-latex,shorten >= 5pt] (4,2) -- (4,0);
\draw[dotted,-latex,shorten >= 5pt] (4,0) -- (2,2);
\draw[dotted,-latex,shorten >= 5pt] (2,2) -- (0,4);
\draw[dotted,-latex,shorten >= 5pt] (2,0) -- (0,2);
\draw[dotted,-latex,shorten >= 5pt] (0,2) -- (0,0);
\draw[dotted,-latex,shorten >= 5pt] (0,0) -- (2,0);
\draw[dotted,-latex,shorten >= 5pt] (2,2) -- (0,2);
\draw[dotted,-latex,shorten >= 5pt] (4,0) -- (2,0);
\draw[dotted,-latex,shorten >= 5pt] (0,2) -- (4,0);
\draw[dotted] (4,3) -- (2,4);
\draw[dotted,-latex,shorten >= 5pt] (0,3) -- (2,2);
\draw[dotted,-latex,shorten >= 5pt] (4,2) -- (2,4);
\draw[loosely dotted] (0,4) -- (0,2);
\draw[loosely dotted] (4,4) -- (4,2);

\draw[thick] (1,4)--(1.5,3.5)--(2.5,2.5)--(3.5,1.5)--(4,1);
\draw[thick] (0,3.5)--(0,2.5)--(2,1.5)--(2,0.5)--(3,0);
\draw[thick] (3,4)--(4,3.5)--(4,2.5)--(2.5,2.5);
\draw[thick] (1.5,3.5)--(0,3.5);
\draw[thick] (3.5,1.5)--(2,1.5);
\draw[thick] (3.5,1.5)--(4,1);
\draw[thick] (0,1)--(.5,.5)--(1,0);
\draw[thick] (.5,.5)--(2,.5)--(3,0);

\draw (.5,.5) node[shape=circle,fill=gray!0,inner sep=0pt] {$\bullet$};
\draw (2.5,2.5) node[shape=circle,fill=gray!0,inner sep=0pt] {$\bullet$};
\draw (4,3.5) node[shape=circle,fill=gray!0,inner sep=0pt] {$\bullet$};
\draw (0,3.5) node[shape=circle,fill=gray!0,inner sep=0pt] {$\bullet$};
\draw (2,1.5) node[shape=circle,fill=gray!0,inner sep=0pt] {$\bullet$};

\draw (1.5,3.5) node[shape=circle,fill=gray!0,inner sep=0pt] {$\circ$};
\draw (3.5,1.5) node[shape=circle,fill=gray!0,inner sep=0pt] {$\circ$};
\draw (4,2.5) node[shape=circle,fill=gray!0,inner sep=0pt] {$\circ$};
\draw (0,2.5) node[shape=circle,fill=gray!0,inner sep=0pt] {$\circ$};
\draw (2,.5) node[shape=circle,fill=gray!0,inner sep=0pt] {$\circ$};

\draw (0,0) node[circle,draw,fill=white,inner sep=1pt] {{\tiny 1}};
\draw (4,0) node[circle,draw,fill=white,inner sep=1pt] {{\tiny 1}};
\draw (4,4) node[circle,draw,fill=white,inner sep=1pt] {{\tiny 1}};
\draw (0,4) node[circle,draw,fill=white,inner sep=1pt] {{\tiny 1}};
\draw (2,2) node[circle,draw,fill=white,inner sep=1pt] {{\tiny 4}};
\draw (2,0) node[circle,draw,fill=white,inner sep=1pt] {{\tiny 2}};
\draw (4,2) node[circle,draw,fill=white,inner sep=1pt] {{\tiny 3}};
\draw (2,4) node[circle,draw,fill=white,inner sep=1pt] {{\tiny 2}};
\draw (0,2) node[circle,draw,fill=white,inner sep=1pt] {{\tiny 3}};

\draw (1,0) node[circle,fill=white,inner sep=1pt] {{\tiny A}};
\draw (1,4) node[circle,fill=white,inner sep=1pt] {{\tiny A}};
\draw (1,2) node[circle,fill=white,inner sep=1pt] {{\tiny F}};
\draw (3,0) node[circle,fill=white,inner sep=1pt] {{\tiny H}};
\draw (3,4) node[circle,fill=white,inner sep=1pt] {{\tiny H}};
\draw (3,2) node[circle,fill=white,inner sep=1pt] {{\tiny B}};
\draw (0,1) node[circle,fill=white,inner sep=1pt] {{\tiny C}};
\draw (4,1) node[circle,fill=white,inner sep=1pt] {{\tiny C}};
\draw (2,1) node[circle,fill=white,inner sep=1pt] {{\tiny E}};
\draw (0,3) node[circle,fill=white,inner sep=1pt] {{\tiny G}};
\draw (4,3) node[circle,fill=white,inner sep=1pt] {{\tiny G}};
\draw (2,3) node[circle,fill=white,inner sep=1pt] {{\tiny D}};
\draw (1.5,.5) node[circle,fill=white,inner sep=1pt] {{\tiny I}};
\draw (3.5,2.5) node[circle,fill=white,inner sep=1pt] {{\tiny K}};
\draw (.5,3.5) node[circle,fill=white,inner sep=1pt] {{\tiny L}};
\draw (2.5,1.5) node[circle,fill=white,inner sep=1pt] {{\tiny J}};
\end{scope}
\end{scope}
\end{scope}
\end{tikzpicture}
\end{center}
$\cL(\rib)$ is generated by $W_1,W_2,W_4, \frac{c}{e},\frac{a}{h}$ and
$\cL(\mu\rib)$ is generated by $W_1',W_2',W_4', \frac{CE}{IJ},\frac{AH}{IL}$.
\begin{itemize}
 \item $W_1=\frac{ab}{cd}$ and $W'_1=\frac{AB}{CD}$ represent the same path that does not intersect $W_4$, so $W_1=W_1'$,
 \item $W_2'=\frac{IJKL}{ABEG}$ has the same homology class as $W_2=\frac{eg}{ab}$  and
its intersection number with $W_4$ is $2$ so $W'_2=W_2(1+W_4)^2$,
 \item $W_3'=\frac{CDFH}{IJKL}$ has the homology of $W_3W_4^2$ and the intersection number
 with $W_4$ is $-2$ so $W_3'= W_3W_4^2(1+W_4)^{-2}$, 
 \item $W_4'$ is the inverse of $W_4$,
 \item $\frac{CE}{IJ}$ has the same homology as $\frac{c}{e}$ and the intersection number 
 with $W_4$  is $-1$ so  $\frac{CE}{IJ}= \frac{c}{e}(1+W_4)^{-1}$.
 \item $\frac{AH}{IL}$ has the same homology as $\frac{a}{h}$ and the intersection number 
 with $W_4$  is $1$ so  $\frac{AH}{IL}= \frac{c}{e}(1+W_4)$.
 \end{itemize}
If we fill these expressions in $H'_{u,0}$ we get $H_{u,0}$.
\end{example}

\subsection{Integrability}
Given a Poisson variety it is natural to study its integrability.
This means that we look for commuting Hamiltonians globally defined on $\cL$.

On the local level this is easy to do.
\begin{lemma}[Integrability of $\cL(\rib)$]
Let $\rib$ be a consistent dimer and let $g$ be the genus of the specular dual $\bir$.
The Poisson center of $\C[\cL(\rib)]$ is generated by zigzag walks in $\rib$. Its dimension is
$\# \rib_1- \#\rib_0 +1 - 2g(\Surf)=\dim \cL(\rib) - 2 g$.
Furthermore we can find $g$ commuting Hamiltonians by taking holonomies of $g$ nonintersecting curves in $\bir$.
\end{lemma}
\begin{proof}
If $W_p$ sits in the center then $p$ does not intersect any curve on $\Surf$, so it is a contractible curve. This means it is a product of face walks
in $\bir$ or equivalently zigzag walks in $\rib$.

There are $\#\bir_2$ face walks in $\bir$ and they safisfy one relation because $\dim H_2(\Surf)=1$. 
Furthermore 
$$\#\bir_2 -1 =  (2-2g(\Surf)- \#\bir_0 +\#\bir_1)-1 = \#\rib_1 -\#\rib_0 +1 -2 g(\Surf).$$ 

We can find $g$ commuting Hamiltonians because the genus of $\bir$ is $g$, so we can find $g$ nonintersecting curves.
Because the Poisson bracket comes from the intersection form these will Poisson commute.
\end{proof}

The main problem is that this method cannot be lifted to the whole Poisson variety $\cL$ because the mutation operation does not map holonomies of
curves to holonomies of curves but to rational functions of them. However, in the previous section we encountered functions
that were invariant under the mutation move: the $H_{a,\theta}$.

\begin{lemma}
Let $\rib$ be a consistent dimer and fix a corner matching $\PM_\theta$, then the following holds:
\begin{enumerate}
 \item if $a$ is a boundary lattice point of the matching polygon then $H_{a,\theta}$ is a Casimir in $\C[\cL(\rib)]$,
 \item if $a,b$ are lattice points of the matching polygon $\{H_{a,\theta},H_{b,\theta}\}=0$.
\end{enumerate}
\end{lemma}
\begin{proof}
If  $\cP_1$ and $\cP_2$ lie on the same side of a matching polygon then their difference consists of zigzag walks in the direction perpendicular to that side.
This means that $W_{\cP_1}/W_{\cP_2}$ is a product of holonomy of zigzag walks and hence a Casimir. Going along the border until we meet $\cP_\theta$ we see that
$\frac{W_{\cP_1}}{W_{\cP_\theta}}=\frac{W_{\cP_1}}{W_{\cP_2}}\frac{W_{\cP_2}}{W_{\cP_3}}\dots \frac{W_{\cP_k}}{W_{\cP_\theta}}$ is a product of Casimirs. 

For the second part let $\cP_1$ and $\cP_2$ be two perfect matchings. The union of these two perfect matchings
consists of a collection of joint edges and walks in $\rib$. These walks can be split in walks with nontrivial homology on the torus and
walks with trivial homology. Using this split we can define two new perfect matchings $(\tilde \cP_1,\tilde \cP_2)$,
that are the same as the original with the edges in the trivial walks swapped. This operation is an involution and hence a bijection on the pairs of matchings.

Because we only swapped the trivial walks, the locations of $\tilde \cP_1$ and $\tilde \cP_2$ in the matching polygon are
the same and $W_{\cP_1}W_{\cP_2}=W_{\tilde\cP_1}W_{\tilde\cP_2}$. Furthermore one can check that
\[
\left\{\frac{W_{\cP_1}}{W_{\cP_\theta}},\frac{W_{\cP_2}}{W_{\cP_\theta}}\right\} = - \left\{\frac{W_{\tilde\cP_1}}{W_{\cP_\theta}},\frac{W_{\tilde \cP_2}}{W_{\cP_\theta}}\right\}
\]
For the proof of this identity we refer to \cite{goncharov2013dimers}[Lemma 3.8].

To calculate $\{H_{a,\theta},H_{b,\theta}\}$ we need to sum the brackets of pairs of matchings, but because the swap operation is a bijection it does not matter
whether we sum over $(\cP_1,\cP_2)$ or $(\tilde \cP_1,\tilde \cP_2)$. Therefore
\[
\{H_{a,\theta},H_{b,\theta}\} =\sum_{(\cP_1,\cP_2)} \left\{\frac{W_{\cP_1}}{W_{\cP_\theta}},\frac{W_{\cP_2}}{W_{\cP_\theta}}\right\} =
- \sum_{(\cP_1,\cP_2)}\left\{\frac{W_{\tilde\cP_1}}{W_{\cP_\theta}},\frac{W_{\tilde \cP_2}}{W_{\cP_\theta}}\right\} 
=  - \{H_{a,\theta},H_{b,\theta}\},
\]
and $\{H_{a,\theta},H_{b,\theta}\}=0$.
\end{proof}

\begin{theorem}[Goncharov-Kenyon \cite{goncharov2013dimers} theorem 3.7]
The functions $H_{a,\theta}$ form a maximal set of commuting Hamiltonians on $\cL$, giving $\cL$ the structure of an integrable system.
\end{theorem}
\begin{proof}
The dimension of $\cL$ is $\#\rib_2+1$, there are $\#\bir_2$ zigzag walks so the dimension of the Poisson center is $\#\bir_2-1$ because there is one relation
between the holonomies. The difference $\#\rib_2-\#\rib_2= \chi(\rib)-\chi(\bir)=2g -2$ so the difference between the dimension of $\cL$ and its Poisson center is
$2g$. The maximal number of commuting Hamiltonians (up to multiplication with Casimirs) is $g$. 
Because there are $g$ internal lattice points the Hamiltonians of the internal lattice points form such a maximal set of commuting Hamiltonians, provided
they are independent. The independence is a consequence of lemma \ref{alltropicaloccur}:
the map identifying a point in $\cL(\rib)$ with its tropical curve is surjective
and the dimension of the space of tropical curves is the same as the number of
commuting Hamiltonians.
\end{proof}

\begin{remark}
This construction also has a quantized version. For more details about this
we refer to \cite{goncharov2013dimers}.
\end{remark}

\subsection{Spectral Curves}

The (deformed) partition function $\Part_{\theta}$ plays a special role in our story. We can see it as a map $\Part:\cL \to \C$ and
hence it defines a hypersurface in $\cL$. This hypersurface does not depend on our choice of corner matching because all 
deformed partition functions are scalar multiples of each other.

We also have a map $\pi: \cL \to \cX$ and the fibers of this projection
are $2$-dimensional complex tori coming from the action $\cT={\C^*}^2$. In each of these tori $\Part_{\theta}$ cuts out a hypersurface in this torus.

\begin{definition}
Fix a point $m \in \cM$, the spectral curve $\cC_m$ is the zero locus of $\Part_{\theta}$ in the torus orbit $\cT \cdot m$.
\end{definition}

To get an equation for the spectral curve the fiber that contains $m$, we use coordinate functions $X,Y$ on the torus $\cT =\C^*\times\C^*$ to 
identify the coordinate ring of $\cT\cdot m$ with $\C[X,X^{-1},Y,Y^{-1}]$. This gives the following equation for the curve
\[
 \Part_{\theta,m}(X,Y) = \sum_{a \in \MP}\pm H_{a, \theta}(m) X^{i_a}Y^{j_a},
\]
with $i_a = \deg_{\cP}X-\deg_{\cP_{\theta}}X$ and $j_a = \deg_{\cP}X-\deg_{\cP_{\theta}}Y$ if $\cP \in PM_a$.
If we don't want to specify corner matching we can also use the undeformed version $\Part_m(X,Y):= W_{P_\theta}(m)\Part_{\theta,m}(X,Y)$, which 
gives us the same spectral curve.

\begin{intermezzo}[The Newton polygon]
The \iemph{Newton polygon} $\NP(f)$ of a polynomial $f(X,Y) \in \C[X,X^{-1},Y,Y^{-1}]$ is the convex hull of the exponents of the monomials with nonzero coefficients.
We say that a Newton polygon is degenerate if it is a line segment.

The Newton polygon of a polynomial can be used to determine topological properties of the zero locus $f^{-1}(0)$ in $\C^*\times \C^*$.
If $f^{-1}(0) \subset \C^*\times \C^*$ is a smooth curve then it is a punctured Riemann surface with genus $g$ and $n$ punctures where
$g$ is the number of internal lattice points of the Newton polygon and $n$ the number of boundary lattice points. 
\end{intermezzo}

If we apply this to our situation we get
\begin{lemma}[The spectral curve and the mirror dimer (Feng et al. \cite{feng2008dimer})]
Let $\rib$ be a consistent dimer on a torus and fix a point $m \in \cL(\rib)$.
If the spectral curve $\cC_m$ is smooth then it is homeomorphic to the punctured surface associated to the mirror dimer of $\rib$
\[
 \cC_m \equiv \punct{\rib}.
\]
\end{lemma}
\begin{proof}
We just have to show that the Newton polygon of $\Part_{\theta,m}$ is equal to the matching polygon of $\rib$ and apply theorem \ref{mirdim}.
Clearly the Newton polygon is contained in the matching polygon. It is also equal to the matching polygon because if $a \in \MP(\rib)$ is 
a corner then there is a unique corner matching on $a$. The Hamiltonian $H_{a, \theta}$ consists of just one term and hence it is nonzero.
\end{proof}

\begin{example}
The suspended pinchpoint has a tropical curve with equation of the form
\[
1 + a_1 Y+ a_2Y^2+ a_3X +a_4 XY^{-1}=0 \implies X = f(Y) =-Y\frac{1 + a_1 Y+ a_2Y^2}{a_4+a_3Y}
\]
Because $f$ has $3$ zeros and $2$ poles (one at infinity), the curve is a $\PP_1$ with $5$ punctures. If we compare this with example \ref{sppmir} we see that the mirror
has $5$ vertices and is embedded in a sphere.
\end{example}

\subsection{Dimer models in statistical physics}

Dimer models originally came from statistical physics, where they are used as models to study phase changes.
The idea is that the nodes of the dimer represent particles that tend to form bonds between one black and one white particle. 
The edges of the dimer represent the different possible bonds that a particle can form, 
and a perfect matching is a state in which all the particles are bound. Each bond $e$ can have a certain energy $E_e$ and the total energy of the state
is the sum of all bonds in the perfect matching.

In statistical physics a system is usually described by its partition function $Z$, this is the sum of all states weighed by a factor $e^{-\frac{E}{ kT}}$ where $E$ is
the energy of the state, $T$ is the temperature and $k$ is the Boltzmann constant. The probability that a certain state with energy $E$ occurs at a given temperature $T$ is then the quotient
of $e^{-\frac{E}{kT}}$ by the partition function.

If we do this for the dimer we get
\[
 Z_\rib(E) = \sum_{\PM} e^{-\frac{\sum_{e \in \PM} E_e}{kT}} 
\]
which, if we forget about the signs, can be seen as the partition function of the integral system for a line bundle $\cE \in \cM$ with $X_e(\cE) = e^{- \frac{E_e}{kT}}$. 
The sign issue can be solved by using the polynomial function $\Part_m(X,Y)$:
\begin{lemma}[\cite{kasteleyn,tesler2000matchings}]
If $\rib$ is a dimer on a torus we have 
$$Z_\rib(E) = \frac12 \left| -\Part_\cE(1,1)  + \Part_\cE(-1,1)+\Part_\cE(1,-1)+\Part_\cE(-1,-1)\right|.$$   
\end{lemma}
\begin{proof}
See \cite{kenyon2006dimers}[Section 3.1]. 
\end{proof}
The statistical partition function can be deduced from the dimer partition function where the dimer weight $\cE$ is an exponentiated version
of the energy $E$. Because $\frac{-1}{kT}$ is just a constant that will play no role in our story we will omit it and assume that $\cE=\exp E$.
The partition function gives a probability measure $\mu_E$ on the set of matchings $\PMs(\rib)$. 
\[
\mu_E(U) = \frac{\sum_{\cP \in U} W_{\cP}(\cE)}{Z_\rib(E)} \text{ for all $U\subset \PMs(\rib)$}
\]
Note that because we take ratios between sums of perfect matchings the probability measure only depends on the line bundle in $\cL$ that corresponds to the point $\cE \in \cM$.

The statistical model is a logarithmic version of the integrable system.
The space of energy configurations for the dimer model on the torus is $\rM = \Rl^{\rib_1}$ and on this space we have an additive action of $\rG =\Rl^{\rib_0}$:
$(g\cdot E)_e = E_e+g_{b(e)}-g_{w(e)}$. The quotient space is the moduli space of physically distinguishable configurations $\rL:= \rM/\rG$.
On this moduli space we have an action of $\rT = \Rl^2$ coming from the homology of the torus and the quotient of this action is denoted by $\rX$.
Each of the fibers in this quotient is a two-dimensional space $\Rl^2$.

What is the physical interpretation of this $\rT$-action? If we embed the dimer periodically in the plane such that the fundamental domain is a unit square
then every edge can be represented by a vector $\vec e$ and if $p=e_1\dots e_n$ is a cyclic walk on the torus then the homology class of this walk is
$\vec e_1+\dots +\vec e_n$. If we let $\vec B \in \Rl^2$ act $\rM$ by $E_e \mapsto E_e + \vec B \cdot \vec e$ it is easy to see that this
will coincide with the $\rT$-action on $\rL$ after taking the quotient. This action is analogous to turning on a magnetic field $\vec B$, which then shifts
the energy levels. So $\rX$ classifies the systems up to a variable magnetic field and the fiber $\rT \cdot E\cong \R^2$ parametrizes the different
strengths of the magnetic field of a given system.

Note that for the statistical model the mutation transformation is an isomorphism because
$(1+W)$ never becomes zero. Therefore $\rL(\rib)\cong \rL(\mu(\rib))\cong \rL$. Therefore we will drop the dimer from the notation in this section.

To go from the integrable system to the statistical model, we can use the logarithm of the absolute value
$\Log : \cL \to \rL : \cE \mapsto \log |\cE|$. This gives us a surjection and we can 
look at the image of the complex spectral curve under $\Log$. This image is called the \iemph{amoeba} of the spectral curve and we denote it by $\Am \subset \R^2$. 

\begin{intermezzo}[Amoebae]
Let $P(X,Y) \in \R[X,X^{-1},Y,Y^{-1}]$ be a Laurent polynomial with real coefficients and let $\NP$ be its Newton polygon. 
The amoeba is the two-dimensional closed subset of $\R^2$, containing the logs of the absolute values of the complex solutions of $P(X,Y)$:
\[
 \Am(P) := \{(\rho_1,\rho_2) | \exists \theta_1,\theta_2 \in \Rl: P(e^{\rho_1+i\theta_1},e^{\rho_1+i\theta_1})=0\}.
\]
It looks like a filled set with holes and tentacles that go to infinity. 
\begin{center}
\begin{tikzpicture}
\draw (2.1,-.5) node[above]{$\scriptstyle{\Am(X^2Y^3+X^2Y^2+2X^2Y-5XY^2+3XY-Y^2+4Y+1)}$};
\begin{scope}[y=0.80pt, x=0.80pt, yscale=-.250000, xscale=.250000, inner sep=0pt, outer sep=0pt]
\path[fill=black] (427.8232,443.8979) .. controls (404.3509,400.8440) and
    (386.8043,375.1624) .. (370.5558,360.0798) .. controls (351.9849,342.8416) and
    (331.2634,334.9700) .. (294.3571,331.1338) .. controls (289.1321,330.5907) and
    (237.4946,329.6185) .. (179.6071,328.9733) .. controls (121.7196,328.3282) and
    (74.3571,327.4285) .. (74.3571,326.9741) .. controls (74.3571,325.8541) and
    (207.2976,325.9799) .. (250.8571,327.1411) .. controls (299.3759,328.4346) and
    (315.7043,330.2764) .. (336.0670,336.7527) .. controls (369.0709,347.2495) and
    (389.8710,369.9908) .. (426.0629,435.1479) .. controls (437.6991,456.0969) and
    (438.6740,458.1479) .. (436.9949,458.1479) .. controls (436.2202,458.1479) and
    (432.1135,451.7675) .. (427.8232,443.8979) -- cycle(435.9329,444.8452) ..
    controls (432.4564,437.5287) and (425.4514,422.3412) .. (420.3663,411.0952) ..
    controls (415.2812,399.8492) and (409.9939,388.1729) .. (408.6168,385.1479) ..
    controls (401.1876,368.8290) and (391.2822,341.7810) .. (388.3445,329.7916) ..
    controls (385.8640,319.6685) and (386.1038,317.0805) .. (390.4871,306.6615) ..
    controls (392.9894,300.7138) and (394.3055,298.9973) .. (399.1475,295.3667) ..
    controls (412.4566,285.3872) and (429.8173,279.3789) .. (453.3571,276.6057) ..
    controls (460.2321,275.7957) and (480.4821,274.6931) .. (498.3571,274.1553) ..
    controls (538.2269,272.9558) and (651.3571,272.8319) .. (651.3571,273.9877) ..
    controls (651.3571,274.4496) and (615.6946,275.1399) .. (572.1071,275.5217) ..
    controls (472.8367,276.3912) and (451.6430,277.5457) .. (430.3571,283.2433) ..
    controls (421.0774,285.7272) and (408.7014,291.7072) .. (401.1279,297.3667) ..
    controls (393.7680,302.8665) and (389.9547,311.0243) .. (390.0115,321.1479) ..
    controls (390.0885,334.8705) and (404.5037,371.8048) .. (435.1878,436.8979) ..
    controls (442.1987,451.7708) and (444.7620,458.1479) .. (443.7293,458.1479) ..
    controls (442.8193,458.1479) and (439.8310,453.0491) .. (435.9329,444.8452) --
    cycle(74.3571,324.2491) .. controls (74.3571,323.4574) and (95.6444,323.1436)
    .. (150.1071,323.1327) .. controls (279.7444,323.1067) and (312.8499,321.0045)
    .. (331.6547,311.6050) .. controls (338.2686,308.2990) and (341.7477,304.9758)
    .. (344.3168,299.5104) .. controls (346.8121,294.2017) and (346.8635,291.7032)
    .. (344.6701,282.3256) .. controls (343.0594,275.4394) and (342.5825,274.6525)
    .. (336.6646,269.1167) .. controls (325.8915,259.0390) and (314.9049,254.9720)
    .. (284.2547,249.7156) .. controls (254.0334,244.5328) and (224.1892,243.2056)
    .. (137.1071,243.1721) .. controls (95.9566,243.1562) and (74.3571,242.8037)
    .. (74.3571,242.1479) .. controls (74.3571,240.8338) and (204.6241,240.8348)
    .. (231.8571,242.1489) .. controls (262.4225,243.6240) and (286.5403,246.7756)
    .. (309.2547,252.2627) .. controls (327.8937,256.7653) and (343.7147,267.8303)
    .. (346.7177,278.4641) .. controls (350.4274,291.6004) and (349.7585,297.4877)
    .. (343.5961,305.9355) .. controls (337.0694,314.8827) and (320.1139,320.1966)
    .. (288.3571,323.2473) .. controls (272.0280,324.8159) and (74.3571,325.7411)
    .. (74.3571,324.2489) -- cycle(494.3571,271.4753) .. controls
    (463.0252,270.6261) and (436.2705,267.6540) .. (420.7341,263.2965) .. controls
    (392.1443,255.2781) and (366.4641,237.9616) .. (344.8813,212.1479) .. controls
    (336.2134,201.7809) and (307.9774,158.3040) .. (296.1727,137.1479) .. controls
    (294.7917,134.6729) and (290.9367,127.9229) .. (287.6061,122.1479) .. controls
    (270.8314,93.0619) and (261.5158,75.1482) .. (263.1749,75.1674) .. controls
    (264.2917,75.1803) and (266.9619,79.1299) .. (271.1187,86.9174) .. controls
    (274.5625,93.3692) and (281.2318,105.3979) .. (285.9393,113.6479) .. controls
    (290.6468,121.8979) and (295.6349,130.6729) .. (297.0241,133.1479) .. controls
    (305.8695,148.9076) and (330.9062,188.5766) .. (340.1971,201.5530) .. controls
    (348.2221,212.7613) and (364.0666,229.0402) .. (374.7437,237.0470) .. controls
    (401.9116,257.4202) and (425.0915,265.0486) .. (469.3571,268.1838) .. controls
    (477.3321,268.7487) and (521.5446,269.4876) .. (567.6071,269.8259) .. controls
    (613.6696,270.1641) and (651.3571,270.8250) .. (651.3571,271.2944) .. controls
    (651.3571,272.0343) and (520.5444,272.1851) .. (494.3571,271.4753) --
    cycle(341.3840,253.6803) .. controls (339.3045,251.1746) and
    (338.9852,244.7678) .. (340.7681,241.3201) .. controls (343.9046,235.2549) and
    (351.2898,233.1668) .. (353.7563,237.6479) .. controls (354.3618,238.7479) and
    (355.2650,239.9430) .. (355.7635,240.3036) .. controls (356.9490,241.1611) and
    (354.7464,249.0163) .. (352.5455,251.7801) .. controls (350.3609,254.5234) and
    (343.1085,255.7581) .. (341.3840,253.6803) -- cycle(350.6969,249.7521) ..
    controls (353.0654,246.1373) and (353.5217,242.9250) .. (352.0611,240.1479) ..
    controls (349.4041,235.0963) and (342.3571,240.0992) .. (342.3571,247.0370) ..
    controls (342.3571,251.5375) and (342.8410,252.1479) .. (346.4088,252.1479) ..
    controls (348.1733,252.1479) and (349.6779,251.3073) .. (350.6969,249.7521) --
    cycle(74.3571,238.7589) .. controls (74.3571,237.2940) and (80.3215,237.1467)
    .. (140.1071,237.1343) .. controls (238.6281,237.1139) and (262.8293,235.3685)
    .. (279.3571,227.0913) .. controls (286.2856,223.6215) and (289.5070,220.2237)
    .. (293.0506,212.6479) .. controls (295.4968,207.4182) and (295.8494,205.4924)
    .. (295.7965,197.6479) .. controls (295.6701,178.9044) and (286.2736,148.9696)
    .. (265.6623,101.6479) .. controls (259.9130,88.4479) and (254.9691,77.0854)
    .. (254.6761,76.3979) .. controls (254.3830,75.7104) and (254.7962,75.1479) ..
    (255.5944,75.1479) .. controls (256.5168,75.1479) and (260.0634,81.9814) ..
    (265.3259,93.8979) .. controls (290.6173,151.1689) and (301.2660,187.7491) ..
    (297.8570,205.6479) .. controls (293.5495,228.2639) and (279.0348,235.4490) ..
    (232.3571,238.0719) .. controls (222.7321,238.6127) and (183.2446,239.3510) ..
    (144.6071,239.7126) .. controls (80.4061,240.3133) and (74.3571,240.2312) ..
    (74.3571,238.7589) -- cycle;
\end{scope}   
\end{tikzpicture}
\end{center}
The complement of the amoeba splits in a number 
of connected components. Some of these components come from the holes in the amoeba and others from the space between tentacles. One should think of each component
as a subset where one of the monomials of $P$ becomes too big in absolute value to be cancelled out by the others so $P$ cannot become zero. This suggests that we must be able
to assign a monomial to each of these components. This can be done using the \iemph{Ronkin function}.
\[
 F(\rho_1,\rho_2) := \frac{1}{(2 \pi)^2} \int_0^{2\pi}\int_{0}^{2\pi} \log |P(e^{\rho_1+i\theta_1},e^{\rho_1+i\theta_1})| d\theta_1d\theta_2
\]
The gradient of this function is constant on each component in the complement of the amoeba and it equals a lattice point in the Newton polygon. 

For a good survey on amoeba we refer to \cite{mikhalkin2004amoebas}.
\end{intermezzo}

\subsection{Phase transitions of a dimer model}
We will now give a physical interpretation to both $\Am$ and the Ronkin function. We will only give a rough sketch of the main ideas in this subject area. For the details we refer
to \cite{kenyon2006dimers,kenyon2006planar,kenyon2003introduction}

The partition function describes the dimer on a torus but it is more interesting from the point of view of solid state physics to
work in the universal cover $\tilde\rib$. This is a periodic graph and can be interpreted as a two-dimensional crystal.

If we want to construct a measure on the set of states, $\PMs(\tilde \rib)$, we can do this by taking larger and larger covers of $\rib$.
Let $\rib^n$ be the $n\times n$-cover of $\rib$ and denote by $E$ the energy function that assigns to all edges in $\rib^n$ lying over an edge $e \in \rib_1$ the energy $E_e$.
This gives a measure $\mu^n(E)$ on $\PMs(\rib^n)$ and we can take the limit $n \to \infty$ to obtain a measure on $\PMs(\tilde \rib)$.
This measure is an example of an \iemph{ergodic Gibbs measure} (EGM). For the precise
definition of an ergodic Gibbs measure we refer to \cite{kenyon2006dimers}.

Energy functions that correspond to different measures on $\PMs(\rib)$ can lead to identical EGMs on $\PMs(\tilde\rib)$. 
Scott Sheffield \cite{sheffield2006random} showed that EGMs on a dimer model are parametrized by their slopes.
\begin{definition}
The \iemph{slope of an EGM} is the expectation value of $(\<\bar X,\PM\>,\<\bar Y,\PM\>)$ where $X,Y$ are the lifts of paths with homology classes that form a basis for the torus.
\end{definition}

\begin{theorem}[Sheffield \cite{sheffield2006random}/Kenyon-Okounkov-Sheffield \cite{kenyon2006dimers}(theorem 2.1 and 3.2)]
The slope of an EGM lies inside the matching polygon $\MP(\rib)$.
Two EGMs with the same slope are the same and for each point $p$ in the matching polygon there is a unique EGM with slope $p$.  
\end{theorem}
This gives a map $\psi:\rM \to \MP$ that assigns to each energy configuration the slope of the corresponding EGM. 
Intuitively the slope will be the location in the matching polygon where the partition function becomes concentrated.

As we have seen we can split the partition function $Z_E(\rib)$ in parts corresponding to the lattice points in the matching polygon. In
general there will be one lattice point that contributes the most. If we go to a cover $\rib^n$, the new matching polygon has the same shape but is
scaled by a factor $n$, or if we scale back we could say that the matching polygon stays the same but the lattice points become more densely distributed.
So for each $n$ we have a rational point $p_n$ inside the matching polygon $\MP(\rib)$ that contributes most to $Z(\rib^n)$. 
The limit of these points $p_\infty = \lim_{n \to \infty}p_n$ will be the slope of the EGM because the perfect matchings on this
location will dominate in the expectation value of $(\<\bar X,\PM\>,\<\bar Y,\PM\>)$.

Building on this idea we can construct a function that measures the contributions of each point in the limit of the partition function.
\[
 \sigma_E: \MP(\rib) \to \R : u \mapsto \lim_{n\to \infty} (Z_{\floor{nu}}(\rib^n))^{\frac 1{n^2}}
\]
In this notation $Z_{\floor{nu}}(\rib_n))$ stands for the part of the partition function $Z(\rib^n)$ coming from the lattice point just to the left and below $nu$ 
in $\MP(\rib^n)$.
The exponent ${\frac 1{n^2}}$ is needed because the fundamental domain enlarges by a factor $n^2$, so if $\cP$ is a perfect matching and $\cP^n$ its lift to the cover
then $W_{\cP^n}=(W_{\cP})^{n^2}$. The function $\sigma_E$ is called the \iemph{surface tension} and the slope of the EGM is the point
in $\MP(\rib)$ where $\sigma$ becomes maximal.

To find the slope of an energy configuration it is interesting to restrict $\psi$ to one $\rT$-orbit.
If $E$ is an energy configuration then $\psi_E: \R^2 \to \MP: \vec B \mapsto \psi(E_e + \vec B \cdot \vec e)$ will indicate how the slope behaves under 
changes of the magnetic field. The behaviour of $\sigma_E$ and $\psi_E$ can be expressed in terms of the partition function $\Part_E(X,Y)$ of the dimer. 

\begin{theorem}[\cite{kenyon2006dimers} theorem 3.6]
Let $\rib$ be a consistent dimer on a torus.
\begin{itemize}
 \item 
The surface tension $\sigma_E$ is the Legendre dual of the Ronkin function $F_E$ associated to $P_E$:
\[
\forall (u_1,u_2) \in \MP(\rib): \sigma_E(u_1,u_2) = \sup_{(x,y)} F_E(x,y) - (u_1x +u_2y).
\]
\item
On the complement of the amoeba $\Am(\Part_E)$ $\psi_E$ is equal to the gradient of the Ronkin function. This means $\psi_E$
maps each component of complement to its corresponding lattice point in the Newton/matching polygon. 
\end{itemize}
\end{theorem}

So we see that there is a qualitative difference in the behaviour of the system depending on whether you are inside
the amoeba our outside. If you are inside and you change the magnetic field a little bit, the slope will also vary while
if you are outside the slope will remain constant.
\newcommand{\vX}{\mathsf{X}}
\newcommand{\vY}{\mathsf{Y}}

There is also a qualitative difference in the EGMs as well. In fact we can make a distinction between three different types of behaviour.
Depending on the behavior of the variable $\vX=\<\bar X,\cP\>$ and $\vY=\<\bar Y,\cP\>$. 
We say that an EGM is 
\begin{enumerate}
 \item \iemph{frozen}, if there are directions $(n,m)$ with $\sqrt{n^2+m^2}$ arbitrarily large
such that $n\mathsf{X}+m\mathsf{Y}$ is a deterministic variable,
 \item \iemph{gaseous}, if $\Var(n\vX+m\vY)$ is bounded for arbitrary large $n,m$,
 \item \iemph{liquid}, if $\Var(n\vX+m\vY)$ is unbounded for arbitrary large $n,m$.
\end{enumerate}

\begin{theorem}[\cite{kenyon2006dimers} theorem 4.1]
Let $E$ be an energy configuration and $B \in\R^2$, the EGM corresponding to $(E_e+B\cdot \vec e)_e$ is 
frozen if $B$ is in an unbounded component of the complement of the amoeba $\Am(P_E)$, gaseous if $B$ is in a bounded component of the complement of the amoeba and
liquid if $B$ is inside the amoeba.
\end{theorem}

\begin{intermezzo}[Harnack curves]
Harnack curves are special real curves that can be characterized in many different ways \cite{mikhalkin2000real}.  
From our point of view these are curves that have nice amoebae. A curve defined by the real Laurent polynomial $f(X,Y) \in \R[X^{\pm 1},Y^{\pm1}]$ is
called \iemph{Harnack} if the map $\Log: \cC_f \to \Am(\cC_f)$ is at most $2\to 1$. This implies many nice features of the amoeba and the real curve.
First of all the number of components of the real curve is maximal, but some of them can be just an isolated point. These components are in $1\to 1$-correspondence 
with the lattice points of the Newton polygon. The real points are also the points of the curve where the $\Log$-map is $1\to 1$. 
It is also easy to determine the genus of a Harnack curve, this is just the number of holes in the amoeba.

The Harnack curves are those curves for which the area of the amoeba is maximal. If the Newton polygon of $f$ has area $A$ (this is half of the elementary area) 
Mikhalkin and Rullg{\aa}rd show that the area of $\Am(f)$ is at most $\pi A$ and the equality is reached if and only if $f$ defines a Harnack curve \cite{mikhalkin2001amoebas}.

Given a fixed lattice polygon $\LP$ we can look at the space of all real polynomials $f$ with a nondegenerate Newton polygon contained in
$\LP$. This space forms an open subset of $\R^m$ where $m$ is the number of lattice points in $\LP$. Because of the $2\to1$-property, the subset of all Harnack curves forms a closed subset of this space. 
\end{intermezzo}

\begin{theorem}[\cite{kenyon2006dimers} theorem 5.2]\label{alwaysHarnack}
Let $\rib$ be a dimer on a torus and $E:\rib \to \R$ be any energy function.
The partition function $\Part_E(X,Y)$ defines a Harnack curve.
\end{theorem}
\begin{proof}
The idea of the proof of the first part, is to show it holds for a special set of dimers coming from the tiling of the plane with hexagons.
This is done using the method of transfer matrices and certain total positivity results. 
For the other dimers we let the edge weights $\exp(E_e)$ of certain edges go to zero. This is the same as deleting this edge in the dimer and 
we can obtain each dimer in this way. Because the Harnack condition is closed the limit partition function $\Part_E(X,Y)$ is still Harnack. 
\end{proof}

\begin{example}\label{hexdim}
Consider the following dimer quiver with matching polygon a hexagon with perimeter $6$ and $2$ internal lattice points.
The weights of the edges are $1$ for the even arrows and $2$ for the odd. 
The homology classes are 
$$
\bar X=-a_{11}+a_{14}-a_{16}+a_{17}-a_{19}+a_{20}+a_{22} \text{ and } \bar Y = -a_{16}+a_{17}.
$$
The Kasteleyn matrix is
\begin{center}
\begin{tikzpicture}
\draw[thick] (-1,0) -- (0,1) -- (1,1) -- (2,0)-- (1,-1) -- (0,-1) -- (-1,0);
\draw (-1,0) node [circle,draw,fill=white,minimum size=15pt,inner sep=1pt] {{\tiny -32}};
\draw (0,1) node [circle,draw,fill=white,minimum size=15pt,inner sep=1pt] {{\tiny -128}};
\draw (0,-1) node [circle,draw,fill=white,minimum size=15pt,inner sep=1pt] {{\tiny -4}};
\draw (0,0) node [circle,draw,fill=white,minimum size=15pt,inner sep=1pt] {{\tiny -264}};
\draw (1,-1) node [circle,draw,fill=white,minimum size=15pt,inner sep=1pt] {{\tiny 1}};
\draw (1,1) node [circle,draw,fill=white,minimum size=15pt,inner sep=1pt] {{\tiny 32}};
\draw (1,0) node [circle,draw,fill=white,minimum size=15pt,inner sep=1pt] {{\tiny 84}};
\draw (2,0) node [circle,draw,fill=white,minimum size=15pt,inner sep=1pt] {{\tiny -2}};
\draw (-2,0) node[left] {$\left(\begin{smallmatrix}
2& 2/x& 1& 1& 0& 0& 0 \\ 0& -1& 1& 0& 2& 0& 0 \\ -x& 0& 2& 0& 0& 0& x \\ 2& 0& 0& -1& 0& 1/(xy)& 0 \\ 0& 0& -2& 2& 0& 0& x\\
0& 1& 0& 0& 2& 2& 0 \\ 0& 0& 0& 0& 1& -2/x& 2xy
\end{smallmatrix}\right)$};
\end{tikzpicture}
\end{center}
and the partition function is
\[
P(x,y)=-2X^2+3XY-84X-128Y+XY^{-1}-264-4Y^{-1}-32X^{-1}.
\]
\begin{center}
\begin{tikzpicture}
\begin{scope}[scale=.33]
\draw[dotted] (-10pt,20pt) rectangle (290pt,320pt);
\draw [-latex,shorten >=5pt] (20pt,320pt) to node [rectangle,draw,fill=white,sloped,inner sep=1pt] {{\tiny 1}} (-7pt,245pt); 
\draw [-latex,shorten >=5pt] (20pt,20pt) to node [rectangle,draw,fill=white,sloped,inner sep=1pt] {{\tiny 3}} (-7pt,95pt); 
\draw (20pt,170pt) -- (-10pt,170pt); 
\draw [-latex,shorten >=5pt] (20pt,20pt) to node [rectangle,draw,fill=white,sloped,inner sep=1pt] {{\tiny 2}} (102pt,20pt); 
\draw [-latex,shorten >=5pt] (20pt,320pt) to node [rectangle,draw,fill=white,sloped,inner sep=1pt] {{\tiny 2}} (102pt,320pt); 
\draw [-latex,shorten >=5pt] (-7pt,245pt) to node [rectangle,draw,fill=white,sloped,inner sep=1pt] {{\tiny 4}} (20pt,170pt); 
\draw [-latex,shorten >=5pt] (20pt,170pt) to node [rectangle,draw,fill=white,sloped,inner sep=1pt] {{\tiny 6}} (102pt,170pt); 
\draw [-latex,shorten >=5pt] (295pt,170pt) to node [rectangle,draw,fill=white,sloped,inner sep=1pt] {{\tiny 7}} (184pt,170pt); 
\draw [-latex,shorten >=5pt] (102pt,320pt) to node [rectangle,draw,fill=white,sloped,inner sep=1pt] {{\tiny 8}} (102pt,170pt); 
\draw [-latex,shorten >=5pt] (102pt,20pt) to node [rectangle,draw,fill=white,sloped,inner sep=1pt] {{\tiny 9}} (102pt,170pt); 
\draw [-latex,shorten >=5pt] (102pt,170pt) to node [rectangle,draw,fill=white,sloped,inner sep=1pt] {{\tiny 10}} (20pt,20pt); 
\draw [-latex,shorten >=5pt] (102pt,170pt) to node [rectangle,draw,fill=white,sloped,inner sep=1pt] {{\tiny 12}} (184pt,20pt); 
\draw [-latex,shorten >=5pt] (184pt,20pt) to node [rectangle,draw,fill=white,sloped,inner sep=1pt] {{\tiny 15}} (102pt,20pt); 
\draw [-latex,shorten >=5pt,yshift=300pt] (184pt,20pt) to node [rectangle,draw,fill=white,sloped,inner sep=1pt] {{\tiny 15}} (102pt,20pt); 
\draw [-latex,shorten >=5pt] (184pt,320pt) to node [rectangle,draw,fill=white,sloped,inner sep=1pt] {{\tiny 16}} (184pt,170pt); 
\draw [-latex,shorten >=5pt] (184pt,20pt) to node [rectangle,draw,fill=white,sloped,inner sep=1pt] {{\tiny 17}} (184pt,170pt); 
\draw [-latex,shorten >=5pt] (184pt,170pt) to node [rectangle,draw,fill=white,sloped,inner sep=1pt] {{\tiny 18}} (293pt,245pt); 
\draw [-latex,shorten >=5pt] (184pt,170pt) to node [rectangle,draw,fill=white,sloped,inner sep=1pt] {{\tiny 19}} (102pt,170pt); 
\draw [-latex,shorten >=5pt] (184pt,170pt) to node [rectangle,draw,fill=white,sloped,inner sep=1pt] {{\tiny 20}} (293pt,95pt); 
\draw [-latex,shorten >=5pt] (-7pt,95pt) to node [rectangle,draw,fill=white,sloped,inner sep=1pt] {{\tiny 21}} (20pt,170pt); 
\draw [-latex,shorten >=5pt] (293pt,95pt) to node [rectangle,draw,fill=white,sloped,inner sep=1pt] {{\tiny 22}} (184pt,20pt); 
\draw [-latex,shorten >=5pt] (102pt,170pt) to node [rectangle,draw,fill=white,sloped,inner sep=1pt] {{\tiny 11}} (20pt,320pt); 
\draw [-latex,shorten >=5pt] (102pt,170pt) to node [rectangle,draw,fill=white,sloped,inner sep=1pt] {{\tiny 13}} (184pt,320pt); 
\draw [-latex,shorten >=5pt] (293pt,245pt) to node [rectangle,draw,fill=white,sloped,inner sep=1pt] {{\tiny 5}} (184pt,320pt); 
\node at (20pt,20pt) [circle,draw,fill=white,minimum size=10pt,inner sep=1pt] {\mbox{\tiny $1$}}; 
\node at (20pt,320pt) [circle,draw,fill=white,minimum size=10pt,inner sep=1pt] {\mbox{\tiny $1$}}; 
\node at (-7pt,245pt) [circle,draw,fill=white,minimum size=10pt,inner sep=1pt] {\mbox{\tiny $2$}}; 
\node at (293pt,245pt) [circle,draw,fill=white,minimum size=10pt,inner sep=1pt] {\mbox{\tiny $2$}}; 
\node at (-7pt,245pt) [circle,draw,fill=white,minimum size=10pt,inner sep=1pt] {\mbox{\tiny $2$}}; 
\node at (293pt,245pt) [circle,draw,fill=white,minimum size=10pt,inner sep=1pt] {\mbox{\tiny $2$}}; 
\node at (-7pt,245pt) [circle,draw,fill=white,minimum size=10pt,inner sep=1pt] {\mbox{\tiny $2$}}; 
\node at (293pt,245pt) [circle,draw,fill=white,minimum size=10pt,inner sep=1pt] {\mbox{\tiny $2$}}; 
\node at (20pt,170pt) [circle,draw,fill=white,minimum size=10pt,inner sep=1pt] {\mbox{\tiny $3$}}; 
\node at (102pt,20pt) [circle,draw,fill=white,minimum size=10pt,inner sep=1pt] {\mbox{\tiny $4$}}; 
\node at (102pt,320pt) [circle,draw,fill=white,minimum size=10pt,inner sep=1pt] {\mbox{\tiny $4$}}; 
\node at (102pt,170pt) [circle,draw,fill=white,minimum size=10pt,inner sep=1pt] {\mbox{\tiny $5$}}; 
\node at (184pt,20pt) [circle,draw,fill=white,minimum size=10pt,inner sep=1pt] {\mbox{\tiny $6$}}; 
\node at (184pt,320pt) [circle,draw,fill=white,minimum size=10pt,inner sep=1pt] {\mbox{\tiny $6$}}; 
\node at (184pt,170pt) [circle,draw,fill=white,minimum size=10pt,inner sep=1pt] {\mbox{\tiny $7$}}; 
\node at (-7pt,95pt) [circle,draw,fill=white,minimum size=10pt,inner sep=1pt] {\mbox{\tiny $8$}}; 
\node at (293pt,95pt) [circle,draw,fill=white,minimum size=10pt,inner sep=1pt] {\mbox{\tiny $8$}}; 
\node at (-7pt,95pt) [circle,draw,fill=white,minimum size=10pt,inner sep=1pt] {\mbox{\tiny $8$}}; 
\node at (293pt,95pt) [circle,draw,fill=white,minimum size=10pt,inner sep=1pt] {\mbox{\tiny $8$}}; 
\node at (-7pt,95pt) [circle,draw,fill=white,minimum size=10pt,inner sep=1pt] {\mbox{\tiny $8$}}; 
\node at (293pt,95pt) [circle,draw,fill=white,minimum size=10pt,inner sep=1pt] {\mbox{\tiny $8$}}; 
\end{scope}
\begin{scope}[xshift=4cm,yshift=7.3cm]
\begin{scope}[y=0.80pt, x=0.80pt, yscale=-.50000, xscale=.50000, inner sep=0pt, outer sep=0pt]
  \path[fill=black] (142.3439,494.6836) .. controls (183.9194,452.4073) and
    (194.5342,440.0992) .. (202.7226,424.6733) .. controls (207.4896,415.6931) and
    (209.3337,409.9605) .. (211.1808,398.3803) .. controls (212.4387,390.4947) and
    (212.4391,388.2375) .. (211.1838,380.8803) .. controls (206.3558,352.5699) and
    (199.3583,342.7336) .. (144.8864,287.6859) .. controls (121.8181,264.3739) and
    (112.1064,253.9336) .. (113.4894,253.9336) .. controls (115.8286,253.9336) and
    (175.8792,314.1688) .. (187.9613,328.6344) .. controls (207.3078,351.7975) and
    (214.1567,367.8923) .. (214.0463,389.9336) .. controls (213.9090,417.3385) and
    (203.7320,435.3840) .. (167.1394,473.1072) .. controls (161.4449,478.9776) and
    (156.7857,484.0921) .. (156.7857,484.4726) .. controls (156.7857,484.8532) and
    (163.2587,478.8874) .. (171.1702,471.2154) .. controls (190.2184,452.7438) and
    (201.3761,443.8455) .. (212.2857,438.4258) .. controls (229.9697,429.6407) and
    (244.8382,430.6080) .. (254.3749,441.1640) .. controls (259.1715,446.4733) and
    (261.3625,451.6689) .. (263.8283,463.5816) .. controls (265.5570,471.9334) and
    (265.7749,475.6036) .. (265.2592,487.6836) .. controls (264.9006,496.0840) and
    (265.0267,501.9336) .. (265.5665,501.9336) .. controls (266.0701,501.9336) and
    (266.7426,495.2961) .. (267.0610,487.1836) .. controls (268.2162,457.7518) and
    (274.1927,442.4487) .. (285.3494,440.3557) .. controls (295.8566,438.3845) and
    (306.5523,442.0632) .. (320.2857,452.3716) .. controls (335.1683,463.5427) and
    (380.7687,507.9336) .. (377.3616,507.9336) .. controls (376.2499,507.9336) and
    (367.5890,500.1890) .. (356.3772,489.1694) .. controls (327.7736,461.0562) and
    (317.1362,452.1125) .. (304.9529,445.9331) .. controls (290.8288,438.7693) and
    (281.1095,441.2148) .. (275.3542,453.3804) .. controls (271.6679,461.1726) and
    (269.1912,478.4752) .. (268.8693,498.6836) .. controls (268.7873,503.8317) and
    (268.7275,503.9336) .. (265.7857,503.9336) -- (262.7857,503.9336) --
    (262.7857,486.9990) .. controls (262.7857,462.8176) and (260.1647,451.7864) ..
    (252.3691,443.1575) .. controls (241.1021,430.6864) and (222.7569,432.4460) ..
    (201.7948,448.0085) .. controls (192.7826,454.6993) and (179.4634,466.8221) ..
    (155.2857,490.3401) .. controls (138.1587,506.9998) and (137.0924,507.8510) ..
    (133.2996,507.8913) -- (129.3135,507.9336) -- (142.3439,494.6836) --
    cycle(360.1539,487.6201) .. controls (338.6082,465.5915) and
    (325.7998,451.2081) .. (318.4401,440.7771) .. controls (304.3930,420.8678) and
    (298.5798,395.1101) .. (303.2417,373.4336) .. controls (308.2225,350.2737) and
    (318.7885,334.2986) .. (351.1725,300.9653) .. controls (364.5429,287.2029) and
    (362.7621,288.1661) .. (347.2888,303.0659) .. controls (319.3644,329.9551) and
    (304.0196,339.9901) .. (290.9466,339.9116) .. controls (285.8061,339.8808) and
    (279.6706,336.8305) .. (276.1444,332.5526) .. controls (269.9976,325.0955) and
    (268.1413,314.1918) .. (266.9307,278.4336) -- (266.1182,254.4336) --
    (265.9520,280.5042) .. controls (265.7644,309.9209) and (264.3022,320.2882) ..
    (258.7389,331.6480) .. controls (248.6970,352.1524) and (224.3881,352.0643) ..
    (197.2857,331.4251) .. controls (181.3803,319.3127) and (114.6700,253.9336) ..
    (118.2164,253.9336) .. controls (119.3801,253.9336) and (131.3138,264.9990) ..
    (148.7070,282.2058) .. controls (189.8662,322.9236) and (200.9469,332.4688) ..
    (215.1447,339.4367) .. controls (234.0249,348.7027) and (248.7417,345.3098) ..
    (256.5578,329.8893) .. controls (262.3345,318.4924) and (263.2386,311.8718) ..
    (263.2097,281.1836) -- (263.1841,253.9336) -- (266.3824,253.9336) --
    (269.5807,253.9336) -- (270.0332,282.1836) .. controls (270.4334,307.1675) and
    (270.7367,311.3576) .. (272.6567,318.4256) .. controls (275.1842,327.7302) and
    (278.5235,332.4681) .. (284.8436,335.7171) .. controls (289.7855,338.2575) and
    (293.6063,338.0878) .. (301.9193,334.9585) .. controls (313.1137,330.7445) and
    (324.2630,321.4242) .. (362.1911,284.5742) .. controls (388.2837,259.2234) and
    (394.5325,254.2859) .. (393.6149,259.7442) .. controls (393.4338,260.8216) and
    (381.7563,273.5674) .. (367.6648,288.0683) .. controls (321.0616,336.0255) and
    (311.7098,348.7155) .. (306.1382,371.5573) .. controls (303.6685,381.6823) and
    (303.9126,399.1934) .. (306.6662,409.4336) .. controls (312.5540,431.3289) and
    (323.0820,445.7148) .. (363.1579,486.6262) .. controls (377.8951,501.6708) and
    (383.4410,507.9336) .. (382.0261,507.9336) .. controls (380.7386,507.9336) and
    (372.9185,500.6708) .. (360.1539,487.6201) -- cycle(228.7377,412.6802) ..
    controls (224.3281,409.7086) and (222.0223,405.5635) .. (220.3605,397.6211) ..
    controls (219.1899,392.0262) and (219.2041,389.9903) .. (220.4600,383.3387) ..
    controls (222.9330,370.2418) and (226.5636,365.3096) .. (234.5167,364.2429) ..
    controls (244.9845,362.8388) and (253.2539,369.0801) .. (256.9207,381.1522) ..
    controls (258.8952,387.6529) and (258.9850,388.9061) .. (257.8645,394.3309) ..
    controls (255.7826,404.4104) and (250.9855,411.4085) .. (244.4706,413.8700) ..
    controls (239.5907,415.7137) and (232.4362,415.1727) .. (228.7376,412.6802) --
    cycle(242.4706,411.8700) .. controls (248.5082,409.5888) and
    (253.0200,403.5152) .. (255.4134,394.4470) .. controls (258.1733,383.9899) and
    (250.7254,369.6035) .. (241.2306,367.0514) .. controls (236.9748,365.9075) and
    (231.9134,366.7528) .. (229.1157,369.0747) .. controls (226.5752,371.1831) and
    (223.9342,377.3972) .. (222.6757,384.2274) .. controls (221.6429,389.8324) and
    (221.7099,391.7908) .. (223.1312,397.5504) .. controls (225.0993,405.5257) and
    (226.7693,408.3431) .. (231.0900,410.9776) .. controls (234.6738,413.1628) and
    (238.2970,413.4469) .. (242.4706,411.8700) -- cycle(277.3927,409.8391) ..
    controls (273.8928,407.0861) and (271.2329,401.7189) .. (269.7069,394.3309) ..
    controls (268.5891,388.9192) and (268.6778,387.6477) .. (270.6182,381.2593) ..
    controls (272.3111,375.6856) and (273.6197,373.4129) .. (276.7796,370.5581) ..
    controls (280.9048,366.8313) and (283.6773,366.1815) .. (288.1683,367.8890) ..
    controls (291.7391,369.2466) and (293.4670,372.3783) .. (295.7859,381.6953) --
    (297.8167,389.8549) -- (295.7473,397.5769) .. controls (293.4294,406.2264) and
    (293.2924,406.5402) .. (290.4626,409.6836) .. controls (287.7500,412.6968) and
    (281.1258,412.7756) .. (277.3927,409.8391) -- cycle(288.5776,407.6551) ..
    controls (293.8425,402.0509) and (295.6887,387.9844) .. (292.4302,378.3012) ..
    controls (289.4148,369.3405) and (284.2681,367.1980) .. (278.6348,372.5586) ..
    controls (274.8447,376.1652) and (274.1393,377.4786) .. (272.3950,384.1753) ..
    controls (271.2152,388.7048) and (271.1996,390.3999) .. (272.2988,394.6374) ..
    controls (274.0363,401.3355) and (276.1703,405.3044) .. (279.4006,407.8453) ..
    controls (282.8458,410.5553) and (285.9127,410.4918) .. (288.5776,407.6551) --
    cycle;
\end{scope}
\end{scope}
\end{tikzpicture}
\end{center}
\end{example}

Kenyon and Okounkov also proved a statement in the reverse direction of theorem \ref{alwaysHarnack}.
\begin{theorem}[\cite{kenyon2006planar}]\label{Harnackfromdimer}
For any Harnack curve $\cC_f$ one can find a dimer and energy function such that the spectral curve $\cC :\Part_E(X,Y)=0$ is isomorphic to $\cC_f$.
\end{theorem}
The identification between Harnack curves and dimers with energy functions is not one to one. Different dimers and energy functions can give rise
to the same Harnack curve. The correspondance can be turned into a bijection by adding extra data to the Harnack curve. This is further explained
in \cite{kenyon2006planar}. In fact if $\NP(f)=\MP(\qpol)$ then it is always possible to find an energy function that does the trick.

\begin{example}
Consider the suspended pinchpoint and number the arrows $a,\dots,g$ as in example  \ref{sppmut}.
For the Kasteleyn weight we assign $-1$ to $d$. The homology classes are $X=[a]$
and $Y=[b]-[e]+[g]$. This gives
\[
\Kast = \begin{pmatrix}
         aX &bY+c\\
f+gY&-d+eY^{-1}
         \end{pmatrix}
\]
and
\[
 \Part = -fc -(fb+gc)Y-gbY^2 -adX+aeXY^{-1}
\]
If we chose $a,c,e,b,g=1$ we get
\[
 Y^2+sY+t+X-XY^{-1}=0
\]
with $s = f+c$ and $p=fc$.

In general, up to a factor, every polynomial with that Newton polygon can be brought into this form by rescaling $X,Y$ with real numbers.

Depending on the roots of $Z^2-sZ+p$ there are $4$ main situations.
\begin{center}
\begin{tabular}{|c|c|c|c|}
\hline
$z_1,z_2>0$&$z_1>0>z_2$&$0>z_1,z_2$&$z_1=\bar z_2$\\ 
\hline
\resizebox{2cm}{!}{
\includegraphics{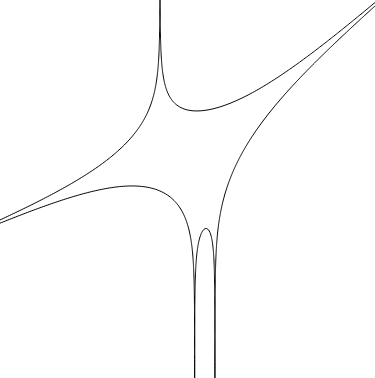}
}&
\resizebox{2cm}{!}{
\includegraphics{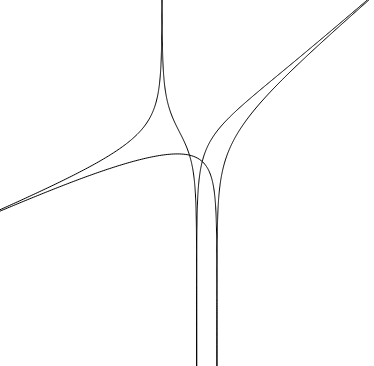}
}&
\resizebox{2cm}{!}{
\includegraphics{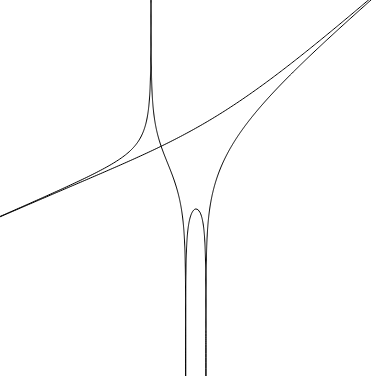}
}&
\resizebox{2cm}{!}{
\includegraphics{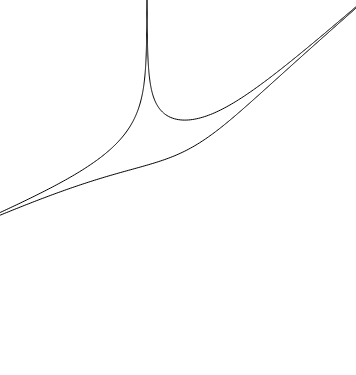}
}
\\
\hline
\end{tabular}
\end{center}
Only in the first case is the curve Harnack.
\end{example}

The theorem that all real spectral curves are Harnack curves has some interesting physical consequences.
\begin{itemize}
 \item The phase transitions occur on the image of the real locus of the spectral curve.
 \item If a gaseous phase is not present its real component has shrunk to an isolated point.
 \item The number of gaseous phases equals the genus of the curve.
\end{itemize}

The most generic situation happens when all real components are curves, not isolated points. The most singular
situation happens when all bounded real components have shrunk to points. In that case the genus of the spectral curve is zero.
In \cite{kenyon2006planar} Kenyon and Okounkov give a nice characterization of this situation.

\begin{theorem}
Suppose $(\rib,\qpol)$ is a dimer and $\qpol$ is isoradially embedded. If we assign to each edge a weight proportial to the length of its arrows then
the spectral curve has genus zero.
\end{theorem}

\begin{example}
We continue example \ref{hexdim} of the dimer with the hexagonal matching polygon.
There are $6$ zigzag paths and we assign them sleeper directions $0^\circ,60^{\circ},\dots,300^\circ$.
This gives us $3$ types of arrows: arrows that cover an arc length of $60^\circ, 120^\circ$ or $180^\circ$.
In the isoradially embedded dimer quiver they have length $1, \sqrt 3$ or $2$. The weight vector and the amoeba look as follows.
\begin{center}
\begin{tikzpicture}
[y=0.80pt, x=0.80pt, yscale=-.2500000, xscale=.2500000, inner sep=0pt, outer sep=0pt]
  \path[fill=black] (131.6648,660.1836) .. controls (183.4480,608.0549) and
    (204.2552,586.2283) .. (221.5127,565.9336) .. controls (253.0132,528.8892) and
    (264.1255,508.3692) .. (271.1879,474.2034) -- (273.4059,463.4731) --
    (270.5807,449.7119) .. controls (269.0268,442.1432) and (266.4805,432.4572) ..
    (264.9222,428.1874) .. controls (253.2679,396.2539) and (229.1973,365.1619) ..
    (167.9539,302.9336) .. controls (104.4488,238.4072) and (98.2665,231.9259) ..
    (99.8793,231.5664) .. controls (100.9641,231.3247) and (109.6844,239.5240) ..
    (125.4045,255.5664) .. controls (138.5372,268.9684) and (158.5835,289.3836) ..
    (169.9518,300.9336) .. controls (222.5765,354.3992) and (246.8383,383.9341) ..
    (260.3127,410.9336) .. controls (265.5626,421.4533) and (273.4876,443.9067) ..
    (272.6870,445.9931) .. controls (272.3830,446.7854) and (272.4988,447.4336) ..
    (272.9444,447.4336) .. controls (274.8593,447.4336) and (274.9541,476.3755) ..
    (273.0476,478.9336) .. controls (272.6377,479.4836) and (271.9911,481.5086) ..
    (271.6106,483.4336) .. controls (266.5065,509.2579) and (252.6040,533.6899) ..
    (223.5276,567.9336) .. controls (199.6465,596.0587) and (107.3512,690.4336) ..
    (103.7269,690.4336) .. controls (102.1938,690.4336) and (109.8474,682.1465) ..
    (131.6648,660.1836) -- cycle(141.5921,658.7381) .. controls
    (208.3347,592.4462) and (231.1771,571.6155) .. (253.8117,556.4013) .. controls
    (285.2955,535.2389) and (310.5560,532.7200) .. (326.8121,549.1218) .. controls
    (336.1541,558.5475) and (343.5765,577.6439) .. (345.5901,597.4336) .. controls
    (347.7110,618.2788) and (349.8365,690.4336) .. (348.3297,690.4336) .. controls
    (346.8596,690.4336) and (346.3167,682.5126) .. (346.1704,658.9336) .. controls
    (346.0109,633.2263) and (344.4747,606.1278) .. (342.4571,593.4336) .. controls
    (339.4484,574.5039) and (333.2464,559.6317) .. (324.8121,551.1218) .. controls
    (309.0649,535.2335) and (286.7817,537.5843) .. (255.8117,558.4013) .. controls
    (233.4320,573.4442) and (211.5587,593.3535) .. (144.8128,659.4336) .. controls
    (126.2449,677.8164) and (113.1248,690.0546) .. (111.7885,690.2381) .. controls
    (110.3236,690.4392) and (120.4345,679.7528) .. (141.5921,658.7381) --
    cycle(352.5438,666.6836) .. controls (354.1526,600.6527) and
    (356.2611,582.1743) .. (364.1991,564.5386) .. controls (371.8840,547.4651) and
    (384.8341,538.4078) .. (401.5029,538.4484) .. controls (423.2572,538.5014) and
    (448.0198,553.0400) .. (487.5714,588.9810) .. controls (518.2014,616.8148) and
    (592.4319,690.6548) .. (589.3613,690.2355) .. controls (588.0274,690.0533) and
    (573.9410,676.8502) .. (553.8314,656.9336) .. controls (500.7344,604.3462) and
    (482.0519,586.7934) .. (462.0714,570.7221) .. controls (437.4023,550.8795) and
    (416.6428,540.5034) .. (401.5029,540.4484) .. controls (390.9708,540.4101) and
    (379.6474,546.3614) .. (372.6219,555.6275) .. controls (359.7952,572.5448) and
    (355.3646,600.0756) .. (354.9503,665.4336) .. controls (354.8062,688.1681) and
    (354.6214,690.4336) .. (352.9113,690.4336) .. controls (352.3267,690.4336) and
    (352.1863,681.3569) .. (352.5438,666.6836) -- cycle(552.3104,646.6836) ..
    controls (504.9788,598.4485) and (497.6282,590.7246) .. (481.0669,571.8228) ..
    controls (452.1755,538.8482) and (436.6821,512.9462) .. (430.5893,487.4336) ..
    controls (429.6042,483.3086) and (428.4307,479.4629) .. (427.9816,478.8875) ..
    controls (427.5325,478.3122) and (427.1945,476.9622) .. (427.2304,475.8875) ..
    controls (427.2664,474.8129) and (426.9819,470.7836) .. (426.5981,466.9336) ..
    controls (425.3328,454.2389) and (431.8295,428.9440) .. (441.1328,410.3432) ..
    controls (449.7725,393.0691) and (463.7437,373.8362) .. (485.3772,349.4355) ..
    controls (509.9821,321.6834) and (599.3794,231.1521) .. (601.7272,231.6097) ..
    controls (602.8506,231.8286) and (589.7894,245.7157) .. (561.4502,274.4336) ..
    controls (489.3978,347.4488) and (472.5925,366.0500) .. (454.9244,392.3432) ..
    controls (441.8300,411.8299) and (435.2488,427.2699) .. (430.6029,449.4031) ..
    controls (428.2599,460.5651) and (428.0998,462.6096) .. (429.0754,468.9031) ..
    controls (436.1055,514.2527) and (455.3468,543.1689) .. (528.9082,618.9336) ..
    controls (538.5203,628.8336) and (558.3486,648.9711) .. (572.9711,663.6836) ..
    controls (592.7083,683.5423) and (599.0015,690.4336) .. (597.3992,690.4336) ..
    controls (595.8752,690.4336) and (582.6243,677.5762) .. (552.3104,646.6836) --
    cycle(284.5199,385.4918) .. controls (272.6068,381.9700) and
    (263.2277,376.8940) .. (247.5714,365.4952) .. controls (225.9158,349.7285) and
    (206.2378,331.5369) .. (144.5191,270.2272) .. controls (117.6016,243.4879) and
    (106.0955,231.4336) .. (107.4900,231.4336) .. controls (108.8144,231.4336) and
    (122.6469,244.4879) .. (147.5422,269.2324) .. controls (196.3973,317.7915) and
    (209.7529,330.5461) .. (229.0714,347.0927) .. controls (258.9862,372.7151) and
    (279.1501,383.9864) .. (297.1981,385.1743) .. controls (320.9008,386.7344) and
    (336.1808,369.6304) .. (342.0677,334.9485) .. controls (344.2694,321.9778) and
    (346.0714,295.3071) .. (346.0714,275.6922) .. controls (346.0714,254.3795) and
    (347.2569,231.4336) .. (348.3581,231.4336) .. controls (348.9838,231.4336) and
    (349.0911,242.0772) .. (348.6578,261.1836) .. controls (347.3897,317.1076) and
    (345.3528,338.7068) .. (339.8642,354.4336) .. controls (333.6515,372.2349) and
    (325.7344,381.3489) .. (312.6433,385.7695) .. controls (306.0889,387.9828) and
    (292.4920,387.8485) .. (284.5199,385.4918) -- cycle(388.4677,386.0154) ..
    controls (372.8204,381.1180) and (361.5908,363.5848) .. (357.0941,337.0303) ..
    controls (354.1727,319.7784) and (350.5983,231.4336) .. (352.8218,231.4336) ..
    controls (354.0209,231.4336) and (355.0610,249.2685) .. (355.0669,269.9290) ..
    controls (355.0729,292.4921) and (356.9011,321.9852) .. (359.1015,335.0303) ..
    controls (362.1232,352.9450) and (368.0235,366.5260) .. (376.4226,374.8995) ..
    controls (383.5922,382.0472) and (389.5533,384.7131) .. (399.5301,385.2335) ..
    controls (418.1046,386.2024) and (440.0047,374.5352) .. (473.4695,345.8427) ..
    controls (491.7970,330.1288) and (505.2391,317.3207) .. (552.7996,270.2543) ..
    controls (578.2711,245.0475) and (592.7327,231.4336) .. (594.0378,231.4336) ..
    controls (595.4111,231.4336) and (583.3015,244.0475) .. (555.8096,271.2538) ..
    controls (507.2095,319.3491) and (493.8648,332.0707) .. (475.4695,347.8427) ..
    controls (451.4831,368.4084) and (435.7129,378.7478) .. (420.1436,384.1156) ..
    controls (411.9816,386.9295) and (409.2455,387.4072) .. (401.5714,387.3578) ..
    controls (396.6214,387.3259) and (390.7248,386.7218) .. (388.4677,386.0154) --
    cycle;
\draw (310,470) node {$\cdot$};
\draw (390,470) node {$\cdot$};
\draw  (70,470) node[left] {$\left( \begin{smallmatrix}
a_1&a_2&a_3&a_4&a_5&a_6&a_7&a_8\\
 1& 1& 1& 1& \sqrt{3}& 1& 2& \sqrt{3}\\
a_9&a_{10}&a_{11}&a_{12}&a_{13}&a_{14}&a_{15}&a_{16}\\
 \sqrt{3}& 2& 2& 2& 2& 2& 1& \sqrt{3}\\ 
a_{17}&a_{18}&a_{19}&a_{20}&a_{21}&a_{22}&&\\
 \sqrt{3}& \sqrt{3}& 1& \sqrt{3}& 1& \sqrt{3}&&
 \end{smallmatrix}\right)
$};
\end{tikzpicture}
\end{center}
\end{example}

\subsection{Freezing in the tropics}
We will now have a look at what happens if we let the system cool down. Mathematically this means that we let the temperature go to zero and therefore
the $\frac{1}{k T}$ will go to infinity. As we absorbed the temperature constant into the energy, we can model this by looking at the one-parameter family
$Et$ with $t\to \infty$. If we look at the partition function we will get a one-parameter family of equations $P_{Et}(x,y)=0$. If we look at the amoeba, its central part
will become larger and larger, but we can rescale the amoeba by the factor $t$. In the limit the rescaled amoeba becomes
infinitely thin and  $\lim_{t\to \infty}\frac{1}{t} \Am(P_t)$  will be a union of line segments and half lines. 
Such a limit is also know as a tropical curve.

\begin{intermezzo}[Tropical curves]
The tropical semifield $\mathbb{F}_{\trop}$ is the set $\R\cup \{-\infty\}$ equipped with two operations $a \oplus b = \max (a,b)$ and
$a\odot b = a+b$. A tropical polynomial $f_\trop(x,y)= \oplus_{ij} a_{ij}\odot x^{\odot i}y^{\odot j}$ looks 
like an ordinary Laurent polynomial but for the evaluation we use the tropical operations instead of the ordinary ones.
\[
f_\trop(x,y) = \max_{i,j} (a_{ij}+ix+jy)
\]
which is a piecewise affine function $f_{\trop}:\R^2 \to \R$. Unlike ordinary polynomials, different tropical polynomials can give rise to the same function.
Another way of calculating  $f_\trop(x,y)$ is
\[
 f_\trop(x,y) = \lim_{t \to +\infty} \log_t f_t(t^x,t^y) \text{ with }f_t(X,Y) = \sum_{ij} t^{a_{ij}}X^iY^j
\]
We will often use $f_t(X,Y)$ instead of $f_\trop$ because for $f_t$ we can use the ordinary sum and product operations. 

The tropical curve defined by a tropical polynomial $f$ is the locus in $\R^2$ where $f_{\trop}$ is not smooth, or equivalently
where the maximum is reached by two terms in the expression of $f_\trop$. Because $f_\trop$ is piecewise affine, the tropical curve is a union of
line segments that cuts $\R^2$ into convex pieces, cells, and each of these convex pieces corresponds to a lattice point $(i,j)$ where the maximum
is reached by $a_{ij}+ix+jy$.

The tropical curve $f_\trop$ defines a subdivision of the Newton polygon $\NP(f)$. Draw all the lattice points corresponding to cells of the tropical curve
and connect two lattice points by and edge if the two cells share a line segment of the tropical curve. In other words this subdivision is the dual graph of the 
tropical curve. 

\begin{center}
\begin{tikzpicture} 
\draw (3,2.5) node[above] {$\scriptstyle{f_t(X,Y) = 1+X+Y+XY+t^{-1}X^2Y^2}$}; 
\draw (1,0) -- (0,0); 
\draw (2,2) -- (1,0); 
\draw (1,1) -- (1,0); 
\draw (0,0) -- (0,1); 
\draw (0,1) -- (2,2); 
\draw (1,1) -- (2,2); 
\draw (0,1) -- (1,1); 
\draw (0,0) node[circle,draw,fill=white,minimum size=10pt,inner sep=1pt] {{\tiny0}};
\draw (1,0) node[circle,draw,fill=white,minimum size=10pt,inner sep=1pt] {{\tiny0}};
\draw (0,1) node[circle,draw,fill=white,minimum size=10pt,inner sep=1pt] {{\tiny0}};
\draw (2,2) node[circle,draw,fill=white,minimum size=10pt,inner sep=1pt] {{\tiny-1}};
\draw (1,1) node[circle,draw,fill=white,minimum size=10pt,inner sep=1pt] {{\tiny0}};

\begin{scope}[xshift=4cm,yshift=.5cm]
\begin{scope}[scale=.5]
\draw (0,0) -- (0,-1); 
\draw (2,0) -- (4,-1); 
\draw (0,0) -- (2,0); 
\draw (0,0) -- (-1,0); 
\draw (0,2) -- (-1,4); 
\draw (2,0) -- (0,2); 
\draw (0,0) -- (0,2); 
\end{scope}
\end{scope}
\end{tikzpicture} 
\end{center}
For more information about tropical geometry we refer to \cite{mikhalkin2004amoebas, mikhalkin2006tropical}
\end{intermezzo}

If we return to the main story, we can interpret the $\lim_{t\to \infty}\frac{1}{t} \Am(P_{Et})$ as a tropical curve.
Look at a fixed point $(u,v) \in \R^2$. As $t$ goes to infinity each term of $P_{Et}(z,w)$ with $(|z|,|w|)=(e^{ut},e^{vt})$ 
will behave as $t^\alpha$ for some order $\alpha$. If $(z,w)$ is a zero there must be at least 2 terms with the highest order
so they can cancel each other.
Therefore one can deduce that the limit amoeba for 
\[
P_{Et}(x,y) = \sum_k \alpha_{k}e^{\beta_{k}t}X^{i_k}Y^{j_k}
\]
will be the tropical curve defined by the tropical polynomial
\[
 P_{E\trop}(x,y) = \max_{k}( i_k x+j_k y +\beta_k)
\]
From the physics point of view this means that the liquid phase shrinks to the line segments of the tropical curve and the cells are the solid and gaseous phases of the phase diagram. 

\begin{theorem}\label{alltropicaloccur}
Given a consistent dimer $\rib$ with matching polygon $\MP(\rib)$ and a tropical polynomial $f_\trop$ with Newton polygon $\NP(f_t)=\MP(\rib)$ then
there is an element $E \in \tL$ for which $f_\trop$ and $P_{E\trop}$ define the same tropical curve.
\end{theorem}
\begin{proof}
This follows from \ref{Harnackfromdimer} and 
and a theorem by Itenberg and Viro \cite{itenberg1996patchworking,viro2001dequantization} 
that every tropical curve can be realized as the limit of a Harnack curve.
\end{proof}

Just like in the nonzero temperature case we can put all these tropical curves together for one fixed dimer. In this way we get a tropical/frozen version
of our statical system for one dimer
\[
 \tL(\rib) \to \tX(\rib)  = \R^{\#\rib_1-\#\rib_0 +1} \to \R^{\rib_2-1}.
\]
The coordinates on $\tL$ and $\tX$ are tropical versions of the holonomies around the cycles, which we should view as 
\[
W_F^\trop(E) = \lim_{t\to \infty}\log_t W_F(t^E) = \sum_{e \in F} \pm E_e.
\]
Although these expressions are just logarithmic versions of the original holonomies, their behaviour is different. This becomes
apparent if we look at mutation. 

If $\rib$ and $\mu\rib$ are related by mutation at $F$ we can take the limit of the mutation formula for $t \to \infty$ and then we get
\[
\phi_\mu^\trop({W'_u}^\trop) = W_{\phi(u')}^\trop +\eps(\phi(u'),F)\max(0, W_F^\trop)
\]
This is a piecewise linear transformation which is everywhere defined and invertible. 
Just like in the statistical model $\tL \to \tX$ is just equal to $\tL(\rib) \to \tX(\rib)$ for all dimers in the same mutation class.

\section{The $B$-side: representation theory and crepant resolutions }\label{Bside}

In this section we will look at 3-dimensional toric Gorenstein singularities and we will use dimers to construct crepant resolutions. The main idea
is that these crepant resolutions can be seen as moduli spaces of representations of a noncommutative algebra associated with the dimer.
Different crepant resolutions arise by varying the stability condition that is used to construct the moduli space and to each dimer
we can associate a space of stability conditions.

This noncommutative algebra is called the Jacobi algebra and it is an example of a noncommutative crepant resolution (NCCR). 
Jacobi algebras of isotopic dimers will be NCCRs for the same singularity. Using quiver mutation it is possible to 
translate stability conditions of one NCCR to stability conditions of another NCCR. 

Finally we relate this to the tropical version of the space $\cX$ that we constructed in the $A$-model.

\subsection{Toric geometry}
\newcommand{\fan}{\mathsf{Fan}}
\newcommand{\grr}{\mathscr{T}}
\newcommand{\sK}{\mathscr{K}}

We start with reviewing some basic toric geometry that we will need in our story, for more details we refer to the classical introduction by Fulton \cite{fulton1993introduction}.

A \iemph{toric variety} $X$ is a normal variety that contains a torus $T=(\C^*)^n$ as an open dense subset and the standard action by $T$ on itself by multiplication extends to an action on 
$X$. To the torus we can associate two dual lattices: $N=\Hom(\C^*,T)$ contains the one parameter subgroups of $T$ and $M=\Hom(T,\C^*)$
its characters. The natural pairing can be written as $\<n,m\>=k\in \Z$ if $m\circ n(t)=t^k$. We will write $N_\R$ for $N\otimes \R$ and $M_\R$ for $M\otimes \R$. 

To describe the toric variety we can consider the set of all one parameter subgroups $\lambda$ for which the limit $\lim_{t\to 0} \lambda(t)$ exists in $X$. 
If it exists we denote by $O_\lambda$ the $T$-orbit of the limit point.
We say that $\lambda \cong \mu$ if they have the same limit and $\mu \preceq \lambda$ if $O_\lambda$ lies in the closure of $O_\mu$.  
To every $\lambda$ we assign a cone
\[
 \sigma_\lambda := \sum_{\mu \preceq \lambda} \R_+ \mu \subset N_\R.
\]
The fan of $X$ is the set containing all these cones:
\[
 \fan(X) = \{ \sigma_\lambda| \lambda \in N \text{ and }\lim_{t\to 0} \lambda(t) \in X\}.
\]

The interior of a cone $\sigma$ contains all $\lambda$ with the same orbit, so we also denote this orbit by $O_\sigma$.
If we take the union of all orbits that come from cones contained in $\sigma$ we get an open subset 
\[
 U_\sigma = \bigcup_{\tau \subset \sigma} O_\tau.
\]
This open subset also has a different description. The dual of a cone $\sigma\subset N\otimes \Rl$ is the
set $$\sigma^\vee =\{m \in M\otimes \Rl| \forall n \in \sigma: \<n,m\>\ge 0\}.$$ 
The intersection of the dual cone with the lattice $M$ gives a semigroup.
Because $\<n,m\>\ge 0$ implies that the limit $m(\lim_{t \to 0} n(t))=\lim_{t \to 0} t^{\<n,m\>}$ exists and $m$, we can see the lattice points in $M\cap \sigma^\vee$
as functions on $U_\sigma$ and therefore
\[
 U_\sigma = \Spec \C[\sigma^\vee \cap M].
\]

In general a \iemph{fan} is a collection of integral cones in $N\otimes \Rl$ that is closed under intersection and taking faces.
For each $\sigma \in \cF$ we define $U_\sigma=\Spec\C[\sigma^\vee \cap M]$ and
if we glue them together by the natural inclusion maps $U_\tau \subset U_\sigma$ whenever $\tau\subset \sigma$, we obtain a toric variety $X_\cF$. 

Toric varieties are completely determined by their fans and a lot of interesting properties can be read of
from the structure of the fan.
\begin{itemize}
\item A toric variety is affine if its fan has one maximal cone. 
\item A toric variety is complete if the fan covers the whole of $N\otimes \Rl$.
\item $U_\sigma$ is smooth if $\sigma\cap N$ is generated by a linearly independent set of $n_i \in N$.
\end{itemize}

If $X$ is a toric variety with fan $\cF$ then we can look at sheaves over $X$. A special type of sheaves are the \iemph{graded rank $1$ reflexive sheaves}.
Let $\{v_i|i \in \cI\}$ be the set of the generators of all $1$-dimensional cones and let $a=(a_i)_{i \in \cI}$ be a sequence of integers.
Out of this sequence we can construct a sheaf $\grr_a$ as follows
\[
 \grr_a(U_\sigma) = \C[ m \in M| \forall v_i : \<v_i,m\>+a_i \ge 0].
\]
From this definition we can easily deduce $\grr_a\otimes \grr_b=\grr_{a+b}$ and $\grr_a \cong \cO_X$ if there is an $m \in M$ such that $a_i=\<v_i,m\>$.
The set of all $\grr_a$ up to isomorphism together with the tensor product give a group, which we denote by $\Pic X$ because if
$X$ is smooth these graded rank one reflexives are precisely the line bundles over $X$. 

This group contains a special element $\sK=\grr_{\id}$ corresponding to the sequence consisting only of $1's$. This element is
the \iemph{canonical sheaf}. This sheaf is trivial if all the $v_i$ lie in a common hyperplane at height $1$.
If that is the case the toric variety is special.
If $X$ is smooth then we call such a variety \iemph{Calabi-Yau}, if it is singular we call it \iemph{Gorenstein}.

\subsection{Resolutions of singularities}
One of the themes in algebraic geometry is the resolution of singularities. Given a singular affine variety $X$ over $\C$ with coordinate ring $\C[X]$, we want to
find a smooth variety $\tilde X$ and a surjection $$\pi:\tilde X \to X$$ that is proper (the fibers are complete) and birational (almost everywhere one to one). 
The map $\pi$ is called a \iemph{resolution} of $X$ and the locus in $\tilde X$ that is not one-to-one is called the \iemph{exceptional locus}.

In general a resolution is far from unique and one would like to define a class of resolutions that remain close to $X$ geometrically. To capture this idea Reid introduced the 
the notion of a crepant resolution \cite{reid1983minimal}. 
\begin{definition}
A resolution $\pi:\tilde X \to X$ is \iemph{crepant} if the pullback of the canonical sheaf on $X$ is the canonical sheaf on $\tilde X$.
\end{definition}
If the canonical sheaf of $X$ is trivial and $\pi:\tilde X \to X$ is crepant then the canonical sheaf of $\tilde X$ will also be trivial, so
crepant resolutions of Gorenstein singularities result in Calabi-Yau varieties.

A fan $\tilde \cF$ is a \iemph{refinement} of a fan $\cF$ if every cone in $\cF'$ sits inside a cone of $\cF$ and every cone is $\cF$ is the union of some cones in $\cF'$. 
This gives a natural $T$-equivariant map $X_{\tilde\cF}\to X_\cF$ which maps the limit point of each cone in $\tilde \cF$ to the limit point of the cone in which it is contained.
This map is proper and surjective because every cone in $\cF$ is completely covered by cones in $\cF'$ and it is birational because it is one-to-one on the torus $T$,
so if $X_{\tilde \cF}$ is smooth then $X_{\tilde\cF}\to X_\cF$ is a resolution.

Let us now look more closely at the 3-dimensional case. To construct a 3-dimensional Gorenstein singularity we fix a convex lattice polygon $\LP \subset \R^2$ and 
put it at height $1$ in 3-dimensional space $\R^3$. The corners of this polygon are vectors of the form $v_i=(x_i,y_i,1) \in \Z^3$. The cone $\sigma = \sum_i \Rl_{\ge 0} v_i$ gives us a variety 
$X=U_\sigma$ with coordinate ring 
$$
R_{\LP}:=\C[\sigma \cap \Z^3]=\C[X^rY^sZ^t| x_ir+y_is+t\ge 0].
$$
This coordinate ring is Gorenstein because the corners of the polygon all sit in the same plane. Unless the polygon is an elementary triangle this variety is not smooth. 

To construct a resolution $\tilde X \to X$, we need to subdivide the cone into smooth subcones. If we can ensure that the generators of all these subcones 
lie in the same plane as the polygon, then the resolution is crepant and $\tilde X$ is Calabi-Yau. This can be done by subdividing the polygon in elementary triangles.
The simplest example is the crepant resolution of $\Spec\C[X,Y,Z,W]/(XY-ZW)$, which is illustrated below.
\begin{center}
\begin{tikzpicture}
\begin{scope}[scale=.5]
\filldraw  
(0,0) circle (2pt)  (1,0) circle (2pt)
(0,1) circle (2pt) (1,1) circle (2pt);
\draw (0,0) -- (1,0) -- (1,1) -- (0,1) -- (0,0);
\draw (0,0) -- (1,1);  
\end{scope}
\draw (1.75,.25) node {$\stackrel{\text{resolution}}{\longrightarrow}$};
\begin{scope}[xshift=3cm]
\begin{scope}[scale=.5]
\filldraw  
(0,0) circle (2pt)  (1,0) circle (2pt)
(0,1) circle (2pt) (1,1) circle (2pt);
\draw (0,0) -- (1,0) -- (1,1) -- (0,1) -- (0,0);  
\end{scope}
\end{scope}
\end{tikzpicture}
\end{center}

In most cases there are several ways to do this and each way gives a different crepant resolution. 
There is a construction that transforms one crepant resolution into another: the flop \cite{kollar1990flip}. Basically the flop blows down a $\PP_1$ and substitutes
it by a different $\PP_1$. In the toric picture one takes an internal edge of the subdivision of the polygon and looks at the two elementary triangles
that it bounds. 
If these $2$ form a convex quadrangle then we can remove the edge and put in the other diagonal of the quadrangle. In this way we obtain
a new subdivision of the polygon or a new crepant resolution of the singularity. The example below is called the Atiyah flop \cite{atiyah1958analytic}.
\begin{center}
\begin{tikzpicture}
\begin{scope}[scale=.5]
\filldraw  
(0,0) circle (2pt)  (1,0) circle (2pt)
(0,1) circle (2pt) (1,1) circle (2pt);
\draw (0,0) -- (1,0) -- (1,1) -- (0,1) -- (0,0);
\draw (0,0) -- (1,1);  
\end{scope}
\draw (1.75,.25) node {$\stackrel{\text{flop}}{\longrightarrow}$};
\begin{scope}[xshift=3cm]
\begin{scope}[scale=.5]
\filldraw  
(0,0) circle (2pt)  (1,0) circle (2pt)
(0,1) circle (2pt) (1,1) circle (2pt);
\draw (0,0) -- (1,0) -- (1,1) -- (0,1) -- (0,0);  
\draw (1,0) -- (0,1);  
\end{scope}
\end{scope}
\end{tikzpicture}
\end{center}

\begin{example}
If we take the lattice polygon $$\<(1,0),(0,1),(-1,0),(-1,-1),(0,-1)\>$$
we see that there are $5$ different toric crepant resolutions.
The arrows indicate the flops.
\begin{center}
\begin{tikzpicture}
\begin{scope}[scale=.5]
\begin{scope}
\draw (1,0) node{$\bullet$}--(0,1) node{$\bullet$}--(-1,0) node{$\bullet$}--(-1,-1) node{$\bullet$}--(0,-1) node{$\bullet$} -- (1,0);
\draw (0,0) node{$\bullet$};
\draw (0,0) -- (1,0);
\draw (0,0) -- (0,1);
\draw (0,0) -- (-1,0);
\draw (0,0) -- (-1,-1);
\draw (0,0) -- (0,-1);
\end{scope}
\begin{scope}[yshift=-4cm]
\draw (1,0) node{$\bullet$}--(0,1) node{$\bullet$}--(-1,0) node{$\bullet$}--(-1,-1) node{$\bullet$}--(0,-1) node{$\bullet$} -- (1,0);
\draw (0,0) node{$\bullet$};
\draw (0,0) -- (1,0);
\draw (0,0) -- (-1,0);
\draw (-1,-1) -- (1,0);
\draw (0,0) -- (-1,-1);
\draw (0,0) -- (0,1);
\end{scope}
\begin{scope}[yshift=4cm]
\draw (1,0) node{$\bullet$}--(0,1) node{$\bullet$}--(-1,0) node{$\bullet$}--(-1,-1) node{$\bullet$}--(0,-1) node{$\bullet$} -- (1,0);
\draw (0,0) node{$\bullet$};
\draw (0,0) -- (1,0);
\draw (0,0) -- (0,1);
\draw (0,0) -- (0,-1);
\draw (0,0) -- (-1,-1);
\draw (-1,-1) -- (0,1);
\end{scope}
\begin{scope}[xshift=4cm]
\draw (1,0) node{$\bullet$}--(0,1) node{$\bullet$}--(-1,0) node{$\bullet$}--(-1,-1) node{$\bullet$}--(0,-1) node{$\bullet$} -- (1,0);
\draw (0,0) node{$\bullet$};
\draw (0,0) -- (1,0);
\draw (0,0) -- (0,1);
\draw (-1,-1) -- (0,1);
\draw (0,0) -- (-1,-1);
\draw (-1,-1) -- (1,0);
\end{scope}
\begin{scope}[xshift=-4cm]
\draw (1,0) node{$\bullet$}--(0,1) node{$\bullet$}--(-1,0) node{$\bullet$}--(-1,-1) node{$\bullet$}--(0,-1) node{$\bullet$} -- (1,0);
\draw (0,0) node{$\bullet$};
\draw (0,0) -- (1,0);
\draw (0,0) -- (0,1);
\draw (0,0) -- (-1,0);
\draw (-1,0) -- (0,-1);
\draw (0,0) -- (0,-1);
\end{scope}
\draw (-2,0) node {$\ot$};
\draw (0,2) node {$\uparrow$};
\draw (0,-2) node {$\downarrow$};
\draw (1.5,-1.75) node {$\nearrow$};
\draw (1.5,1.75) node {$\searrow$};
\end{scope}
\end{tikzpicture}

\end{center}
The middle one resolves the singularity by adding a single del Pezzo surface, $dP_5$.
For the upper and lower resolution the exceptional fiber consists of a Hirzebruch
surface and an extra $\PP_1$. For the left resolution we have a $\PP_1\times \PP_1$ and a $\PP_1$ and for the right one a $\PP_2$ and 2 $\PP_1$'s.
The surface can be deduced from the neighborhood of the central lattice point 
and the number of extra $\PP_1$'s is the number of line segments not attached
to the central lattice point.
\end{example}

\subsection{Moduli spaces and geometric representation theory}

In this section we will review the basics on moduli spaces of representations of quivers, for more information we refer to
\cite{king1994moduli} and \cite{le2007noncommutative}.

Recall that the \iemph{path algebra} $\C Q$ of a quiver $Q$ is the vector space
spanned by all paths in $Q$. The product $pq$ of two paths $p,q$
is their concatenation if $t(p)=h(q)$ and zero otherwise.

For a dimension vector $\alpha:\qpol_0\to \N$ we define the space of $\alpha$-dimensional representations of $Q$ as  
$$
\Rep_\alpha Q :=  \bigoplus_{a \in \qpol_1}\Mat_{\alpha(h(a))\times \alpha(t(a))}(\C).
$$
We will denote the total dimension of these representations by $n=\sum_{v \in \qpol_0} \alpha_v$ and the group $\prod_{v \in \qpol_0} \GL_{\alpha(v)}(\C)$ by $\GL_\alpha$.

Every point $(m_a)$ in the space $\Rep_\alpha Q$ corresponds to a representation $\rho:\C Q \to \Mat_n(\C)$ that maps an arrow $a$ to a block matrix with $(m_a)$ in 
the appropriate position and zeros everywhere else. The group $\GL_\alpha$ acts on these representations by conjugation and the orbits
of this action can be seen as the isomorphism classes of the $\alpha$-dimensional modules/representations of the path algebra $\C Q$. 
\footnote{a $\C Q$-module $M$ is $\alpha$-dimensional if $\dim vM=\alpha_v$ for every $v \in \qpol_0$.} 

The most basic way to construct a quotient for the $\GL_\alpha$-action is to look at the ring of invariants
\[
\C[\CRep_\alpha A]^{GL_\alpha}  := \{f \in \C[\CRep_\alpha A]| \forall g \in \GL_\alpha: f\circ g = f\}.
\]
The spectrum of this ring parametrizes the closed orbits in $\Rep_\alpha Q$ and is called the \iemph{categorical quotient}, $\Rep_\alpha Q/\!\!/ \GL_\alpha$. The closed orbits are precisely the isomorphism classes of semisimple representations.

A more refined quotient can be taken if we specify a stability condition. This is a character
\[
\theta : \GL_\alpha \to \C^* : g \to \prod_{v \in \qpol_0} \det g_v^{\theta_v}. 
\] 
The \iemph{Mumford quotient} $\cM_\alpha^\theta(Q)$ is by definiton the proj of the graded ring of semi-invariants
\[
\C[\Rep_\alpha Q]_\theta := \bigoplus_{n=0}^\infty \{f \in \C[\Rep_\alpha Q]| f\circ g = \theta(g)^n f\}
\]
To interpret this quotient we call a representation \iemph{$\theta$-stable} if $\theta\cdot \alpha:=\sum_{v \in \qpol_0}\theta_v \alpha_v=0$ and it has no 
proper subrepresentation with dimension vector $\beta$ such that $\theta\cdot \beta\le 0$. If we relax the condition to $\theta\cdot \beta< 0$ 
then we call the module \iemph{$\theta$-semistable} and if it is a direct sum of stable representations we call it \iemph{$\theta$-polystable}.

The set of semistable representations forms an open subset $\Rep_\alpha^{\theta-ss} Q \subset \Rep_\alpha Q$ that is closed under the $\GL_\alpha$-action. The
moduli space $\cM_\alpha^\theta(Q)$ parametrizes the orbits that are closed in $\Rep_\alpha^{\theta-ss} Q$. It is well known that these are precisely the orbits of polystable representations. 

Because the moduli space $\cM_\alpha^\theta(Q)$ is the proj of the graded ring $\C[\Rep_\alpha Q]_\theta$ 
it comes with a natural projection to the spectrum of its degree zero part. The latter is just the ring of invariant functions, so this projection map 
is $\cM_\alpha^\theta(Q) \onto \Rep_\alpha Q/\!\!/ \GL_\alpha$.
The target can also be identified with the moduli space for the trivial character,$\cM_\alpha^0(Q)$, because $\C[\CRep_\alpha A]_0=\C[\CRep_\alpha A]^{GL_\alpha} \otimes \C[t]$.

A character is called \iemph{generic} if all vectors $\beta \in \N^{\qpol_0}$ with $0<\beta<\alpha$ satisfy $\theta \cdot\beta\ne 0$. For generic characters the concepts
stable, polystable and semistable $\alpha$-dimensional representations coincide. 

If $\theta$ is generic the moduli space is a \iemph{fine moduli space}. This means $\cM_\alpha^\theta(Q)$ comes with a bundle of which the fibres are precisely the
modules parametrized by the points in the moduli space. This bundle is called the \iemph{tautological bundle}.

To construct it, we assume for simplicity that there is a vertex $v$ with $\alpha_v=1$. In this case $\cU$ is the fibred product
\[
\Rep_\alpha^{\theta-ss} Q \times_{\GL_\alpha} \C^n 
\]
where $\GL_\alpha$ acts on the last part by $g\cdot x = g_v^{-1}g x$. 
On each $\{\rho\} \times \C^n$  the algebra $\C Q$ acts via the representation $\rho$ and the $\GL_\alpha$-action induces isomorphism between
the different modules. This implies that each fiber of the projection map $\cU \to \cM_\alpha^\theta(Q)$ has a natural $\C Q$-module structure and
$\cU$ can be seen as a sheaf of modules over $\cM_\alpha^\theta(Q)$.

The whole construction carries over to \iemph{path algebras with relations}. If
$A = \C Q/I$ we can construct the subscheme
$\Rep_\alpha A \subset \Rep_\alpha  Q$ containing the representations of $\C Q$ that factor through $A$.
Its ring of coordinates is
\[
 \C[\Rep_\alpha A] = \frac{\C[X_{ij}^a| a \in Q_1]}{\<r_{st}(X_{ij}^a)| r \in I\>}
\]
where the $r_{st}$ are the coeffients of the matrix we obtain when we substitute the arrows in the relation $r$ by their matrices $(X_{ij}^a)$.
In general this scheme is not reduced and it can consist of many components. 

The subscheme $\CRep_\alpha A$ is closed under the $\GL_\alpha$-action and if we fix a character $\theta$, we can define 
$\CM_\alpha^\theta(A)$ as $\Proj \C[\CRep_\alpha A]_\theta$. This moduli space is a closed subscheme of $\CM_\alpha^\theta(Q)$ that
classifies the polystable representations of $A$. If $\theta$ is generic this is a fine moduli space and the bundle is the restriction
of $\cU$ to the base $\CM_\alpha^\theta(A)$.

\subsection{Noncommutative geometry}

The basic idea of noncommmutative geometry is to extend the classical duality between
commutative rings and affine schemes to noncommutative algebras. There are many different ways to do 
this, all with there own advantages and disadvantages. We will focus on a categorical approach. 

To a variety $X$ we can associate its category of coherent sheaves, $\Coh X$. This is an abelian category and it contains the same information as the variety itself because
we can reconstruct $X$ from it. 
If $X$ is affine with coordinate ring
$R$, this category is equivalent to the category of finitely generated $R$-modules: $\Mod R$. If $X$ is not affine it is in general not possible
to find an algebra $A$ (not even a noncommutative one) such that $\Coh X \cong \Mod A$.

To address this problem we look at derived categories. A good reference for the the topics covered in this section is the review article by Keller \cite{keller2007derived}.
The \iemph{derived category of coherent sheaves} $\Db\Coh X$ is the homotopy category of complexes of coherent sheaves, with all quasi-isomorphisms inverted. This category contains $\Coh X$ as a full subcategory and it does not capture all information of $X$. Different Calabi-Yau varieties can have equivalent derived 
categories \cite{bondal2001reconstruction}. 

In a similar way we can define the \iemph{derived category of modules of an algebra} $\Db\Mod A$ as an enlargement of $\Mod A$. Unlike in the underived setting
it is possible to find non-affine varieties and algebras with equivalent derived categories. The easiest example is the projective line, which is derived equivalent to
the path algebra of the Kronecker quiver: 
$$\Db\Coh\PP_1 \cong \Db\Mod \C Q \text{ with }Q=\xymatrix{\vtx{}&\vtx{}\ar@{=>}[l]}.$$

Algebras and varieties can be derived equivalent in subtle ways and there are interesting constructions to go from one derived category to another.
The main idea is to use an object that has two structures as a bridge between the two.

Let us first look at \iemph{Morita theory}. If $M$ is a $A-B$ bimodule then there is a natural functor by tensoring
with $M$: $- \otimes_A M : \Mod A \to \Mod B$. This functor has an adjoint $\Hom_B(M,-): \Mod B \to \Mod A$ and
if these two functors are inverses of each other then $A$ and $B$ are called Morita equivalent. In that
case $\End_A(M,M)\cong B^{op}$ and $M$ considered as a $B$-module is a projective generator of $\Mod A$.
We can also go in the other direction: if $M$ is a projective generator in $\Mod A$, we can define $B = \End_A(M,M)^{op}$ and
view $M$ as an $A-B$-bimodule. The functor $- \otimes_A M: \Mod A \to \Mod B$ will then be an equivalence of categories.

If we move into the derived world we can do a similar thing. For any $A$-module $T$ there is a natural derived functor
$\RHom^\bullet_{A}(T,-): \Db\Mod A \to \Db\Mod B$ with $B=\Hom_{A}(T,T)^{op}$. We call $T$ a \iemph{tilting module} if it is a generator and $\Hom_{\Db\Mod A}(T,T[i])=0$ if $i>0$.
In that case the functor $\RHom_{A}^\bullet(T,-)$ is an equivalence.

This construction also works when we start with a sheaf $\cTT \in \Db\Coh X$ satifying the same conditions. In this case 
$\cTT$ is called a \iemph{tilting sheaf} and we get an equivalence between $\Db\Coh X$ and $\Db\Mod B$ with $B=\Hom_{\Db\Mod A}(\cTT,\cTT)^{op}$.

If we want to go in the other direction (start with an algebra and end with a variety) we can try a moduli construction.
If $A$ is an algebra and $\cU \to \cM$ is a fine moduli space of representations then $\cU$ can be seen as a sheaf of $A$-modules, so
there is a natural morphism $A \to \End_\cM(\cU)$. If this morphism is a bijection and $\cU$ is a tilting sheaf, the functor
$\Hom^\bullet_{\Db\Mod A }(\cU,-)$ is an equivalence with inverse $- \otimes^L \cU$.

Finally, we also want to go directly between two varieties $X$ and $Y$. Instead of a bimodule we need an object that is a sheaf on
both $X$ and $Y$, in other words a sheaf $\cPP$ on the product $X\times Y$. The functor that induces the equivalence is called the \iemph{Fourier-Mukai transform} and $\cPP$ is called its \iemph{kernel}.
It first pulls back a sheaf on $X$ to a sheaf on $X\times Y$
using the projection operator $p_X: X\times Y \to X$, then it tensors this sheaf with $\cPP$ and pushes it forward to a sheaf on $Y$ along the projection operator $p_Y: X\times Y \to X$:
$$\FM:\Db\Coh X \to \Db\Coh Y: \cG \mapsto p_{Y*} (\cPP \otimes p_X^*(\cG)).$$ 
Note that the functors above are used in their derived versions. For more on Fourier-Mukai transforms we refer to \cite{bridgeland1999equivalences} and \cite{huybrechts2006fourier}

\begin{example}
We illustrate these concepts with a few basic examples:
\begin{enumerate}
 \item The algebras $A=R$ and $B=\Mat_n(R)$ are Morita equivalent via
 the $A$-$B$-bimodule $R^n$.
 \item The path algebras of the quivers 
 $Q_1=\circ\!\ot\! \circ\! \to \!\circ$ and $Q_2=\circ\!\to\! \circ\!\ot\! \circ$ are derived equivalent by the tilting bundle
 $M_{110}\oplus M_{111} \oplus M_{011}$, where $M_{abc}$ is the unique indecomposable
 right $\C Q_1$-module with dimension vector $(a,b,c)$.
 \item The Kronecker quiver and $\PP_1$ are derived equivalent using the tilting bundle
 $\cO\oplus \cO(1)$. In the other direction we can do the moduli construction
 $\PP_1 \cong \cM_{(1,1)}^{(1,-1)}(Q)$.
 \item 
Let $\cE$ be an elliptic curve and identify $\cE=\Pic_0(\cE)$.
We can construct a bundle $\cPP$ on $\cE\times \cE$ such that $\cPP|_{\cE\times \ell}$
is the line bundle corresponding to $\ell \in \Pic_0(\cE)$. The Fourier-Mukai transform
with kernel $\cPP$ gives an nontrivial auto-equivalence of $\Db\Coh \cE$, see \cite{hille2007fourier}.
\end{enumerate}
\end{example}

An important application of Fourier-Mukai transform to the theory of 
crepant resolutions is the following result by Bridgeland.
\begin{theorem}(Bridgeland \cite{bridgeland2002flops})\label{flopsderived}
All crepant resolutions of a 3-dimensional variety are derived equivalent. 
\end{theorem}
The idea behind the proof is that all crepant resolutions of a 3-dimensional variety are related by sequences of flops and a flop induces a Fourier-Mukai transform 
that is an equivalence if the dimension is $3$.

\subsection{Noncommutative crepant resolutions}

Given an algebra $A$, dimension vector $\alpha$ and stability condition $\theta$ we get a natural proper map
\[
\tilde X := \CM_\alpha^\theta(A) \onto \CM_\alpha^0(A) =: X.
\]
Under certain conditions this map can give us a crepant resolution of $\tilde X \to X$. If $A$ is derived equivalent to $\tilde X$ then it makes sense to view $A$ as a noncommutative crepant resolution of $X$.

The natural way for this to happen is that $A$ is the endomorphism ring of the tautological bundle $\cU$ on $\CM_\alpha^\theta(A)$. If we push this bundle forward along $\pi:\tilde X \to X$
we get an $R=\C[X]$-module $U$ with $A\cong\End_R(U)$. This module is reflexive ($\Hom_R(\Hom_R(U,R),R) \cong U$) because $\cU$ is a vector bundle.  
Furthermore, because $A$ and $\tilde X$ are derived equivalent and $\tilde X$ is smooth, 
the global dimension of $A$ is finite. 

Building on this ideas Van den Bergh developed the notion of a noncommutative crepant resolution without refering to a commutative crepant resolution \cite{van2004non}.
\begin{definition}
Let $R$ be a Gorenstein domain with Krull dimension $k$ and $X=\Spec R$. A noncommutative crepant resolution (NCCR) of $X$ is an algebra $A$ such that
\begin{enumerate}
 \item [N1] $A=\End(U)$, with $U$ a reflexive $R$-module, 
 \item [N2] All simple $A$-modules have projective dimension $k$.
\end{enumerate}
\end{definition}

If $A=\End_R(U)$ and $U$ decomposes as a direct sum $U=\bigoplus_i U_i$ of indecomposable reflexive modules we can write
$A$ as the path algebra of a quiver with relations. The vertices of this quiver correspond to the summands $U_i$ and there is a natural dimension
vector $\alpha$ that maps vertex $v_i$ to $\rank U_i := \dim_\KK U_i \otimes_R \KK$ where $\KK$ is the function field of $R$.

If we put $n = |\alpha|$ then $A\otimes_R \KK = \Mat_{n \times n}(\KK)$ and hence the generic representation in $\Rep_{\alpha} A$
is a simple. Therefore $\Rep_{\alpha} A$ has a special component that is the closure of the simple representations.
For this component we can construct the moduli space $\CM_\theta^\alpha(A)$. If $\theta$ is generic then $\CM_\theta^\alpha(A)$ is fine
with tautological bundle $\cU$. 

The main result is the following.
\begin{theorem}[Van den Bergh \cite{van2004non}]
Let $X$ be a 3-dimensional terminal singularity then $X$ admits a commutative crepant resolution if and only if it admits a noncommutative crepant resolution. Also all crepant resolutions and noncommutative crepant resolutions are derived equivalent. 
\end{theorem}
The proof first constructs a noncommutative crepant resolution by finding a tilting bundle on a commutative crepant resolution. Then it shows that the moduli construction above gives a commutative crepant resolution $\CM_\alpha^\theta(A) \mapsto \CM_\alpha^0(A)$ and
induces a derived equivalence between $A$ and $\CM_\alpha^\theta(A)$. Finally, it uses the fact that all commutative crepant resolutions are related by flops and a result by Bridgeland that a flop induces a Fourier-Mukai transform that is an equivalence.

The technical details of this proof depend on the fact that the fibers of a crepant resolution of a 3-dimensional terminal singularity are at most
$1$-dimensional. These ideas have been extended to other settings (see f.i. Iyama and Wemyss \cite{iyama2014maximal}).
In the remainder of this section we will show how this works for 3-dimensional toric Gorenstein singularities. 

\subsection{The Jacobi algebra}
First we will look at a special noncommutative algebra associated to a dimer model.

\begin{definition}
Let $Q$ be a quiver and $\C Q$ be its path algebra. A \iemph{superpotential} $W$ is a 
sum of cyclic paths, which we view as an element in $\C Q/[\C Q,\C Q]$.
The \iemph{cyclic derivative} of a cyclic path $p=a_1\dots a_k$ with respect to an arrow $a$ is equal to
\[
\partial_a p = \sum_{i: a_i=a} a_{i+1}\dots a_k a_1\dots a_{i-1}.
\] 
The \iemph{Jacobi algebra of a superpotential} is
\[
\Jac(W) = \frac{\C Q}{\<\partial a W: a \in \qpol_1\>}.
\]
The \iemph{Jacobi algebra $\Jac(\qpol)$ of a dimer model} $\qpol$ is the Jacobi algebra of
\[
W = \sum_{c \in \qpol_2^+} c - \sum_{c \in \qpol_2^-} c. 
\]
It can also be defined directly as 
\[
\Jac(\qpol) =\frac{\C \qpol}{\<p_a^+-p_a^-| ap^+_a \in \qpol_2^+ \text{ and }ap^-_a \in \qpol_2^- \>},
\]
which means that for every arrow the path going back along its left cycle is the same in this algebra as the path going back around its right cycle. 
\end{definition}

\begin{example}[Galois covers]\label{mackay1}
Look at the following dimer quiver on the torus 
\begin{center}
\vspace{.2cm}
\begin{tikzpicture}
\draw [-latex,shorten >=3pt] (0,0) -- (0+1,0); 
\draw [-latex,shorten >=3pt] (0+1,0) -- (0+1,0+1); 
\draw [-latex,shorten >=3pt] (0+1,0+1) -- (0,0);
\draw [-latex,shorten >=3pt] (0,0) -- (0,0+1); 
\draw [-latex,shorten >=3pt] (0,0+1) -- (0+1,0+1); 
\draw (0,0) node [circle,draw,fill=white,minimum size=6pt,inner sep=1pt] {}; 
\draw (0,0+1) node [circle,draw,fill=white,minimum size=6pt,inner sep=1pt] {}; 
\draw (0+1,0) node [circle,draw,fill=white,minimum size=6pt,inner sep=1pt] {}; 
\draw (0+1,0+1) node [circle,draw,fill=white,minimum size=6pt,inner sep=1pt] {}; 
\end{tikzpicture}
\vspace{.2cm}
\end{center}
This is a quiver with one vertex and $3$ loops
$X,Y,Z$. The path algebra is the free algebra with $3$ variables and the superpotential is
$W=XYZ-XZY$. The Jacobi algebra is
\[
\Jac(\qpol)= \frac{ \C Q}{\<XY-YX,ZX-XZ,YZ-ZY\>},
\]
which is just the polynomial ring in $3$ variables.

If we choose a larger fundamental domain by quotienting out the universal cover of $\qpol$ by $\Z\vec v+\Z\vec w$ instead of $\Z\times \Z$, we get a quiver $\bar\qpol$ that is $|\vec v \times \vec w|$ as large.
This is a Galois cover of the original dimer with cover group $\Z\times \Z/\Z\vec v+\Z\vec w$.
\begin{center}
\vspace{.2cm}
\begin{tikzpicture}
\foreach \x in {2,...,5}{
\foreach \y in {1,...,3}{
\draw [-latex,shorten >=3pt] (\x,\y) -- (\x+1,\y); 
\draw [-latex,shorten >=3pt] (\x+1,\y) -- (\x+1,\y+1); 
\draw [-latex,shorten >=3pt] (\x+1,\y+1) -- (\x,\y);
\draw [-latex,shorten >=3pt] (\x,\y) -- (\x,\y+1); 
\draw [-latex,shorten >=3pt] (\x,\y+1) -- (\x+1,\y+1); 
\draw (\x,\y) node [circle,draw,fill=white,minimum size=6pt,inner sep=1pt] {}; 
\draw (\x,\y+1) node [circle,draw,fill=white,minimum size=6pt,inner sep=1pt] {}; 
\draw (\x+1,\y) node [circle,draw,fill=white,minimum size=6pt,inner sep=1pt] {}; 
\draw (\x+1,\y+1) node [circle,draw,fill=white,minimum size=6pt,inner sep=1pt] {}; 
}}
\draw[dotted] (2,2)--(3,4);
\draw[dotted] (2,2)--(5,1);
\draw[dotted] (3,4)--(6,3);
\draw[dotted] (6,3)--(5,1);
\draw (4,0.5) node{$G = \frac{\Z\times \Z}{\Z(1,2)+\Z(3,-1)}\cong \Z/7\Z$};
\end{tikzpicture}
\end{center}
\vspace{.2cm}
We can view the new Jacobi algebra as a smash product of the old one.
Remember that if $A$ is an algebra and $K$ a finite group that acts on $A$, the smash product is defined as
\[
 A\star G = A \otimes \C G \text{ with }(a_1\otimes g_1)\cdot (a_2\otimes g_2)=a_1(g_1\cdot a_2)\otimes g_1g_2.
\]
Take $G =\Hom\left(\frac{\Z\times \Z}{\Z\vec v+\Z\vec w},\C^*\right)$ and define
the action of $\rho \in G$ on $\C[X,Y,Z]$ by
\[
 \rho \cdot X=\rho(1,0)X,~\rho \cdot Y=\rho(0,1)Y \text{ and }\rho \cdot Z=\rho(-1,-1)Z.
\]
It is an interesting exercise to check that
\[
 \Jac(\bar \qpol) \cong  \C[X,Y,Z] \star G.
\]
The vertices of $\bar \qpol$ correspond to the idempotents 
$$v_{(i,j)} = \frac{1}{|G|}\sum_{\rho \in G} \rho(i,j)\rho \text{ where }(i,j) \in \frac{\Z\times \Z}{\Z\vec v+\Z\vec w}.$$
In other words the cover group and the group in the smash product are each others dual and the vertices
of the quiver correspond to the characters of the group in the smash product.

This generalizes to arbitrary dimers on the torus. In that case you have to choose a $\Z\times \Z$-grading $\deg$ on $\Jac(\qpol)$ for which
the arrows are homogeneous and the degree of a cyclic path is its homology class. The action of $\rho$ satisfies $\rho\cdot a = \rho(\deg a) a$.
\end{example}

Now let us have a look at the behaviour of the Jacobi algebra under dimer moves.
\begin{theorem}
Let $\qpol$ and $\qpol'$ be zigzag consistent dimer quivers.
\begin{enumerate}
 \item If $\qpol$ and $\qpol'$ are related by a split move then
 their Jacobi algebras are isomorphic.
 \item If $\qpol$ and $\qpol'$ are related by a quiver mutation then
 their Jacobi algebras are derived equivalent.
 \end{enumerate}
\end{theorem}
\begin{proof}
The first statement can be read off directly from the superpotential.
The split move will add one extra cycle $a_1a_2$ in the superpotential
and splits another cycle $b_1\dots b_k$ into $b_1\dots b_ia_1$ and $a_2b_{i+1}\dots b_k$.
If we look at the partial derivatives $\partial_{a_i}$ we get $2$ new relations
\[
 a_1 - b_{i+1}\dots b_k \text{ and } a_2 - b_1\dots b_i.
\]
All other partial derivatives remain the same except that the subpaths
$b_{i+1}\dots b_k$ and $b_1\dots b_i$ are substituted by $a_1$ and $a_2$. 
This clearly gives an isomorphism between $\Jac(\qpol)$ and $\Jac(\qpol')$.

The second part follows from a general theorem by Vit\'oria \cite{vitoria2009mutations} based on work by Derksen, Weyman and Zelevinsky \cite{derksen2008quivers,derksen2010quivers}.
This explains how superpotentials behave under the mutation operation and
shows that quiver mutations between Jacobi algebras with superpotentials
give rise to a derived equivalence if the superpotential satisfies some
additional condition. For dimer quivers this condition is satisfied if
every vertex has valency at least $2$. This is satisfied if $\qpol$ is consistent.

If $v \in \qpol_0$ is the vertex we mutate, the tilting bundle that realizes the derived equivalence is
\[
T_v = \bigoplus_{w \ne v} P_w \oplus \underbrace{\left(\bigoplus_{h(a)=v}P_{t(a)} \stackrel{a\cdot}{\to} P_v \right)}_{\Tilt_v}.
\]
Here $P_w$ is the projective $w\Jac(\qpol)$ and the last term is a complex for which the left term is in degree $0$. 
\end{proof}

The Jacobi algebra of a dimer model has a special central element $\ell$ which is the sum of cycles $c_v \in \qpol_2$, one starting in each vertex $v$  
\[
\ell = \sum_{v \in \qpol_0} c_v
\]
Different $c_v$ for the same vertex are equivalent in the Jacobi algebra, so $\ell$ does not depend on a particular choice of the $c_v$.
Furthermore the relations in $\Jac(\qpol)$ imply that for any path $p$ we have $c_{h(p)}p=pc_{t(p)}$ and therefore $\ell$ is central.  

By inverting $\ell$ we can form the \iemph{weak Jacobi algebra}
\[
\widehat{\Jac}(\qpol) = \Jac(\qpol) \otimes_{\C[\ell]} \C[\ell,\ell^{-1}].
\]
This algebra is the same as the universal localization of $\Jac(\qpol)$ by all arrows.

\begin{definition}
A dimer is called \iemph{cancellative} if the canonical map $\Jac(\qpol)\to \widehat{\Jac}(W) $ is an embedding.
\end{definition}
This condition is introduced in \cite{mozgovoy2010noncommutative} as the \emph{first 
consistency condition}
\begin{theorem}
If $\qpol$ is a zigzag consistent dimer then it is cancellative. If $\chi(\qpol)\le 0$ then the reverse also holds: all cancellative dimers are zigzag consistent.
\end{theorem}
\begin{proof}
We give a brief sketch of the argument. Details can be found in \cite{bocklandt2012consistency} or in a paper by Ishii and Ueda \cite{ishii2009dimer}.
If $\pZ_1$ and $\pZ_2$ are two zigzag paths that intersect twice, we can find a simply connected piece in the universal cover bounded by these zigzag paths.
Look at the opposite paths of the zigzag paths outside this piece.
These opposite paths are equivalent up to a power of $\ell$ in the weak Jacobi algebra, because they are homotopic.

In the ordinary Jacobi algebra they are not equivalent. 
The opposite paths consist of subpaths of the form $a_1\dots a_{k-2}$ where $a_1\dots a_{k} \in \qpol_2$.
Hence they cannot contain any $p_a^+$ or $p_a^-$ and we cannot make the piece they cut out smaller using the relations $\partial_aW$. 
Note that this argument does not work on the sphere because we could go the other way round over the sphere.

If a dimer is not cancellative we can find two homotopic paths that cannot be transformed
into each other up to a power of $\ell$. It is possible to show that there are
doubly intersecting zigzag paths inside the connected piece that these homotopic paths cut out.
\end{proof}
\begin{remark}
Dimers on a sphere are never zigzag consistent, but they can be cancellative. An example of this is the octahedron quiver, the fourth quiver in example \ref{lotsofmirrors}. 
\end{remark}

The cancellative property is important because it allows us to transfer nice properties of the weak Jacobi algebra to the ordinary Jacobi algebra.

\begin{definition}[Ginzburg \cite{ginzburg2006calabi}]
An algebra $A$ is Calabi-Yau-$n$ if 
\[
\Ext_{A \otimes A^{op}}^i(A,A\otimes A) = \begin{cases}
0 & i \ne n\\
A & i = n
\end{cases}
\]
This is the same \cite{van2015calabi} as the property that $A$ has a selfdual projective resolution of length $n$ as a bimodule over itself:
\[
A \equiv P^\bullet \text{ and }\Hom_{A\otimes A^{op}}(P^\bullet, A\otimes A^{op})\equiv A[n] 
\]
\end{definition}
\begin{remark}
A Calabi-Yau algebra is a noncommutative version of a Calabi-Yau variety. 
Recall that a smooth compact $n$-dimensional variety is called Calabi-Yau if its canonical sheaf is trivial. Using Poincar\'e duality
this implies that in the derived category of coherent sheaves we have a natural nondegenerate pairing
\[
\Hom(\sF,\sG)\times \Hom(\sG,\sF[n]) \to \C.
\] 
When an algebra is Calabi-Yau-$n$ then the bounded derived category of finite dimensional modules has a similar pairing.

From this pairing we can deduce that all simples have projective dimension $n$.
The global dimension of $A$ is $n$ because of the selfdual resolution and if $S$
is a simple module we have that $\Ext^n_A(S,S)=\Hom_A(S,S)^\vee =\C$ so its projective dimension must be $n$.
\end{remark}

There is a close relation between Jacobi algebras and Calabi-Yau-3 algebras. A Jacobi algebra $A$ has a standard selfdual complex
\[
C_W:\hspace{.5cm}\xymatrix{
\underset{i \in \qpol_0}{\bigoplus} F_{ii}\ar[r]^{(\cdot \tau d a \cdot)}&
\underset{a \in \qpol_1}{\bigoplus} F_{t(a)h(a)}\ar[r]^{(\cdot\partial^2_{ba}W \cdot)}&
\underset{b \in \qpol_1}{\bigoplus} F_{h(b)t(b)}\ar[r]^{(\cdot db \cdot)}&
\underset{i \in \qpol_0}{\bigoplus} F_{ii}\ar[r]^m&A}
\]
Here $F_{ij}=Ai\otimes jA$ and $da$ represents $a \otimes t(a) - h(a) \otimes A$, $\tau$ swaps the tensor product
and $\partial^2_{ba}p$ is the sum of all $p_1\otimes p_2$ such that $bp_1ap_2$ is equal to a cyclic shift of $p$.
If this complex is a resolution then the Jacobi Algebra is Calabi-Yau-$3$. In some settings such as the graded \cite{bocklandt2008graded} or complete case \cite{van2015calabi} one can also show that every Calabi-Yau-3 algebra is a Jacobi Algebra.
If we restrict to the toric setting there is also a result that ties Calabi-Yau algebras
to dimers: if $A \subset Mat_n(\C[X^{\pm 1},Y^{\pm 1},Z^{\pm 1}]$ is a $\Z^3$-graded subalgebra and $A$ is CY-3 then $A$ is the Jacobi algebra of a dimer \cite{bocklandt2013calabi}.

\begin{theorem}[Homological properties of Jacobi Algebras]
The weak Jacobi algebra of a dimer $\qpol$ is a Calabi-Yau-3 algebra if and only if $\chi(\qpol)\le 0$. 
If $\qpol$ is cancellative then the ordinary Jacobi algebra is also a Calabi-Yau-3 algebra if $\chi(\qpol)\le 0$.
\end{theorem}
\begin{proof}
The first part follows from the fact that the weak Jacobi algebra is Morita equivalent to the fundamental group algebra of circle bundle over the surface. This
is an aspherical manifold (its universal cover is contractible) and therefore the fundamental group algebra is a 3d-Poincar\'e duality group, which implies that the group algebra is Calabi-Yau-3.

The second part is proved by checking that the complex we defined above is indeed a resolution. This can be done using a grading by the fundamental groupoid of the surface.
The original proof by Mozgovoy and Reineke \cite{mozgovoy2010noncommutative} used
an extra consistency condition on $\qpol$. Later Davison \cite{davison2011consistency} showed that this second condition followed from the cancellative property.
\end{proof}

Apart from nice homological properties, we can also look at finiteness conditions.
\begin{theorem}[Finiteness properties of Jacobi Algebras]
The weak Jacobi algebra of a dimer $\qpol$ is a finite over its center if and only if $\chi(\qpol)\ge 0$.
If $\qpol$ is cancellative then the ordinary Jacobi algebra is also finite over its center if and only if $\chi(\qpol)\ge 0$.
\end{theorem}
\begin{proof}
The first part follows from the fact that if $\chi(\qpol)\ge 0$, the fundamental group of the surface is abelian. If we fix a weak path $p_{ij}$ between
each pair of vertices $i,j$ then these generate $\widehat\Jac(\qpol)$ as a module over its center. If $\chi(\qpol)< 0$ then the fundamental group is not finite over its center and neither is the weak Jacobi algebra.

If $\qpol$ is cancellative then $Z(\Jac(\qpol)) = Z\widehat{\Jac}(\qpol) \cap \Jac(\qpol)$.
Two weak paths that are homotopic differ by a power of $\ell$ so for
each homotopy class there is a path $p \in \Jac(\qpol)$ such that
all homotopic paths are in $\C[\ell]p \subset Z(\Jac(\qpol)p$.
On a sphere there are only a finite number of homotopy classes so
$\Jac(\qpol)$ is finitely generated over $\C[\ell]$.
For the torus the second part follows from the explicit description of $\Jac(\qpol)$ as an endomorphism ring of a sheaf theorem \ref{JacNCCR}.
\end{proof}

The behaviour of the Jacobi algebra depends a lot on the Euler characteristic of the surface and we have a trichotomy as illustrated in the table below.
\begin{center}
\begin{tabular}{|c|c|c|}
\hline
$\chi(\qpol)>0$&$\chi(\qpol)=0$&$\chi(\qpol)<0$ \\
\hline
finite over center&finite over center&never finite over center\\
$\Uparrow$&$\Uparrow$&\\
cancellative&cancellative&cancellative\\
&$\Updownarrow$&$\Updownarrow$\\
never zigzag consistent&zigzag consistent&zigzag consistent\\
&$\Downarrow$&$\Downarrow$\\
never CY-3&CY-3&CY-3\\
\hline
\end{tabular}
\end{center}

The key property that makes the Jacobi algebra nice is the cancellative property. For nonnegative Euler characteristic it gives the Jacobi algebra nice finiteness properties, 
while for the nonpositive Euler characteristic it induces nice homological properties. Clearly, the best of both worlds happens if the Euler characteristic is zero and
then the notion also coincides with zigzag consistency. We will study this case in more detail.

\subsection{Jacobi algebras and NCCRs}

From now on $\qpol$ will be a zigzag consistent dimer on a torus.
In this case, we know there is a close connection between the zigzag paths and perfect matchings.

A perfect matching $\PM$ defines a grading $\deg_\PM$ on $\Jac(\qpol)$ by giving the arrows in the perfect matching degree $1$ and the others
degree zero. For this grading $\ell$ has degree $1$. This grading can be used to characterize the weak Jacobi algebra.

\begin{theorem}\label{weakjac}
If $\qpol$ is a cancellative dimer on a torus with $|\qpol_0| =n$ then 
\[
\widehat{\Jac}(\qpol) \cong \Mat_n(\C[\Z^3]) 
\]
\end{theorem}
\begin{proof}
Number the vertices $1,\dots,n$ and fix paths $p_{1i}$ from each vertex to one.
By multiplying these paths with the appropriate powers of $\ell$ we can assume that $\deg_\PM p_{1i}=0$.

To each path $q$ we can associate a triple $(x,y,z) \in\Z^3$ that represents the homology class and the $\PM$-degree of $p_{1h(q)}qp_{1t(q)}^{-1}$ and
let $m_q \in \C[\Z^3]$ the monomial that represents that triple.
The map $q \mapsto m_q E_{h(q)t(q)} \in \Mat_n(\C[\Z^3])$, where $E_{ij}$ is the elementary matrix with $1$ on entry $i,j$,  gives the required isomorphism.
\end{proof}

So, paths are uniquely determined by their homotopy and  $\PM$-degree. Because the construction did not depend on which perfect matching we used,
two homotopic paths that have the same degree for one perfect matching will also agree for any other perfect matching.

To describe Jacobi algebra we will need two important lemmas.
\begin{lemma}\label{minimalpaths}
If $\qpol$ is a zigzag consistent dimer on a torus then for each pair of vertices $v,w$ and each homotopy class there is a real path $p:v\ot w$ in $\qpol$
with this homotopy class and a perfect corner matching $\PM$ such that $\deg_{\PM} p=0$.
\end{lemma}
\begin{proof}
The proof is based on an idea by Broomhead \cite{broomhead2012dimer} in the geometrically consistent setting and uses the corner matchings an extension of the argument to all consistent dimers is found in \cite{bocklandt2012consistency}. 

Let $v,w$ be the head and tail vertices of the path in the universal cover. Draw zigzag paths emanating from $v$ and look at the pair of consecutive zigzag paths in the clockwise direction
that encompass $v$. Draw shifts of these zigzag paths such that you get a rhombus that contains $w$. 
\[
\xymatrix@=.3cm{
&&&&&&&&\\
&&\vtx{v}\ar@{.>}[rrrr]|{\pZ_{i+1}}&&&\vtx{v}\ar@/^3pc/[llllldd]\ar@/_3pc/[llllldd]\ar[llld]&&&\\
&&\vtx{w}\ar[lld]&&&&&&\\
\vtx{v}\ar@{.>}[uuurrr]|{\pZ_{i}}\ar@{.>}[rrrr]|{\pZ_{i+1}}&&&\vtx{v}\ar@{.>}[uuurrr]|{\pZ_{i}}&&&&&
}
\]
\vspace{.5cm}

Outside the rhombus
there are two paths going in the opposite direction along the boundary. Both paths are zero on the corner matching between
the two zigzag paths, so these paths are equivalent. We can apply relations to turn the two paths into each other. These relations will move the path from the upper left to the lower right. One of the intermediate paths will meet $w$ and this will give us an equivalent path that meets $v$.
\end{proof}

\begin{lemma}\label{realpath}
If $\qpol$ is a cancellative dimer on a torus and $p \in \widehat{\Jac}(\qpol)$ is a weak path then 
$p \in \Jac(\qpol)$ if and only if it has nonnegative $\PM$-degree for all perfect corner matchings.
\end{lemma}
\begin{proof}
Let $p$ be such a weak path then we can construct a real path $q$ with the same homology class and a corner matching $\PM$ such that
$\deg_\PM q =0$ therefore $p=q\ell^{\deg_\PM p} \in \Jac(\qpol)$. 
\end{proof}

Using these lemmas we can prove the main theorem
\begin{theorem}[$\Jac(\qpol)$ is an NCCR]\label{JacNCCR}
If $\qpol$ is a cancellative dimer on a torus then its center is the coordinate ring of the 3d-Gorenstein singularity defined by its matching polygon
\[
 Z(\Jac(\qpol)) = R_{\MP(\qpol)}.
\]
Moreover $\Jac(\qpol)$ is a noncommutative crepant resolution of $\Spec R_{\MP(\qpol)}$.
\end{theorem}
\newcommand{\vo}{1}
\begin{proof}
Fix a vertex $\vo$ and let $X,Y$ be weak cycles in $\vo$ that form a basis for the homology of 
the torus and let $Z$ be a positive cycle. We already know from $\ref{weakjac}$ that $\widehat{\Jac}(\qpol)=\Mat_n(\C[X^{\pm1} ,Y^{\pm1},Z^{\pm1}])$. 

We use this identification and the previous lemma to determine the pieces $v\Jac(\qpol)w$:
\se{
 v\Jac(\qpol)w &= \C\< p \in v\widehat \Jac(\qpol)w \,|\,\forall \PM: \deg_{\PM}p\ge 0 \>\\
 &\cong \C\< p_{\vo v}pp_{\vo w}^{-1} \,|\, \forall \PM:\deg_{\PM}p\ge 0\>\\
 &\cong \C\<X^iY^jZ^k\,|\,\forall \PM: i\PM(X)+j\PM(Y)+k +\deg_{PM} p_{\vo v}p_{\vo w}^{-1} \ge 0\>.
}
Here $\forall \PM$ stands for all corner matchings.

The ring $v\Jac(\qpol)v$ is equal to $$\C[X^iY^jZ^k\,|\,\forall \PM: (\PM(X),\PM(Y),1) \cdot (i,j,k) \ge 0],$$  which is precisely the coordinate ring of the 3d-Gorenstein singularity defined by the matching polygon.

A central element of  $\Jac(\qpol)$ is also central in $\widehat{\Jac}(\qpol)$ and hence it is the sum of cyclic paths, one starting in each vertex, which all have the same homology class and $\PM$-degrees, 
so $R=Z(\Jac(\qpol))= v\Jac(\qpol)v$.

On the other hand $v\Jac(\qpol)w$ is the rank 1 reflexive module $\grr_{a^{vw}}$ where $a_{\PM}^{vw} = \deg_{\PM} p_{1v}- \deg_{\PM} p_{1w}$.
From this it easily follows that 
\se{
\End_R\bigoplus_{w \in \qpol_0} \grr_{a^{\vo w}}
&= \bigoplus_{v,w \in \qpol_0} \grr_{a^{\vo w}}^{-1}\otimes \grr_{a^{\vo v}}\\
&= \bigoplus_{v,w \in \qpol_0} \grr_{a^{\vo v} - a^{\vo w}}\\
&= \bigoplus_{v,w \in \qpol_0} v\Jac(\qpol)w\\
&= \Jac(\qpol).
}

Further we already know that $\Jac(\qpol)$ is CY-3 so all simples have projective
dimension $3$ and $\Jac(\qpol)$ is an NCCR. 
\end{proof}
Various versions of this theorem can be found in the literature \cite{broomhead2012dimer,ishii2009dimer,mozgovoy2009crepant,bocklandt2012consistency}.
The main ingredient is always the proof that the map
\[
 \Jac(\qpol) \to  \End_R\bigoplus_{w \in \qpol_0} \grr_{a^{\vo w}}
\]
is an isomorphism. This condition is also called \emph{algebraic consistency}
(see \cite{broomhead2012dimer}). 

\subsection{Moduli spaces for the Jacobi algebra}

Let $\qpol$ be a zigzag consistent dimer quiver on a torus. The Jacobi algebra is an NCCR and the reflexive module $U=\bigoplus_{w \in \qpol_0} 1\Jac(\qpol)w$ is a direct sum of rank $1$ reflexive modules.
If we want to use geometric representation theory to construct crepant resolutions then we have to use the dimension vector $\alpha=\id$, which maps every vertex to $1$.
We will denote the Jacobi algebra of $\qpol$ by $A=\Jac(\qpol)$, its center by $R$, and the weak versions by $\widehat{A} $ and $\widehat{R}$. 

\newcommand{\Ct}{\mathsf{Cons}}
Let us first have a look at the representation theory of $\widehat A$.
The space $\Rep_\alpha \widehat A$ can be seen as 
\se{
\Rep_\alpha \widehat A &= \{\rho \in \Maps(\qpol_1,\C^*)| \rho(r^+_a)=\rho(r^-_a)\}\\
&=\{\rho \in \Maps(\qpol_1,\C^*)| \rho(a)\rho(r^+_a)=\rho(a)\rho(r^-_a)\}\\
&=\{\rho \in \Maps(\qpol_1,\C^*)| \forall c_1,c_2 \in \qpol_2:\rho(c_1)=\rho(c_2)\}\\
&= d^{-1}\Ct
}
where $\Ct$ is the space of constant maps in $\Maps(\qpol_2,\C^*)$ and
$d$ is the differential in the complex $\Maps(\qpol_\bullet,\C^*)$.
On the other hand $\GL_\alpha = \Maps(\qpol_0,\C^*)$
and it acts on $\Rep_\alpha \widehat A$ by multiplication with $d\GL_\alpha$.
For $\widehat A$ all $\alpha$-dimensional representations are simple so
\[
 \cM^\theta_\alpha(\widehat A) = \frac{\Rep_\alpha \widehat A}{\GL_\alpha} = \frac{d^{-1}\Ct}{d\Maps(\qpol_0,\C^*)}.
\]
This is a quotient of two complex tori and hence also a complex torus.
We will denote these $3$ tori by $T_G,T_A$ and $T_R$.
\[
\xymatrix@R=.1cm{
T_G & T_A & T_R\\
\GL_\alpha \ar@{^(->}[r]& \Rep_\alpha \widehat A\ar@{->>}[r] & \cM^0_\alpha(\widehat A)
}
\]
The dimensions of these tori are
\se{
\dim T_G &= \#\qpol_0,\\
\dim T_A &= \#\qpol_1-\#\qpol_2 +1 + \dim \Ct= \#\qpol_0+2,\\
\dim T_R &= \dim T_A - \dim T_G +1 = 3.
}
(The extra $1$'s come form the fact that the map $d:\Maps(\qpol_1,\C^*)\to \Maps(\qpol_2,\C^*)$ is not surjective and the map $d:\Maps(\qpol_0,\C^*)\to \Maps(\qpol_1,\C^*)$
is not injective.)
To each of these tori we can associate a pair of dual lattices $M_X =\Hom(T_X,\C^*)$
and $N_X = \Hom(\C^*,T_X)$.

On the $M$-side we get that
\[
\xymatrix@R=.1cm{
M_G & M_A& M_R\\
\Z \qpol_0 & \frac{\Z \qpol_1}{\<\bar c_1-\bar c_2 | c_1,c_2 \in \qpol_2\>}\ar@{->>}[l]& \{m \in M_A| dm=0\}\ar@{_(->}[l]
}
\]
where $d$ is the differential in $\Z\qpol_\bullet$. 

Each weak path $p=a_1\dots a_k$  gives an element
$\bar p = a_1 +\dots + a_k \in M_A$. Two paths $p,q$ that are equivalent in $\widehat A$
give the same element in $M_A$. The reverse also holds provided $h(p)=h(q)$ and $t(p)=t(q)$. 

A path represents an element in $M_R$ if it is cyclic.
We will fix $3$ real minimal cyclic paths $X,Y,Z$ in $\vo\widehat A\vo$. $X,Y$ will represent the $2$ homology classes on the torus and for $Z$ we take a cycle in $\qpol_2$. This will identify $M_R$ with $\Z^3$:
\[
 (i,j,k) \mapsto i\bar X+j\bar Y+k\bar Z \in M_R.
\]

On the $N$-side we get
\[
\xymatrix@R=.1cm{
N_G & N_A& N_R\\
\Maps(\qpol_0,\Z)\ar@{^(->}[r]&\{n \in \Maps(\qpol_1,\Z)|n(\bar  c_1)=n(\bar c_2)\} 
\ar@{->>}[r] & \{dn| n \in N_A\}.
}
\]
We can also see $N_A$ as the set of all $\Z$-gradings on 
$A$ such that the arrows are homogeneous. The map $N_A\to N_R$ restricts
the grading to the cyclic paths. Again we identify $N_R$ with $\Z^3$.
\[
n \mapsto (n(\bar X),n(\bar Y),n(\bar Z)).
\]
In particular if $\cP$ is a perfect matching we can associate to it
the degree function $\deg_\cP$ that gives the arrows in $\cP$ degree $1$ and
the other arrows degree $0$. We denote the corresponding point in $N_A$ by
$n_\cP$ and its image in $N_R$ by $\bar n_\cP$.

Now we will look at the representations of the Jacobi algebra $A$. 
The space 
\[
\Rep_\alpha A = \{\rho \in \Maps(\qpol_1,\C)| \rho(r^+_a)=\rho(r^-_a)\}.
\]
is in general not an irreducible variety but it contains $\Rep_\alpha \widehat A$ 
as an open subset. We will denote its closure by $\trep_\alpha A$. 
\[
\C[\trep_\alpha A] = \C[ M_A^+] \text{ with }M_A^+=\sum_{a \in \qpol_1} \N a \subset M_A.
\]
\newcommand{\nrep}{\overline{\trep}}
In general $\trep_\alpha A$ is not normal and we will denote its normalization by $\nrep_\alpha A$.

The variety $\nrep_\alpha A$ is an affine toric variety and its cone will consist
of all $n \in N_A\oplus \R$ that give the arrows nonnegative degrees
\[
 \sigma_A = \{ n \in N_A\otimes \R|\forall a \in \qpol_1: n_a\ge 0\}. 
\]
\begin{lemma}
$\sigma_A$ is generated by the perfect matchings. 
\[
 \sigma_A = \sum_{\cP\in \PMs(\qpol)} \R_+ n_{\cP} \subset N_A\otimes \R.
\]
\end{lemma}
\begin{proof}
If $n \in \sigma_A$ then $n/n(Z)$ is a unit flow and
the statement follows from lemma \ref{flowpm}.
\end{proof}
Likewise the cone of the ring 
\se{
  \C[\nrep_\alpha A]^{\GL_\alpha} &= \C[m \in M_R\,|\,\forall \cP \in \PMs(\qpol): \bar n_\cP(m)\ge 0]\\
  &=\C[(i,j,k)\,|\, \forall \cP \in \PMs(\qpol): i\deg_\cP X+j\deg_{\cP}Y+k\ge 0]\\
}
Because of lemma \ref{minimalpaths}, each $X^iY^jZ^k \in \C[\nrep_\alpha A]^{\GL_\alpha}$ can be represented by a path and therefore this ring is also equal to
$\C[\nrep_\alpha A]^{\GL_\alpha}= \C[M_A^+\cap M_R]$.

By the same lemma we can restrict to corner matchings, so 
$\C[\trep_\alpha A]^{\GL_\alpha}\cong R_{\MP(\qpol)}$ and $\trep_\alpha A/\!\!/\GL_\alpha$ is the Gorentstein singularity associated to $\MP(\qpol)$. 
We will denote its cone by
\[
 \sigma_R = \sum_{p\in \PMs(\qpol)} \R_+ \bar n_{\cP} \subset N_R\otimes \R.
\]

Now we will look at the moduli spaces 
\[
 \CM_\alpha^\theta(A)= \Proj \C[\trep_\alpha A]_\theta \text{ and }
 \overline{\CM}_\alpha^\theta A = \Proj \C[\nrep_\alpha A]_\theta.
 \]
\begin{theorem}[The crepant resolution (Ishii-Ueda \cite{ishii2007moduli,ishii2009dimer}]
Let $\qpol$ be a consistent dimer on a torus, $A=\Jac(\qpol)$ and $R=Z(A)$.
If $\theta:\qpol_0 \to \Z$ is a generic stability condition then $\CM_\alpha^\theta \to \CM_\alpha^0=\Spec R$ is a crepant resolution.  
\end{theorem}
\begin{proof}
The idea of the proof is to look at $T_R$-invariant points of  
$\CM_\alpha^\theta A$. Such a point $\varrho$ can be represented by a representation
$\varrho$ that maps every arrow to either zero or one. In \cite{ishii2007moduli}
Ishii and Ueda construct an embedding $\C^3 \to \trep_\alpha A$ that
maps $0$ to $\varrho$ and projects onto an open subset of $\CM_\alpha^\theta A$.

If we look at all $T_R$-invariant points we get a cover of $\CM_\alpha^\theta A$.
The transition functions preserve the standard volume form on $\C^3$, so
$\CM_\alpha^\theta A$ is Calabi-Yau and the resolution is crepant.
\end{proof}

Because $\CM_\alpha^\theta(A)$ is smooth, it is also normal and equal to $\overline{\CM}_\alpha^\theta(A)$.
To construct the fan of the moduli space $\CM_\alpha^\theta(A)=\Proj \C[\Rep_\alpha A]_\theta$ we need to classify the torus orbits of $\theta$-semistable representations.

Again we can use perfect matchings to describe these.
If $\{\PM_1,\dots,\PM_k\}$ is a collection of perfect matchings we can look at the representation 
$$
\rho_{\PM_1,\dots,\PM_k}: a \mapsto \begin{cases}
                                     0 &a \in \PM_1\cup \dots\cup \PM_k\\
                                     1 &a \notin \PM_1\cup \dots\cup \PM_k
                                    \end{cases}
$$
We call the collection \iemph{$\theta$-stable} if $\rho_{\PM_1,\dots,\PM_k}$ is a stable representation.

\begin{theorem}[The Fan of $\CM_\alpha^\theta(A)$ (Mozgovoy \cite{mozgovoy2009crepant})]
Let $\qpol$ be a consistent dimer on a torus, $A=\Jac(\qpol)$ and $R=Z(A)$.
If $\theta:\qpol_0 \to \Z$ is a generic stability condition then the fan of $\CM_\alpha^\theta$ consists of all cones $\R_+ \bar n_{\PM_1} + \dots+\R_+ \bar n_{\PM_k}$ where
$\{\PM_1,\dots,\PM_k\}$ is a stable collection. 
There are 3 types of nontrivial closed stable collections:
\begin{enumerate}
\item Singletons $\{\PM\}$. On each lattice point $\bar n$ of the matching polygon there is a unique stable perfect matching $\PM$ with $\bar n_\PM=\bar n$
The torus orbit of $\rho_{\PM}$ in $\CM_\alpha^\theta$ is $2$-dimensional.
\item
Pairs $\{\PM_1,\PM_2\}$. The points $\bar n_{\PM_1}$,$\bar n_{\PM_2}$ form an elementary line segment and the orbit of $\rho_{\PM_1,\PM_2}$ in $\CM_\alpha^\theta$ is $1$-dimensional.
\item
Triples $\{\PM_1,\PM_2,\PM_3\}$. The points $\bar n_{\PM_1}$,$\bar n_{\PM_2},\bar n_{\PM_3}$ form an elementary triangle and the orbit of $\rho_{\PM_1,\PM_2,\PM_3}$ in 
$\CM_\alpha^\theta$ is a point.
\end{enumerate}
The empty collection corresponds to the 3-dimensional torus containing the $\widehat A$-representations.
\end{theorem}
\begin{proof}
We already know that $\CM_\alpha^\theta$ is a crepant resolution.
So the fan splits the matching polygon in elementary triangles.

On each of the lattice points there can only be one stable perfect matching
because otherwise $\rho_{\PM_1}$ and $\rho_{\PM_2}$ would correspond to
the same point in $\CM_\alpha^\theta$. If $\theta$ is generic there is a unique stable
orbit for each point so these two representations are isomorphic.
\end{proof}

\begin{theorem}[The tilting bundle]
Let $\qpol$ be a zigzag consistent quiver, $A=\Jac(\qpol)$ and $R=Z(A)$.
If $\theta:\qpol_0 \to \Z$ is a generic stability condition then $\CM_\alpha^\theta$ is a fine moduli space with tautological bundle $\cU = \bigoplus_{v \in \qpol_0} \tilde \grr_v$.
Here $\tilde\grr_v$ is the rank one reflexive sheaf
$$
\tilde \grr_v = \tilde\grr_{(\deg_{\PM_i} p_{\vo v})}.
$$
and $p_{1v}$ is any path from vertex $v$ to vertex $\vo$. Note that now 
$\PM_i$ represents all stable matchings, not all corner matchings like in \ref{JacNCCR}.

The tautological bundle $\cU$ is a tilting bundle on $\CM_\alpha^\theta$ and hence we have an equivalence
\[
 \Db\Mod A \cong \Db\Coh \CM_\alpha^\theta(A).
\]
\end{theorem}
\begin{proof}
We need to show the following two things: the sheaves 
\[
 \tilde\grr_v\otimes \tilde\grr_{w}^\vee = \tilde\grr_{(\deg_{\PM_i} p_{1w}^{-1}p_{1v})}
\]
have no higher cohomology and they generate the derived category $\Db\Coh \CM_\alpha^\theta$.
This can be done using techniques from Van den Berg \cite{van2004non} or
by an in depth analysis of the combinatorics of the dimer \cite{ishii2009dimer}.
\end{proof}

\begin{intermezzo}[The McKay Correspondence]\label{mckay2}
Consider $\C^3$ with the action of a finite group $G\subset \mathsf{GL}_3(\C)$.
The quotient $\C^3/G= \Spec \C[X,Y,Z]^G$ is a Gorenstein singularity if $G\subset \mathsf{SL}_3(\C)$.
It has a crepant resolution called $\GHilb$ which parametrizes the ideals in $\C[X,Y,Z]$
with a quotient that is isomorphic to the regular representation of $G$.
\[
\GHilb =\{ \ideal m \lhd \C[X,Y,Z] \,|\, \C[X,Y,Z]/\ideal m \cong \C G \}
\]
The space $\GHilb$ is a fine moduli space: we have a bundle $\cU\to \GHilb$ such that the fiber
over $\ideal m$ is $\C[X,Y,Z]/\ideal m$.

A famous result by Bridgeland, King and Reid \cite{bridgeland2001mckay} states that 
the derived category of $\GHilb$ is equivalent to the derived category of $G$-equivariant sheaves on $\C^3$.
\[
\Db \Coh \GHilb \cong \Db \Coh_G \C^3  
\]
As $\C^3=\Spec \C[X,Y,Z]$ a $G$-equivariant sheaf is just a $\C[X,Y,Z]$-module $M$ with a $G$-action on it such that 
$$
\forall f \in \C[X,Y,Z]:\forall m \in M: g\cdot fm = (g\cdot f)(g\cdot m).
$$
This is the same as a module of $A=\C[X,Y,Z]\star G$, so
\[
\Db \Coh \GHilb \cong \Db \Mod A  
\]
and $A$, which is also isomorphic to $\End(\cU)$, is an NCCR of $\C^3/G$.

If $G\subset \mathsf{SL}_3(\C)$ is abelian we can see $A=\C[X,Y,Z]\star G$ as the Jacobi algebra of a dimer.
Without loss of generality we can suppose that $G$ acts diagonally.
Let $N=\{(i,j)| X^iY^j \in \C[X,Y,Z]^G\}$ then 
$$
G \mapsto \Hom(\Z^2/N,\C^*) : g \mapsto (i,j) \mapsto \frac{g(X^iY^j)}{X^iY^j}  
$$
is an isomorphism and $A\cong \Jac(\bar\qpol)$ where $\bar \qpol$ is the Galois cover of $\qpol$ with cover group $\Z^2/N$.

We can also see $\GHilb$ as a moduli space $\CM_\alpha^\theta(A)$ for a special $\theta$.
Let $v$ be the vertex corresponding to $\frac{1}{|G|} \sum_{g \in G} g$ and choose $\theta_v=-|G|+1$ and $\theta_w=1$ for $w\ne v$. 
An $A$-module $S$ is $\theta$-stable if and only if $Sv$ generates $S$.
If $S=\C[X,Y,Z]/\ideal m$ is the regular $G$-representation then viewed as an $A$-module $S$ is a module
with dimension vector $(1,\dots,1)$ and it is $\theta$-stable because $S$ is cyclic as a $\C[X,Y,Z]$-module.
\end{intermezzo}

We end this section with a classification of all resolutions that can occur in this way. First of all such a resolution $\tilde X \to X$ must be toric.
Which means that both $\tilde X$ and $X$ are toric varieties and the map $\pi:\tilde X \to X$ is equivariant for the torus-action.
A second condition is that $\tilde X$ must be projective, i.e. $\tilde X=\proj \tilde R$ with $\tilde R$ a positively graded ring such
that $\Spec\tilde R_0 =X$. Not every subdivision of the matching polygon in elementary triangles is projective. In order to be projective the variety must allow an ample line bundle. A counterexample is given by the following polygon \cite{oda1988convex}.
\begin{center}
\begin{tikzpicture}[scale=.5]
\draw (0,0) node{$\bullet$}--(0,2) node{$\bullet$}--(1,0) node{$\bullet$}--(0,1) node{$\bullet$} --(0,2) node{$\bullet$}--(2,0) node{$\bullet$}--(1,1) node{$\bullet$}--(1,0);
\draw (0,2) node{$\bullet$}--(-2,-2) node{$\bullet$};
\draw (0,1) node{$\bullet$}--(-1,0) node{$\bullet$};
\draw (0,1)--(-1,-1) node{$\bullet$};
\draw (0,1)--(-2,-2) node{$\bullet$};
\draw (2,0) node{$\bullet$}--(-2,-2) node{$\bullet$};
\draw (2,0) node{$\bullet$}--(-2,-2) node{$\bullet$};
\draw (-1,-1) node{$\bullet$}--(0,-1) node{$\bullet$};
\draw (-1,-1) node{$\bullet$}--(2,0) node{$\bullet$};
\draw (-1,-1) node{$\bullet$}--(1,0) node{$\bullet$};
\draw (0,0) -- (2,0);
\draw (0,0) -- (-2,-2);
\end{tikzpicture}
\end{center}

\begin{theorem}[All projective toric resolutions occur (Craw-Ishii \cite{craw2004flops})]
Let $\qpol$ be a zigzag consistent quiver, $A=\Jac(\qpol)$ and $R=Z(A)$.
If $\cM \to \Spec R$ is any projective toric resolution then there is a character $\theta$ such that
\[
 \cM \to \Spec R  \cong \CM_\alpha^\theta \to \Spec R
\]
\end{theorem}
\begin{proof}
The proof by Craw and Ishii in (Craw-Ishii \cite{craw2004flops}) is given for dimers coming from the McKay-correspondence, i.e. those for which the matching polygon is a triangle. It works by analysing what happens if we change the
stability condition. The space of stability conditions 
\[
 \Theta = \{\theta \in \Z \qpol_0 \,|\, \theta\cdot \alpha=0\}
\]
can be divided into chambers, which are separated by walls of the form
$\beta\cdot \theta=0$ where $\beta$ is a dimension vector with $\forall v \in \qpol_0:\beta_v\le \alpha_v$.
On the positive side of the wall representations with a subrepresentation of
dimension vector $\beta$ are stable while on the negative side they are unstable.
On the negative side on the wall representations with a subrepresentation of
dimension vector $1-\beta$ are stable and on the positive side they are unstable.

If such representations exist wall crossing will cut out certain representations
and glue in others, which can result in a flop. 
By studying the combinatorics of this problem Ishii and Craw were able to 
construct all projective toric crepant resolutions. 
This technique also works for general dimers. 

Another way to prove this result is to use theorem \ref{alltropicaloccur} and the connection between tropical curves and resolutions given by theorem \ref{tropmod}.
\end{proof}

\begin{example}
We illustrate this with the suspended pinchpoint.
The quiver has 3 vertices, so $\Theta=\Z^2$. We parametrize it by $(\theta_2,\theta_3)$.
There are three walls corresponding to the equations $\theta_i=0$. 

There are $2$ perfect matchings located on $(0,1)$: 
\begin{tikzpicture}
\begin{scope}[xshift=-1cm,yshift=-.5cm]
\begin{scope}[scale=.125]
\draw [dotted] (0,0) -- (3,0) -- (3,3) -- (0,3);
\draw [dotted] (0,0) -- (3,1) -- (3,2) -- (0,1)--(0,2)--(3,3);
\draw [thick] (0,0) -- (3,1);
\draw [thick] (0,2) -- (0,3);
\draw [thick] (3,2) -- (3,3);
\end{scope}
\end{scope}
\end{tikzpicture}
and
\begin{tikzpicture}
\begin{scope}[xshift=-1cm,yshift=-.5cm]
\begin{scope}[scale=.125]
\draw [dotted] (0,0) -- (3,0) -- (3,3) -- (0,3);
\draw [dotted] (0,0) -- (3,1) -- (3,2) -- (0,1)--(0,2)--(3,3);
\draw [thick] (0,2) -- (3,3);
\draw [thick] (0,0) -- (0,1);
\draw [thick] (3,0) -- (3,1);
\end{scope}
\end{scope}
\end{tikzpicture} and for every $\theta$ only one of them is stable.
If we cross the wall $\theta_1=0$ they change between stable and unstable.
In that case a $2$-dimensional subvariety of $\CM_\alpha^\theta$ gets replaced
by another. The moduli space stays the same variety but it
classifies different representations. Also the tautological remains
the same vector bundle but has a different $\Jac(\qpol)$-action. 

Flops occur at the walls $\theta_2=0$ and $\theta_3=0$ because there a $1$-dimensional
orbit of representations becomes unstable and gets substituted.
\begin{center}
\begin{tikzpicture}
\begin{scope}[scale=1]
\begin{scope}[yshift=-1.5cm]
\draw [-latex,shorten >=5pt] (0,0) -- (3,0);
\draw [-latex,shorten >=5pt] (3,0) -- (3,1);
\draw [-latex,shorten >=5pt] (3,1) -- (0,0);
\draw [-latex,shorten >=5pt] (0,0) -- (0,1);
\draw [-latex,shorten >=5pt] (0,1) -- (3,2);
\draw [-latex,shorten >=5pt] (3,2) -- (3,3);
\draw [-latex,shorten >=5pt] (3,2) -- (3,1);
\draw [-latex,shorten >=5pt] (0,2) -- (0,1);
\draw [-latex,shorten >=5pt] (0,3) -- (3,3);
\draw [-latex,shorten >=5pt] (3,3) -- (0,2);
\draw [-latex,shorten >=5pt] (0,2) -- (0,3);
\draw (0,0) node[circle,draw,fill=white,minimum size=10pt,inner sep=1pt] {{\tiny1}};
\draw (0,1) node[circle,draw,fill=white,minimum size=10pt,inner sep=1pt] {{\tiny2}};
\draw (0,2) node[circle,draw,fill=white,minimum size=10pt,inner sep=1pt] {{\tiny3}};
\draw (0,3) node[circle,draw,fill=white,minimum size=10pt,inner sep=1pt] {{\tiny1}};
\draw (3,0) node[circle,draw,fill=white,minimum size=10pt,inner sep=1pt] {{\tiny1}};
\draw (3,1) node[circle,draw,fill=white,minimum size=10pt,inner sep=1pt] {{\tiny2}};
\draw (3,2) node[circle,draw,fill=white,minimum size=10pt,inner sep=1pt] {{\tiny3}};
\draw (3,3) node[circle,draw,fill=white,minimum size=10pt,inner sep=1pt] {{\tiny1}};
\end{scope}
\end{scope}
\begin{scope}[xshift=7cm]
\begin{scope}
\draw[-latex] (-2,0)--(2,0) node[right] {$\theta_3=0$};
\draw[-latex] (0,-2)--(0,2) node[above] {$\theta_2=0$};
\draw[-latex] (-1.5,1.5)--(1.5,-1.5) node[right] {$\theta_1=0$};

\begin{scope}[xshift=1cm,yshift=1cm]
\begin{scope}[scale=.25]
\begin{scope}[xshift=-.5cm,yshift=-.5cm]
\draw[thick](0,0) -- (0,2) -- (1,0) -- (1,-1) -- (0,0); 
\draw (0,0) node{$\scriptstyle{\bullet}$};
\draw (0,2) node{$\scriptstyle{\bullet}$};
\draw (0,1) node{$\scriptstyle{\bullet}$};
\draw (1,0) node{$\scriptstyle{\bullet}$};
\draw (1,-1) node{$\scriptstyle{\bullet}$};

\draw (1,0)--(0,1);
\draw (1,-1)--(0,1);
\end{scope}
\end{scope}
\end{scope}

\begin{scope}[xshift=-1cm,yshift=-1cm]
\begin{scope}[scale=.25]
\begin{scope}[xshift=-.5cm,yshift=-.5cm]
\draw[thick](0,0) -- (0,2) -- (1,0) -- (1,-1) -- (0,0); 
\draw (0,0) node{$\scriptstyle{\bullet}$};
\draw (0,2) node{$\scriptstyle{\bullet}$};
\draw (0,1) node{$\scriptstyle{\bullet}$};
\draw (1,0) node{$\scriptstyle{\bullet}$};
\draw (1,-1) node{$\scriptstyle{\bullet}$};

\draw (1,0)--(0,1);
\draw (1,-1)--(0,1);
\end{scope}
\end{scope}
\end{scope}

\begin{scope}[xshift=-1.5cm,yshift=.5cm]
\begin{scope}[scale=.25]
\begin{scope}[xshift=-.5cm,yshift=-.5cm]
\draw[thick](0,0) -- (0,2) -- (1,0) -- (1,-1) -- (0,0); 
\draw (0,0) node{$\scriptstyle{\bullet}$};
\draw (0,2) node{$\scriptstyle{\bullet}$};
\draw (0,1) node{$\scriptstyle{\bullet}$};
\draw (1,0) node{$\scriptstyle{\bullet}$};
\draw (1,-1) node{$\scriptstyle{\bullet}$};

\draw (1,-1)--(0,2);
\draw (1,-1)--(0,1);
\end{scope}
\end{scope}
\end{scope}

\begin{scope}[xshift=-.5cm,yshift=1.5cm]
\begin{scope}[scale=.25]
\begin{scope}[xshift=-.5cm,yshift=-.5cm]
\draw[thick](0,0) -- (0,2) -- (1,0) -- (1,-1) -- (0,0); 
\draw (0,0) node{$\scriptstyle{\bullet}$};
\draw (0,2) node{$\scriptstyle{\bullet}$};
\draw (0,1) node{$\scriptstyle{\bullet}$};
\draw (1,0) node{$\scriptstyle{\bullet}$};
\draw (1,-1) node{$\scriptstyle{\bullet}$};

\draw (1,-1)--(0,2);
\draw (1,-1)--(0,1);
\end{scope}
\end{scope}
\end{scope}

\begin{scope}[xshift=.5cm,yshift=-1.5cm]
\begin{scope}[scale=.25]
\begin{scope}[xshift=-.5cm,yshift=-.5cm]
\draw[thick](0,0) -- (0,2) -- (1,0) -- (1,-1) -- (0,0); 
\draw (0,0) node{$\scriptstyle{\bullet}$};
\draw (0,2) node{$\scriptstyle{\bullet}$};
\draw (0,1) node{$\scriptstyle{\bullet}$};
\draw (1,0) node{$\scriptstyle{\bullet}$};
\draw (1,-1) node{$\scriptstyle{\bullet}$};

\draw (1,0)--(0,1);
\draw (1,0)--(0,0);
\end{scope}
\end{scope}
\end{scope}

\begin{scope}[xshift=1.5cm,yshift=-.5cm]
\begin{scope}[scale=.25]
\begin{scope}[xshift=-.5cm,yshift=-.5cm]
\draw[thick](0,0) -- (0,2) -- (1,0) -- (1,-1) -- (0,0); 
\draw (0,0) node{$\scriptstyle{\bullet}$};
\draw (0,2) node{$\scriptstyle{\bullet}$};
\draw (0,1) node{$\scriptstyle{\bullet}$};
\draw (1,0) node{$\scriptstyle{\bullet}$};
\draw (1,-1) node{$\scriptstyle{\bullet}$};

\draw (1,0)--(0,1);
\draw (1,0)--(0,0);
\end{scope}
\end{scope}
\end{scope}

\end{scope}
\end{scope}
\end{tikzpicture}
\end{center}

\end{example}

\newcommand{\rCU}{\mathcal V_\R}
\subsection{The space of semi-invariants}\label{spaceofsemi}

In this section we will extend the results in the previous section to nongeneric $\theta$. Instead of working with the complex toric variety, we will work with its quotient by the group action of $U_1^3\subset \C^{*3}$. This is a topological space
\[
(\CM_\alpha^\theta)_\R = \overline{\CM}_\alpha^\theta/U_1^3 
\]
with 3 real dimensions and it has the structure of a manifold with corners. 
If we can write a toric variety as the proj of a ring $\C[\sigma^\vee \cap \Z^n]$ then this manifold is isomorphic to the
slice of $\sigma^\vee$ that corresponds to the degree $1$ elements. In particular in our case,
\[
 \overline{\CM}_\alpha^\theta = \Proj \C[\nrep_\alpha A]_\theta,  
\]
this is the slice of the $\theta$-semi-invariant monomials on $\nrep_\alpha A$.
This isomorphism is not canonical but depends on a choice of moment map of the $U_1^3$ action. More details can be found in \cite[Chapter 4]{fulton1993introduction}.

To construct this slice we identify a monomial function $f= \prod X_a^{m_a}$ with $(m_a) \in \Maps(\qpol_1,\R)$. 
If we look at the standard complex
\[
 \Maps(\qpol_2,\R) \stackrel{d}{\to} \Maps(\qpol_1,\R) \stackrel{d}{\to} \Maps(\qpol_0,\R),
\]
the functions on $\rep_\alpha \widehat A$ are integral points of the quotient of $\Maps(\qpol_1,\R)$ by the image of 
\[
\rG_0 = \{ (c_f)_{f \in \qpol_2}| \sum_{f \in \qpol_2^+} c_f - \sum_{f \in \qpol_2^-} c_f=0\} \subset \rG := \Maps(\qpol_2,\R)
\]
under $d$. The subspace of $\theta$-semi-invariants is the fiber $d^{-1}(\theta)/\rG_0$.

In this subspace the functions over $\nrep_\alpha A$ are those that have positive $\PM$-degree for all perfect matchings. These form a convex subset $\rCU \subset d^{-1}(\theta)/\rG_0$.

To describe this subset we can fix one semi-invariant $m=(m_a) \in d^{-1}(\theta)/\rG_0$. All other semi-invariants are of the form
$(m_a) + (x_a)$ with $(x_a)$ an invariant. The space of invariants is spanned by 3 invariants $\bar X,\bar Y,\bar Z$. We can identify $d^{-1}(\theta)/\rG_0$ with $\R^3$
by  $(\xi,\eta,\zeta)\mapsto (m_a)+ \xi X+ \eta Y-\zeta Z$.

Because the degree of $Z$ is $1$ for all perfect matchings we have
\[
\rCU = \{(\xi,\eta,\zeta)\,|\, \zeta \le \xi \deg_{\PM}(X) +\eta \deg_{\PM}(Y)+\deg_{\PM} m\} 
\]
So $\rCU$ is the space above the graph of the function
$F_m =\max_{\PM} (\deg_{\PM}(X)+ \eta \deg_{\PM}(Y)+\deg_{\PM} m)$. This expression is the same as the tropical function that we used 
to define the tropical curve associated to the dimer model in the part on statistical physics. 
To interprete the tropical curve in this setting, we first consider the quotient $d^{-1}(\theta)/\rG$ instead of $d^{-1}(\theta)/\rG_0$. This space is two-dimensional
and it corresponds to the projection $\rC$ along the $\Z$-direction. The boundary of $\rCU$ maps bijectively onto $d^{-1}(\theta)/\rG$ and
the tropical curve splits $d^{-1}(\theta)/\rG$ into pieces that corresponds to the facets of $\rCU$.

\begin{center}
\begin{tikzpicture}[line join=round]
\begin{scope}[scale=.5]
\draw[-latex](-6.769,-5.059)--(-9.477,-5.938);
\draw[-latex](-6.769,-5.059)--(-6.769,4);
\draw[-latex](-6.769,-5.059)--(1.354,-6.232);
\draw(-5.957,-3.27)--(-3.52,-.762);
\draw(-5.957,2.449) --(-3.52,-.762);
\draw(-1.895,-.997)--(-3.52,-.762);
\draw(.542,1.51) node[above] {$\tC$}--(-1.895,-.997);
\draw(.542,-4.208)--(-1.895,-.997);
\fill [fill=lightgray,fill opacity=0.6](-4.061,-.938)--(-8.123,-3.974)--(-1.625,-4.912)--(-2.437,-1.173)--cycle;
\fill[fill=lightgray,fill opacity=0.6](-4.061,-.938)--(-8.123,1.745)--(-8.123,-3.974)--cycle;
\fill[fill=lightgray,fill opacity=0.6](-2.437,-1.173)--(-1.625,-4.912)--(-1.625,-4.912+1.745+3.974)--cycle;
\fill[fill=lightgray,fill opacity=0.6](-8.123,1.745)--(-4.061,-.938)--(-2.437,-1.173)--(-1.625,-4.912+1.745+3.974)--cycle;
\draw (-1.625,-4.912+1.745+3.974)--(-2.437,-1.173)--(-4.061,-.938)--(-8.123,1.745);
\draw (-1.625,-4.912)--(-2.437,-1.173)--(-4.061,-.938)--(-8.123,-3.974);

\draw (-8.123,1.745) node[above] {$\rCU$};
\end{scope}
\end{tikzpicture}
\end{center}

\begin{theorem}[the tropical curve and the moduli space]\label{tropmod}
Let $m=(m_a)$ be a weighting on the arrows and construct the tropical curve with function
$F_m =\max_{\PM} (\deg_{\PM}(X)+ \eta \deg_{\PM}(Y)+\deg_{\PM} m)$.
The subdivision of the Newton polygon that corresponds to the tropical curve is the same as the subdivision of the matching polygon that corresponds to 
the fan of $\overline{\cM}_\alpha^\theta$ with $\theta_v = \sum_{h(a)=v} m_a - \sum_{t(a)=v} m_a$.
\end{theorem}
\begin{proof}
Each $2$-dimensional face of $\rCU$ corresponds to a $2$-dimensional torus orbit or
a lattice point in the matching polygon. Two faces share an edge if the corresponding
lattice points are connected by an edge.

The picture for the tropical curve is the same, the connected components of the complement
correspond to lattice points in the Newton polygon and these lattice points
are connected if the components share an edge of the tropical curve.
\end{proof}

We are now able to interprete the tropicalization of the A-model $\tL(\qpol) \to \tX(\qpol)$ in terms of the $B$-model.
$\tL(\qpol)$ can be seen as the space of all monomial functions with real exponents on $\nrep_\alpha \qpol$ modulo the powers of $\ell$. The space $\tX(\qpol)$ can be seen as the character space $\Theta\otimes \R$ and the projection map $\tL(\qpol) \to \tX(\qpol)$ associates to each monomial the character for which it is a semi-invariant.

The fibers of the projection map are two-dimensional and they each contain a tropical curve that cuts the fiber in $2,1,0$-dimensional cells
which parametrize the $2,1,0$-dimensional torus ortbits in $\overline{\cM}_\alpha^\theta$.

The final step in the construction is to relate these bundles for different dimers  
using quiver mutation.

\begin{theorem}
Let $\qpol$ and $\mu\qpol$ be two consistent dimers on a torus, related by mutation. 
The map $\phi^\trop_\mu$ restricts to a bijection 
\[
 \phi^\trop_\mu: \Theta(\qpol) \to \Theta(\mu\qpol),
\]
such that there is an isomorphism of toric varieties 
\[
\overline{\CM}_\alpha^\theta(\qpol) \cong \overline{\CM}_\alpha^{\phi^\trop_\mu(\theta)}(\mu \qpol).
\]
\end{theorem}
\begin{proof}
The map $\phi_\mu^\trop$ maps $\Theta(\qpol)$ to $\Theta(\mu\qpol)$ because
its expression only uses the tropical operations which map integers to integers. It is a bijection because its inverse is
the same map for the inverse mutation.

There is an isomorphism 
\[
\overline{\CM}_\alpha^\theta(\qpol) \cong \overline{\CM}_\alpha^{\phi^\trop_\mu(\theta)}(\mu \qpol).
\]
because the tropical curves are the same up to translation.
\end{proof}

\begin{example}
If we go back to example \ref{P1P1} we see that the tropical 
transition function becomes
\se{
 \theta_1'&=\theta_1\\ \theta_2'&=\theta_2+2\max(0,\theta_4)\\ \theta_3'&=\theta_3+2\theta_4-2\max(0,\theta_4)\\ \theta_4'&=-\theta_4.
}
For the first quiver we will look at the plane of stability conditions
\[
 \Pi = \{ \alpha(1,0,-1,0)+\beta(0,-1,0,1)| \alpha,\beta \in \R\}.
\]
and we draw an overview of the moduli spaces coming from stability conditions in this plane
\begin{center}

\end{center}
If we perform the tropical transition function the upper and the lower half planes of $\Pi$
will end up in different planes of stability conditions of the mutated quiver:
\se{
 \Pi_1 &= \{\alpha(1,0,-1,0)+\beta(0,1,0,-1)| \alpha,\beta \in \R\},\\
 \Pi_2 &= \{\gamma(1,0,-1,0)+\delta(0,-1,2,-1)| \alpha,\beta \in \R\}.\\
}
If we draw these planes we see that they overlap with the upper and lower halves of the first plane.
\begin{center}
%
\end{center}

\weg{
Let us fix $\theta_1=1$ and look at the slices of $\Theta(\qpol)$ and $\Theta(\mu\qpol)$.
We can parametrize these slices by $(\theta_2,\theta_3)$ and $(\theta'_2,\theta'_3)$

\begin{center}
%
 
\end{center}
}
\end{example}

\section{Connections: cluster and categories}

\renewcommand{\Stab}{\mathsf{Stab}\,}
\newcommand{\Cat}{\mathtt{C}}
We have seen two stories that come from completely different parts of mathematics
but which look very much alike.
On the one hand we studied the dynamics of a cluster variety and interpreted 
its positive part as the parameter space of a model in statistical physics.
On the other hand we studied resolutions of 3-dimensional singularities and
how one could realize these as moduli spaces of representations of a noncommutative algebra. This also gave rise to a parameter space which turned out
to be the same as the tropical limit of the parameter space of the statistical model.
The reason why these two stories line up so neatly is homological mirror symmetry.

\subsection{Categories and mirror symmetry}
Mirror symmetry is a phenomenon that first occured in superstring theory.
This is quantum theory that studies the dynamics of strings in a 4+6-dimensional space.
The first $4$ dimensions represent ordinary space-time and the other $6$
are internal degrees of freedom that have the mathematical shape of a Calabi-Yau manifold. This Calabi-Yau manifold has both a symplectic and a complex structure. 

Studying the full physics of superstrings is beyond our mathematical technology, but one can make approximations that turn the theory into something more manageable. 
There are $2$ main approaches to this called the A-model and the B-model.
For an introduction on how physicists arrive at these models we refer to \cite{aspinwall2003d}.

The A-model focusses on the symplectic side of the Calabi-Yau manifold: the strings in the theory satisfy boundary conditions that require their end points to be fixed
on Lagrangian submanifolds. These are submanifolds of a symplectic manifold
that look locally like space embedded in phase space. These Lagrangian submanifolds
are dynamical objects in the theory and therefore it
is interesting to study their intersection theory.

The B-model focusses on the complex side of the Calabi-Yau manifold: the strings in the
theory satisfy boundary conditions that require their end points to be fixed
on complex submanifolds equiped with vector bundles. Mathematically it means
that we are interested in coherent sheaves of the Calabi-Yau manifold and 
their cohomology.

Homological mirror symmetry is a strange duality between the two models. 
It states that the A-model on one Calabi-Yau $X$ is physically indistinguishable
from the B-model on another Calabi-Yau $X^\vee$. These two Calabi-Yaus are called
mirror pairs. 

To turn this observation into a mathematical statement, Kontsevich \cite{kontsevich1994homological} formulated the
homological mirror symmetry conjecture. This conjecture associates to both
models a category. For the $A$-model this is the Fukaya category of $X$, which has
as objects the Lagrangian submanifolds and as morphisms their intersection points.
For the $B$-model this category is the category of coherent sheaves.
Two manifolds $X,X^\vee$ form a mirror pair if these categories are derived equivalent: 
\[
 \Db\Fuk X \cong \Db \Coh X^\vee.
\]

\subsection{Choosing a mirror pair}
The two stories we presented above are related to one particular sort of conjectured
mirror pairs. For the $A$-model the Calabi-Yau manifold is a smooth hypersurface of
the form
\[
X_f = \{(x,y,u,v) \in \C^{*2} \times \C^2 \,|\, uv+f(x,y)=0 \}
\]
where $f$ is a generic polynomial with Newton polygon $\LP$.
This is a compact submanifold but we can give it a symplectic structure coming
from the standard symplectic structure on $\PP_1^2\times \PP_2$. The symplectic structure does not change under small deformations of the coefficients
of $f$. The intersection theory of certain Lagrangians on $X_f$ can be studied by looking at the intersection theory of curves in the Riemann surface $\cC_f:f(X,Y)=0$.

For the $B$-model the manifold $X^\vee_f$ is a crepant resolution of the Gorenstein
singularity defined by the same polygon $\LP$.
The homological mirror symmetry conjecture in this particular case states that
for a generic polynomial $f \in \C[X^{\pm 1},Y^{\pm 1}]$ we have that
\[
 \Db\Fuk X_f \cong \Db\Coh X_f^\vee.
\]
From this statement it is clear that the mirror of a manifold is not uniquely determined,
there are parameters we can adjust without changing the underlying category.
On the $A$-side these are the coeffients of the polynomial and on the $B$-side
this is a stability condition to construct the moduli space. 

If we fix a dimer we have parameter space $\cX(\rib)$ 
and for each point $\xi \in \cX(\rib)$ we can construct a symplectic manifold $X_{P_\xi}$
and a complex manifold $\CM_\alpha^\theta(A)$ where $\theta= \log |\xi|$. 
Hence $\cX(\rib)$ can be seen as a moduli space of mirror pairs. 

A second conjecture related to mirror symmetry is the Strominger-Yau-Zaslow (SYZ) conjecture \cite{strominger1996mirror}. It states that if $(X,X^\vee)$ is a mirror pair then $X$ and $X^\vee$ should be dual torus fibrations over a common base: 
we have a diagram 
\[
 \xymatrix{X\ar@{->>}[dr]_{\pi_X}&&X^\vee\ar@{->>}[dl]_{\pi_X^\vee}\\
 &B&}
\]
such that the generic fibers $\pi_X^{-1}(b)$ and $\pi_{X^\vee}^{-1}(b)$
are both real tori which are naturally dual to each other: 
$\pi_X^{-1}(b)=M_\R/M$ and $\pi_{X^\vee}^{-1}(b)=N_\R/N$ where $M$ and $N$
are dual lattices.

The common base $B$ is usually singular.  
There are two ways to construct it, one can look at an action of $U_1^3$
and take the quotient or one can construct a tropical degeneration. For more information
on this we refer to \cite{gross2011tropical}.
This is also reflected in our stories. On both sides we ended up with a union of
line segments, which could be interpreted as the tropical limit
of the spectral curve $\cC_f$ on the $A$-side and as the $U_1^3$-quotient
of the singular locus of $\ell^{-1}(0) \subset \CM_\alpha^\theta(A)$.
This one-dimensional version of the SYZ-fibration represents the locus
where the geometry of the whole SYZ-fibration is the most interesting.

There is also a second choice underneath, the dimer, which determines the parametrization we use to describe the parameter space $\cX$. It also picks out some 
special objects in the category. On the $A$-side the vertices of $\qpol$ correspond
to zigzag paths in $\polq$. These can be drawn as strands on the spectral curve.
Each of the strands can be lifted to a Lagrangian submanifold, which gives
us an object in $\Db\Fuk X_f$. On the $B$-side the vertices of $\qpol$ correspond
to the indecomposable summands of the tilting bundle $\cU$, which are objects
in $\Db\Coh X_f^\vee$.

\subsection{Clusters and stability}

This type of description with a discrete and a continuous part is a phenomenon that pops up in a lot of situations and its combinatorics has been formalized in the theory of cluster algebras and cluster Poisson varieties \cite{fomin2002cluster,fomin2003cluster,fock2003cluster,fock2009cluster}.
The discrete part is a cluster, which is encoded by a quiver and a set of cluster variables. The continuous part is (the spectrum of) the cluster algebra or
the cluster Poisson variety. Mutation changes the cluster and 
tells you how the corresponding change in coordinates works on the continuous part.

Starting from a dimer model we have interpreted this structure in two different ways
associated to two interpretations of the same category $\Db\Fuk X_f \cong \Db\Coh X_f^\vee$. This suggests that there is a way to define this structure on the level of the category. We will end with a tentative description on how this might work in a
more general setting.

Given a triangulated category $\Cat$ one can define its space of Bridgeland 
stability conditions $\Stab \Cat$ \cite{bridgeland2007stability}. This consists of pairs $(Z,\heartsuit)$ where $Z:\mathtt{K}_0 \to \C$ is an additive map from the Grothendieck group of $\Cat$ to $\C$ and $\heartsuit\subset \Cat$ is the heart of a $t$-structure.
These two objects satisfy some axioms. One of these is that the image $\Z(\heartsuit)$ must be in the upper half plane $\mathbb{H}^+=\{z| \arg(z)\in (0,\pi] \}$.

The main result of \cite{bridgeland2007stability} is that $\Stab \Cat$ can be given the structure of a manifold such that 
\[
 \psi: \Stab \Cat \to \Hom(\mathtt{K}_0,\C): (Z,\heartsuit)\mapsto Z
\]
is a local homeomorphism. If $(Z,\heartsuit)$ is a stability condition and we change $Z$ a little bit to $Z'$, two things can happen: the image of $\heartsuit$ stays in
$\mathbb{H}^+$ or some of the objects in $\heartsuit$ get negative imaginary part. If the latter is the case we can construct a new heart $\heartsuit'$ that substitutes
these bad objects for new ones such that $(Z',\heartsuit')$ is again a stability condition. This process is analogous to mutation.

The space of stability conditions has a cell decomposition according to the hearts. So again we have a continuous part and a discrete part but now
the cells are glued together without overlap. This is different from the Poisson cluster variety, where different coordinate patches are identified
on an open dense subset. The situation is similar to what happens with toric varieties. On the one hand you have the fan, where the top-dimensional
cones are glued together at their boundaries and on the other hand you have the affine pieces $U_\sigma$ which are glued together on affine open subsets.
This suggests that it is possible to construct a `fan/toric variety' type duality between the theory of Bridgeland stability conditions and cluster Poisson varieties and let
$\Stab^\vee \Cat$ be the dual to $\Stab \Cat$.

Given a derived category that is CY-3 there is an antisymmetric form on the Grothendieck
group $\<A,B\>=\sum_i (-1)^i\dim \Hom(A,B[i])$, which can be used to define the Poisson structure on $\Stab^\vee \cC$. The symplectic $A$-manifold should be extracted from  $\Stab^\vee \cC$ by looking at level sets of Hamiltonians for its dynamics. The complex $B$-manifold should be constructed by a moduli problem. We then have two processes that can be used to degenerate these manifolds, tropicalization and quotienting out a torus action. If we apply the first method to the $A$-manifold and the second to the $B$-manifold we should get the same result.

\bibliography{overview}{}
\bibliographystyle{plain}

\printindex
\end{document}